\documentclass[a4paper,11pt,]{article}
\usepackage{graphicx}

\usepackage{tabularx}
 \usepackage{amssymb}
 \usepackage{amsmath}
\usepackage{multirow}
\usepackage{supertabular}
\usepackage{lscape}

\newtheorem{THM}{Theorem}[section]
\newtheorem{LMA}[THM]{Lemma}

\global \count10 = 0

\global \count11 = 1

\global \count12 = 1

\global \count13 = 1

\def\squarebox#1{\hbox to #1{\hfill\vbox to #1{\vfill}}}
\def\qed{\hspace*{\fill}%
        \vbox{\hrule\hbox{\vrule\squarebox{.667em}\vrule}\hrule}\smallskip}
\newenvironment{proof}{\begin{trivlist}
\item[\hspace{\labelsep}{\em\noindent Proof.~}]}{\qed\end{trivlist}}

\def\squarebox#1{\hbox to #1{\hfill\vbox to #1{\vfill}}}
\def\qed{\hspace*{\fill}%
        \vbox{\hrule\hbox{\vrule\squarebox{.667em}\vrule}\hrule}\smallskip}

\begin{document}

\title{
Connected domination in grid graphs
}

\author{
	Masahisa Goto,
	Koji M. Kobayashi 
	\\  
	{\footnotesize 
		The University of Tokyo, 
		3-8-1 Komaba, Meguro-ku, Tokyo 113-8902 JAPAN
	}
}

\date{}

\maketitle

\pagestyle{plain}
\thispagestyle{plain}

\begin{abstract}
	\ifnum \count10 > 0
	\fi
	\ifnum \count11 > 0
	%
	%
	Given an undirected simple graph, 
	a subset of the vertices of the graph is a {\em dominating set} 
	if every vertex not in the subset is adjacent to at least one vertex in the subset. 
	A subset of the vertices of the graph is a {\em connected dominating set} 
	if the subset is a dominating set and the subgraph induced by the subset is connected. 
	In this paper, 
	we determine the minimum cardinality of a connected dominating set, 
	called the {\em connected domination number}, of an $m \times n$ grid graph for any $m$ and $n$. 
	\fi
\end{abstract}


\section{Introduction} \label{sec:intro}
\ifnum \count10 > 0
\fi
\ifnum \count11 > 0
%
%
Given an undirected simple graph, 
a subset of the vertices of the graph is a {\em dominating set} 
if every vertex not in the subset is adjacent to at least one vertex in the subset. 
By imposing restrictions on dominating sets, 
various variants are proposed. 
One of the main variants is a {\em connected dominating set} ({\em CDS}). 
A subset of the vertices of the graph is a {\em connected dominating set} 
if the subset is a dominating set and the subgraph induced by the subset is connected. 
Given an undirected simple graph, 
the problem of computing a minimum CDS of the graph is called the {\em connected dominating set} (CDS) problem, 
which was shown to be NP-hard~\cite{GJ1979}. 
All the vertices composing a CDS are connected and the other vertices are adjacent to vertices in the CDS. 
Computing a CDS is equivalent to computing some spanning tree 
by regarding vertices in the CDS (not in the CDS, respectively) as non-leaves (leaves, respectively). 
Namely, 
the problem of computing a spanning tree with the maximum number of leaves is called the {\em maximum laves spanning tree} ({\em MLST}) problem, 
which is equivalent to the CDS problem \cite{HL1984}. 
The CDS and MLST problems have been extensively studied for various classes of graphs. 
A few results on CDSs in grid graphs with restricted size have been shown in \cite{TF2003,LG2010,SN2021},
but a minimum CDS ({\em MCDS}) and the cardinality of an MCDS (called, the {\em connected domination number}) of any grid graph remain open and furthermore, no conjecture about the numbers of grid graphs has been posed.
\fi
\ifnum \count10 > 0
\fi
\ifnum \count11 > 0
%
%
\noindent
\textbf{Previous Results and Our Results.}~~~
%
%
Fujie~\cite{TF2003} proposed an integer linear programming approach to obtain a spanning tree with the maximum number of leaves for any graph. 
For $m \times n$ grid graphs such that $m, n \leq 9$, 
he showed computationally the maximum number of leaves (i.e., the connected domination number) using this approach. 
Let $\gamma_{m,n}$ denote the connected domination number for an $m \times n$ grid graph. 
Whether the following results are on the CDS problem or on the MLST problem, 
we describe them using the connected domination number. 
Fujie~\cite{TF2003} showed that for any $m,n$, 
$\gamma_{m,n} \geq mn - \lfloor \frac{2mn}{3} \rfloor$. 
Also, 
for any $m,n \geq 4$, 
he proved 
$\gamma_{m,n} \leq \min \{ 2m + n-4 + \lfloor \frac{n-4}{3} \rfloor (m-2), 2n + m-4 + \lfloor \frac{m-4}{3} \rfloor (n-2) \}$. 
Lie and Toulouse~\cite{LG2010} showed the following results: 
\[		
	\gamma_{2, n} 
		=
		\begin{cases}
			2 & n = 2 \\
			2 & n = 3 \\
			n & n \geq 4
		\end{cases}. 
\]
For any $n \geq 3$, 
$
	\gamma_{3, n} = n
$. 
For any $n \geq 4$, 
$
	\gamma_{4, n} = 2n - \lfloor \frac{n}{3} \rfloor
$. 
For any $n \geq 5$, 
\[		
	\gamma_{5, n} \leq
		\begin{cases}
			2n   & (n \mod 3) = 0 \\
			2n+1 & \mbox{otherwise}
		\end{cases}. 
\]
For any $n \geq 6$, 
$
	\gamma_{4, n} = 2n + 2
$. 
For any $q \geq 2, n \geq 3q$, 
$	
	\gamma_{3q, n} \leq qn + 2q - 2
$. 
For any $q \geq 2, n \geq 3q+1$, 
$
	\gamma_{3q+1, n} \leq qn + n + 2q - 3 - \lfloor \frac{n}{3} \rfloor
$. 
For any $q \geq 2, n \geq 3q+2$, 
$
	\gamma_{3q+2, n} \leq qn + n + 2q - 3 + i
$, 
in which 
\[
	i = 
		\begin{cases}
			0 & (i \mod 3) = 0 \\
			1 & \mbox{otherwise}
		\end{cases}. 
\]
Srinivasan and Narayanaswamy~\cite{SN2021} showed for any $m, n$, 
\[		
	\gamma_{m, n} \geq \lceil \frac{mn + \frac{ 2 \min \{ m, n \} }{3} }{3} \rceil. 
\]
In this paper, 
we determine the connected domination number of any grid graph. 
Specifically, we show the following theorem: 
\fi
\begin{THM} \label{THM:cdn}
	\ifnum \count10 > 0
	\fi
	\ifnum \count11 > 0
	%
	%
	For any $m \geq 4, n \geq 4$, 
	\[
		\gamma_{m, n} = \frac{mn - {a'}_{m,n} }{3} + \bar{r}'_{m,n} + {c'}_{m,n}, 
	\]
	in which 
	\[
		{a'}_{m,n} = (m \mod 3) \cdot (n \mod 3), 
	\]
	\[
			\bar{r}'_{m,n} = 
				\begin{cases}
					3 & \hspace{10mm} (m \mod 3) \cdot (n \mod 3) = 4 \\
					2 & \hspace{10mm} (m \mod 3) \cdot (n \mod 3) = 2 \\
					1 & \hspace{10mm} (m \mod 3) \cdot (n \mod 3) = 1 \\
					0 & \hspace{10mm} (m \mod 3) \cdot (n \mod 3) = 0
				\end{cases}, 
	\]
	\[
		{c'}_{m,n} = 
			\begin{cases}
				\min\{ \frac{m}{3}, \frac{n}{3} \} & (m \mod 3) = 0   \mbox{ and } (n \mod 3) = 0 \\
				\frac{m}{3} & (m \mod 3) = 0   \mbox{ and } (n \mod 3) \ne 0 \\
				\frac{n}{3} & (m \mod 3) \ne 0 \mbox{ and } (n \mod 3) = 0 \\
				\lfloor \frac{m}{3} \rfloor + \lfloor \frac{n}{3} \rfloor - 1 & (m \mod 3) \ne 0 \mbox{ and } (n \mod 3) \ne 0 \\
			\end{cases}. 
	\]
	\fi
\end{THM}
\ifnum \count10 > 0
\fi
\ifnum \count11 > 0
%
%
\noindent
\textbf{Our Ideas.}~~~
To determine the connected domination number, 
we only have to show the matching upper and lower bounds on it. 
We obtain upper bounds by providing instances of connected dominating sets. 
The typical and advantageous approach to derive lower bounds for this kind of numbers  
is to categorize vertices composing an MCDS into some patterns, 
called blocks or tiles, 
according to their characteristics 
(e.g., the number of their adjacent vertices in the MCDS) 
and show combinations of the patterns and lower bounds on the numbers of such patterns which are necessary to construct the MCDS. 
Using the values of the lower bounds, 
one can obtain lower bounds on the connected domination number. 
Gon{\c{c}}alves et~al.~\cite{GPRT2011} used this kind of approach to their lower bounds 
when they determined minimum dominating sets in grid graphs. 
It is also used in the proofs of lower bounds for CDSs~\cite{SN2021} 
and total dominating sets~\cite{SG2002,NS2010}.  
However, 
it is difficult for one to show that a CDS has the property of being connected
by the numbers of vertices in patterns and combinations of patterns. 
Thus, 
if one shows a lower bound using the above approach, 
it sacrifices the tightness of the lower bound 
and makes it difficult to match an upper bound. 
Then, 
not using such approach but using an algorithmic approach,  
we show the matching upper and lower bounds on the connected domination number 
in Secs.~\ref{sec:up} and \ref{sec:low}, respectively, 
as the proof of Theorem~\ref{THM:cdn}.
For the lower bound proof, 
we introduce a {\em regular} MCDS in Sec.~\ref{sec:regularity}. 
Then, 
a {\em completely} regular MCDS has very characteristic properties and thus, 
we show that it is easy to evaluate a lower bound on the number of vertices composing the MCDS in Sec.~\ref{sec:regularityproperty}.  
Moreover, 
in Sec.~\ref{sec:routinefeasibility}, 
we show that 
any MCDS can be converted algorithmically into a completely regular MCDS according to some routine, 
which is defined in Sec.~\ref{sec:routine}. 
%


%
\fi
\ifnum \count10 > 0
\fi
\ifnum \count11 > 0
%
%
\noindent
\textbf{Related Results.}~~~
The CDS problem is proved to be equivalent to the MLST problem~\cite{HL1984}. 
Garey and Johnson~\cite{GJ1979} showed that the MLST problem on planar graphs with maximum degree four is NP-hard. 
Research on the CDS problem (MLST problem) has been conducted for various classes of graphs in terms of complexity.
The CDS and MLST problems are shown to be NP-hard for several classes: bipartite graphs~\cite{PLH1983}, cubic graphs~\cite{PL1988} and unit disk graphs~\cite{CCJ1990}. 
On the other hand, 
polynomial-time algorithms were designed for other classes: 
cographs~\cite{CP1984}, distance-hereditary graphs~\cite{DM1988} and cocomparability graphs~\cite{KS1993}
(see the detail in \cite{CF2020}). 
For the problems which are NP-hard for some classes of graphs, 
work on approximation algorithms for them has been done~\cite{GK1998,RDJWLK2004,DGPWWZ2008}. 
For the CDS problem, 
Du et~al.\ \cite{DGPWWZ2008} designed an approximation algorithm with a performance ratio $\epsilon (1 + \ln \Delta)$ for any $\epsilon > 1$, 
in which 
$\Delta$ is the maximum degree of a given graph. 
For any $\epsilon' < 1$, 
Guha and Khuller~\cite{GK1998} showed that there does not exist a polynomial-time approximation with a performance ratio $\epsilon' \ln \Delta$ unless $\mbox{NP} \subseteq \mbox{DTIME}(a^{O(\ln \ln a)})$, 
in which $a$ is the number of vertices. 
Note that approximation algorithms for the CDS problem are not used as those for the MLST problem in general. 
For the MLST problem, 
Lu and Ravi~\cite{LR1998} designed a 3-approximation algorithm and 
Solis-Oba et~al.\ \cite{RS1998,SBL2017} improved it by showing a 2-approximation algorithm. 
There are some approximation algorithms for other classes of graphs: 
bipartite graphs~\cite{FA2007}, directed graphs~\cite{DT2009,DV2010,SSW2012} and cubic graphs~~\cite{LZ2002,CFMW2008,BZ2011}. 
Refer to the comprehensive surveys about the CDS problem (e.g., \cite{DW2013,HHH2020}). 
There is much work on important problems in a grid graph. 
Consecutive studies~\cite{JK1983,TYC1992,CC1993,GPRT2011} had been conducted 
on the dominating set problem in grid graphs for long 
to determine the minimum size of a dominating set. 
Also, 
research on variants of the dominating set problem in a grid graph has been extensively done: 
independent dominating sets~\cite{CO2015}, 
total dominating sets~\cite{SG2002,NS2010} 
and 
power dominating sets~\cite{DH2006}. 
Other work on major problems on a grid graph except for variants of the dominating set problem has been conducted: 
the independent set problem~\cite{CW1998}, 
the hamilton paths problem~\cite{IPS1982}
and 
the shortest path problem~\cite{FH1977}. 
\fi
%

\section{Notation and Definitions} \label{sec:notation}
\ifnum \count10 > 0
%

%
\fi
\ifnum \count11 > 0
%
%
For any positive integers $m \geq 4$ and $n \geq 4$, 
$G_{m,n} = (V, E)$ is an $m \times n$ grid graph, 
in which 
\[
	V = \{ (x, y) \mid x \in [1, m], y \in [1, n] \}
\]
and 
\[
	E = 
		\{
			\{ (x, y),(x, y+1) \} \mid x \in [1, m], y \in [1, n-1]
		\}
		\cup
		\{
			\{ (x, y),(x+1, y)  \} \mid x \in [1, m-1], y \in [1, n]
		\}. 
\]
%
%
For any vertices $u, v \in V$ such that $\{ u, v \} \in E$, 
we say that $u$ and $v$ are {\em adjacent}. 
For any two vertices $v_{1}, v_{a} \in V$, 
a {\em path} between $v_{1}$ and $v_{a}$ is a vertex sequence $v_{1} v_{2} \cdots v_{a}$ 
such that for any $i = 1, \ldots, a-1$, $v_{i}$ and $v_{i+1}$ are adjacent. 
For any path $v_{1} v_{2} \cdots v_{a}$ 
such that for any $i, j \ne i$, $v_{i} \ne v_{j}$ holds, 
we say that the path is {\em simple}. 
For a vertex set $S \subseteq V$, 
if for any vertex $v \in V$, 
either $v \in S$ or $v$ is adjacent to a vertex in $S$, 
then $S$ is a {\em dominating set} of $G_{m,n}$. 
If a vertex set $S \subseteq V$ is a dominating set and 
the subgraph induced by $S$ is connected, 
namely, for any two vertices $u, v \in S$, 
there exists a path between $u$ and $v$ of only vertices in $S$, 
then 
$S$ is a {\em connected dominating set} ({\em CDS}) of $G_{m,n}$. 
{\em MCDS} denotes a CDS with minimum cardinality. 
For any dominating set $D$ and any vertex $v \in D$, 
we say that $v$ {\em dominates} $v$ itself and vertices adjacent to $v$. 
For any set $S$, 
$|S|$ denotes the number of elements in $S$. 
\fi
%

%
\section{Upper Bounds} \label{sec:up}

\ifnum \count10 > 0
%
%
%

%
\fi
\ifnum \count11 > 0
%
%

%
\fi
%

%
\begin{LMA}\label{LMA:UB}
	\ifnum \count10 > 0
	%
	%
	\fi
	\ifnum \count11 > 0
	%
	%
	For any $m \geq 4, n \geq 4$, 
	\[		
		\gamma_{m, n} \leq \frac{mn - {a'}_{m,n} }{3} + \bar{r}'_{m,n}+ {c'}_{m,n},  
	\]
	in which 
	\[
		{a'}_{m,n} = (m \mod 3) \cdot (n \mod 3), 
	\]
	\[
		\bar{r}'_{m,n} = 
			\begin{cases}
					3 & \hspace{10mm} (m \mod 3) \cdot (n \mod 3) = 4 \\
					2 & \hspace{10mm} (m \mod 3) \cdot (n \mod 3) = 2 \\
					1 & \hspace{10mm} (m \mod 3) \cdot (n \mod 3) = 1 \\
					0 & \hspace{10mm} (m \mod 3) \cdot (n \mod 3) = 0
			\end{cases}
	\]
	and 
	\[
	{c'}_{m,n} = 
		\begin{cases}
			\min\{ \frac{m}{3}, \frac{n}{3} \} & (m \mod 3) = 0   \mbox{ and } (n \mod 3) = 0 \\
			\frac{m}{3} & (m \mod 3) = 0   \mbox{ and } (n \mod 3) \ne 0 \\
			\frac{n}{3} & (m \mod 3) \ne 0 \mbox{ and } (n \mod 3) = 0 \\
			\lfloor \frac{m}{3} \rfloor + \lfloor \frac{n}{3} \rfloor - 1 & (m \mod 3) \ne 0 \mbox{ and } (n \mod 3) \ne 0 \\
		\end{cases}. 
	\]
	\fi
\end{LMA}
\begin{proof}
	\ifnum \count10 > 0
	\fi
	\ifnum \count11 > 0
	%
	%
	We show an upper bound on the connected domination number by presenting a CDS instance for each $m \geq 4$ and $n \geq 4$. 
	First, 
	we consider the case of any $n$ and $(m \mod 3) = 0$. 
	We consider the vertex set 
	$S_{0,*} = A \cup B \cup C_{0,*}$, 
	in which 
	$A = \{ (1, 2), (2, 2), \cdots, (m, 2) \}$, 
	$B = \{ (2, 3), (2, 4), \cdots, (2, n) \}$, and 
	$C_{0,*} = \{ (x, 3), \cdots, (x, n) \mid x = 5, \ldots, m-1\}$
	(see Fig.~\ref{fig:up}).
	Vertices in $A$ ($B$ and $C_{0,*}$, respectively) dominate vertices in $\{ (1, 1), \cdots, (m, 1) \}$ ($\{ (1, 3), \cdots, (1, n) \}$ and $\{ (x-1, 3), \cdots, (x-1, n), (x+1, 3), \cdots, (x+1, n) \mid x = 5, \ldots, m-1 \}$, respectively). 
	Also, 
	since vertices in $S_{0,*}$ are connected as we can see from Fig.~\ref{fig:up}, 
	$S_{0,*}$ is a CDS. 
	Since $|A| = m$, 
	$|B| = n-2$, and 
	$|C_{0,*}| = (n-2)(\frac{(m-1) - 5}{3} + 1) = \frac{(n-2)(m-3)}{3}$, 
	we have 
	\begin{equation} \label{LMA:UB:eq.1}
		|S_{0,*}| = m + n-2 + \frac{(n-2)(m-3)}{3} = \frac{(n+1)m}{3}. 
	\end{equation}
	Next in the case of any $m$ and $(n \mod 3) = 0$, 
	we consider the vertex set 
	$S_{*,0} = A \cup B \cup C_{*,0}$, 
	in which 
	$C_{*,0} = \{ (3, y), \cdots, (m, y) \mid y = 5, \ldots, n-1 \}$. 
	Vertices in $C_{*,0}$ dominate vertices in 
	$\{ (3, y-1), \cdots, (m, y-1), (3, y+1), \cdots, (m, y+1) \mid y = 5, \ldots, n-1 \}$. 
	Since vertices in $S_{*,0}$ are connected, 
	$S_{*,0}$ is a CDS. 
	$|A| = m$, 
	$|B| = n-2$ and 
	$|C_{*,0}| = (m-2)(\frac{(n-1) - 5}{3} + 1) = \frac{(m-2)(n-3)}{3}$, 
	and we have 
	\begin{equation} \label{LMA:UB:eq.2}
		|S_{*,0}| = m + n-2 + \frac{(m-2)(n-3)}{3} = \frac{(m+1)n}{3}. 
	\end{equation}
	By the above two instances, 
	we can have the following bounds. 
	\noindent
	{\bf\boldmath $(m \mod 3) = 0$ and $(n \mod 3) = 0$:}
	By Eqs.~(\ref{LMA:UB:eq.1}) and (\ref{LMA:UB:eq.2}), 
	\begin{equation} \label{LMA:UB:eq.3}
		\gamma_{m,n} \leq \min \left \{ \frac{(n+1)m}{3}, \frac{(m+1)n}{3} \right \} 
						= \frac{mn}{3} +  \min \left \{ \frac{m}{3}, \frac{n}{3} \right \} 
						= \frac{mn - {a'}_{m,n} }{3} + \bar{r}'_{m,n} + {c'}_{m,n}. 
	\end{equation}
	Note that 
	\[
		{a'}_{m,n} = (m \mod 3) \cdot (n \mod 3) = 0, 
	\]
	\[
		\bar{r}'_{m,n} = 0
	\]
	and 
	\[
		{c'}_{m,n} = \min \left \{ \frac{m}{3}, \frac{n}{3} \right \}. 
	\]
	\noindent
	{\bf\boldmath $(m \mod 3) = 0$ and $(n \mod 3) \ne 0$:}
	If $(m \mod 3) = 0$ and $(n \mod 3) \ne 0$, 
	it follows from (\ref{LMA:UB:eq.1}) that 
	\begin{equation} \label{LMA:UB:eq.4}
		\gamma_{m,n} \leq \frac{(n+1)m}{3} 
						= \frac{mn - {a'}_{m,n} }{3} + \bar{r}'_{m,n} + {c'}_{m,n}. 
	\end{equation}
	Note that 
	\[
		{a'}_{m,n} = (m \mod 3) \cdot (n \mod 3) = 0, 
	\]
	\[
		\bar{r}'_{m,n} = 0
	\]
	and 
	\[
		{c'}_{m,n} = \frac{m}{3}. 
	\]
	\noindent
	{\bf\boldmath $(m \mod 3) \ne 0$ and $(n \mod 3) = 0$:}
	By Eq.(\ref{LMA:UB:eq.2}), 
	\begin{equation} \label{LMA:UB:eq.5}
		\gamma_{m,n} \leq \frac{(m+1)n}{3} 
						= \frac{mn - {a'}_{m,n} }{3} + \bar{r}'_{m,n} + {c'}_{m,n}. 
	\end{equation}
	Note that 
	\[
		{a'}_{m,n} = (m \mod 3) \cdot (n \mod 3) = 0, 
	\]
	\[
		\bar{r}'_{m,n} = 0
	\]
	and 
	\[
		{c'}_{m,n} = \frac{n}{3}. 
	\]
	\noindent
	{\bf\boldmath $(m \mod 3) = 1$ and $(n \mod 3) = 1$:}
	We consider the vertex set 
	$S_{1,1} = A \cup B \cup C_{1,1} \cup D_{1,1} \cup E_{1,1}$, 
	in which 
	$C_{1,1} = \{ (3, y), \cdots, (m, y) \mid y = 5, \ldots, n-2 \}$, 
	$D_{1,1} = \{ (x, n-1), (x, n) \mid x = 5, \ldots, m-2 \}$
	and 
	$E_{1,1} = \{ (m, n-1) \}$. 
	Vertices in $C_{1,1}$ ($D_{1,1}$, respectively)
	dominate 
	vertices in $\{ (3, y-1), \cdots, (m, y-1), (3, y+1), \cdots, (m, y+1) \mid y = 5, \ldots, n-2 \}$
	($\{ (x-1, n-1), (x-1, n), (x+1, n-1), (x+1, n) \mid x = 5, \ldots, m-2 \}$, respectively) 
	and 
	$E_{1,1}$ dominates $(m, n)$. 
	Since $S_{1,1}$ is connected, 
	$S_{1,1}$ is a CDS. 
	$|A| = m$, 
	$|B| = n-2$, 
	$|C_{1,1}| = (m-2)(\frac{(n-2) - 5}{3} + 1)$, 
	$|D_{1,1}| = 2 (\frac{(m-2) - 5}{3} + 1)$ and 
	$|E_{1,1}| = 1$. 
	Thus, 
	we have 
	\begin{eqnarray} \label{LMA:UB:eq.6}
		|S_{1,1}| &=& m + n-2 + (m-2) \left (\frac{(n-2) - 5}{3} + 1 \right ) + 2 \left ( \frac{(m-2) - 5}{3} + 1 \right ) + 1 \nonumber \\
				  &=& m + n-1 + (m-2)\frac{n-4}{3} + 2 \frac{m-4}{3} \nonumber \\
				  &=& \frac{3m + 3n-3 + mn-4m-2n+8 +2m-8 }{3} 
				   =  \frac{mn + m + n - 3}{3}  \nonumber \\
				  &=& \frac{mn-1}{3} + 1 + \frac{m-1}{3} + \frac{n-1}{3} 
				   =  \frac{mn-1}{3} + 1 + \lfloor \frac{m}{3} \rfloor + \lfloor \frac{n}{3} \rfloor - 1 \nonumber \\
				  &=& \frac{mn - {a'}_{m,n} }{3} + \bar{r}'_{m,n} + {c'}_{m,n}. 
	\end{eqnarray}
	Note that 
	\[
		{a'}_{m,n} = (m \mod 3) \cdot (n \mod 3) = 1, 
	\]
	\[
		\bar{r}'_{m,n} = 1, 
	\]
	and 
	\[
		{c'}_{m,n} = \lfloor \frac{m}{3} \rfloor + \lfloor \frac{n}{3} \rfloor - 1. 
	\]
	\noindent
	{\bf\boldmath $(m \mod 3) = 1$ and $(n \mod 3) = 2$:}
	We consider the vertex set 
	$S_{1,2} = A \cup B \cup C_{1,2} \cup D_{1,2} \cup E_{1,2}$, 
	in which 
	$C_{1,2} = \{ (3, y), \cdots, (m, y) \mid y = 5, \ldots, n-3 \}$, 
	$D_{1,2} = \{ (x, n-2), (x, n-1), (x, n) \mid x = 5, \ldots, m-2 \}$ 
	and 
	$E_{1,2} = \{ (m, n-2), (m, n-1) \}$. 
	Vertices in $C_{1,2}$ ($D_{1,2}$ and $E_{1,2}$, respectively)
	dominate 
	vertices in $\{ (3, y-1), \cdots, (m, y-1), (3, y+1), \cdots, (m, y+1) \mid y = 5, \ldots, n-3 \}$
	($\{ (x-1, n-2), (x-1, n-1), (x-1, n), (x+1, n-2), (x+1, n-1), (x+1, n) \mid x = 5, \ldots, m-2 \}$ and $\{ (m, n-1), (m, n) \}$, respectively). 
	Since $S_{1,2}$ is connected, 
	$S_{1,2}$ is a CDS. 
	$|C_{1,2}| = (m-2)(\frac{(n-3) - 5}{3} + 1)$, 
	$|D_{1,2}| = 3 (\frac{(m-2) - 5}{3} + 1)$ and 
	$|E_{1,2}| = 2$. 
	Hence, 
	we have 
	\begin{eqnarray} \label{LMA:UB:eq.7}
		|S_{1,2}| &=& m + n-2 + (m-2) \left (\frac{(n-3) - 5}{3} + 1 \right ) + 3 \left (\frac{(m-2) - 5}{3} + 1 \right ) + 2 \nonumber\\
				  &=& m + n + (m-2) \frac{n-5}{3} + 3 \frac{m-4}{3} \nonumber \\
				  &=& \frac{3m + 3n + mn-5m-2n+10 + 3m-12}{3} 
				   =  \frac{mn + m + n - 2}{3} \nonumber\\
				  &=& \frac{mn-2}{3} + 2 + \frac{m-1}{3} + \frac{n-2}{3} - 1 
				   =  \frac{mn-2}{3} + 2 + \lfloor \frac{m}{3} \rfloor + \lfloor \frac{n}{3} \rfloor - 1 \nonumber \\
				  &=& \frac{mn - {a'}_{m,n} }{3} + \bar{r}'_{m,n} + {c'}_{m,n}. 
	\end{eqnarray}
	Note that 
	\[
		{a'}_{m,n} = (m \mod 3) \cdot (n \mod 3) = 2, 
	\]
	\[
		\bar{r}'_{m,n} = 2
	\]
	and 
	\[
		{c'}_{m,n} = \lfloor \frac{m}{3} \rfloor + \lfloor \frac{n}{3} \rfloor - 1. 
	\]
	\noindent
	{\bf\boldmath $(m \mod 3) = 2$ and $(n \mod 3) = 1$:}
	We consider the vertex set $S_{2,1} = A \cup B \cup C_{2,1} \cup D_{2,1} \cup E_{2,1}$, 
	in which 
	$C_{2,1} = \{ (3, y), \cdots, (m, y) \mid y = 5, \ldots, n-2 \}$, 
	$D_{2,1} = \{ (x, n-1), (x, n) \mid x = 5, \ldots, m-3 \}$
	and 
	$E_{2,1} = \{ (m, n-1), (m, n) \}$. 
	Vertices in $C_{2,1}$ ($D_{2,1}$ and $E_{2,1}$, respectively)
	dominate 
	$\{ (3, y-1), \cdots, (m, y-1), (3, y+1), \cdots, (m, y+1) \mid y = 5, \ldots, n-2 \}$
	($\{ (x-1, n-1), (x-1, n), (x+1, n-1), (x+1, n) \mid x = 5, \ldots, m-3 \}$ and $\{ (m-1, n), (m, n) \}$, respectively). 
	Since $S_{2,1}$ is connected, 
	$S_{2,1}$ is a CDS. 
	$|A| = m$, 
	$|B| = n-2$, 
	$|C_{2,1}| = (m-2)(\frac{(n-2) - 5}{3} + 1)$, 
	$|D_{2,1}| = 2 (\frac{(m-3) - 5}{3} + 1)$ and 
	$|E_{2,1}| = 2$. 
	Thus, 
	we have 
	\begin{eqnarray} \label{LMA:UB:eq.8}
		|S_{1,2}| &=& m + n-2 + (m-2) \left (\frac{(n-2) - 5}{3} + 1 \right ) + 2 \left (\frac{(m-3) - 5}{3} + 1 \right ) + 2 \nonumber\\
				  &=& m + n + (m-2) \frac{n-4}{3} + 2 \frac{m-5}{3} \nonumber\\
				  &=& \frac{3m + 3n + mn-4m-2n+8 + 2m-10}{3} 
				   =  \frac{mn + m + n - 2}{3} \nonumber\\
				  &=& \frac{mn-2}{3} + 2 + \frac{m-2}{3} + \frac{n-1}{3} - 1 
				   =  \frac{mn-2}{3} + 2 + \lfloor \frac{m}{3} \rfloor + \lfloor \frac{n}{3} \rfloor - 1 \nonumber\\
				  &=& \frac{mn - {a'}_{m,n} }{3} + \bar{r}'_{m,n} + {c'}_{m,n}. 
	\end{eqnarray}
	Note that 
	\[
		{a'}_{m,n} = (m \mod 3) \cdot (n \mod 3) = 2, 
	\]
	\[
		\bar{r}'_{m,n} = 2
	\]
	and 
	\[
		{c'}_{m,n} = \lfloor \frac{m}{3} \rfloor + \lfloor \frac{n}{3} \rfloor - 1. 
	\]
	\noindent
	{\bf\boldmath $(m \mod 3) = 2$ and $(n \mod 3) = 2$:}
	We consider the vertex set 
	$S_{2,2} = A \cup B \cup C_{2,2} \cup D_{2,2} \cup E_{2,2}$, 
	in which 
	$C_{2,2} = \{ (3, y), \cdots, (m, y) \mid y = 5, \ldots, n-3 \}$, 
	$D_{2,2} = \{ (x, n-2), (x, n-1), (x, n) \mid x = 5, \ldots, m-3 \}$ 
	and 
	$E_{2,2} = \{ (m, n-2), (m, n-1), (m, n) \}$. 
	Vertices in $C_{2,2}$
	($D_{2,2}$ and $E_{2,2}$, respectively)
	dominate
	$\{ (3, y-1), \cdots, (m, y-1), (3, y+1), \cdots, (m, y+1) \mid y = 5, \ldots, n-3 \}$
	($\{ (x-1, n-2), (x-1, n-1), (x-1, n), (x+1, n-2), (x+1, n-1), (x+1, n) \mid x = 5, \ldots, m-2 \}$
	 and $\{ (m-1, n-1), (m-1, n), (m, n-1), (m, n) \}$, respectively). 
	Since $S_{2,2}$ is connected, 
	$S_{2,2}$ is a CDS. 
	$|A| = m$, 
	$|B| = n-2$, 
	$|C_{2,2}| = (m-2)(\frac{(n-3) - 5}{3} + 1)$, 
	$|D_{2,2}| = 3 (\frac{(m-3) - 5}{3} + 1)$ and 
	$|E_{2,2}| = 3$. 
	Hence, 
	we have 
	\begin{eqnarray} \label{LMA:UB:eq.9}
		|S_{2,2}| &=& m + n-2 + (m-2) \left (\frac{(n-3) - 5}{3} + 1 \right ) + 3 \left (\frac{(m-3) - 5}{3} + 1 \right ) + 3 \nonumber\\
				  &=& m + n+1 + (m-2) \frac{n-5}{3} + 3 \frac{m-5}{3} \nonumber\\
				  &=& \frac{3m + 3n+3 + mn-5m-2n+10 +3m-15}{3} 
				   =  \frac{mn + m + n - 2}{3} \nonumber\\
				  &=& \frac{mn-4}{3} + 3 + \frac{m-2}{3} + \frac{n-2}{3} - 1 
				   =  \frac{mn-4}{3} + 3 + \lfloor \frac{m}{3} \rfloor + \lfloor \frac{n}{3} \rfloor - 1 \nonumber\\
				  &=& \frac{mn - {a'}_{m,n} }{3} + \bar{r}'_{m,n} + {c'}_{m,n}. 
	\end{eqnarray}
	Note that 
	\[
		{a'}_{m,n} = (m \mod 3) \cdot (n \mod 3) = 4, 
	\]
	\[
		\bar{r}'_{m,n} = 3
	\]
	and 
	\[
		{c'}_{m,n} = \lfloor \frac{m}{3} \rfloor + \lfloor \frac{n}{3} \rfloor - 1. 
	\]
	By the above equalities from (\ref{LMA:UB:eq.3}) to (\ref{LMA:UB:eq.9}), 
	for any $m \geq 4, n \geq 4$, 
	we have 
	\[
		\gamma_{m,n} \leq |S_{i, j}|, 
	\]
	in which 
	$(i, j) \in \{ (0, 0), (0, *), (*, 0), (1, 1), (1, 2), (2, 1), (2, 2) \}$. 
	Therefore, 
	we have shown the inequalities in the statement of this lemma. 
	\fi
\end{proof}
\ifnum \count12 > 0
\begin{figure*}
	 \begin{center}
	  \includegraphics[width=\linewidth]{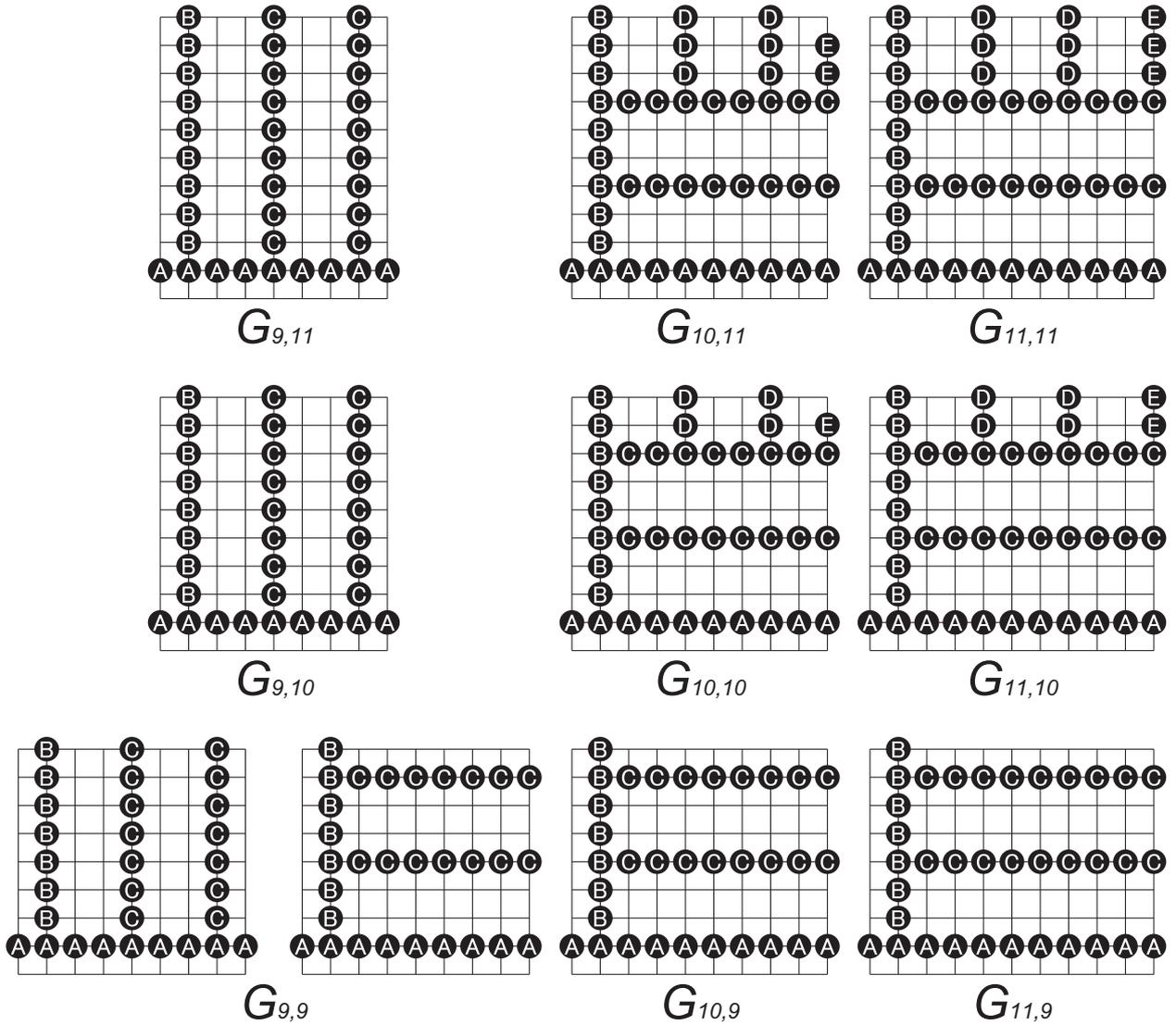}
	 \end{center}
	 \caption{
\ifnum \count10 > 0
%
%
\fi
\ifnum \count11 > 0
%
%
Upper bound instances of CDSs for any $m, n \in [9, 11]$. 
Black circles denote vertices in the CDSs. 
A vertex labeled with white letter A (B,C,D and E, respectively) belongs to the set $A$ ($B, C_{i,j}, D_{i,j}$ and $E_{i,j} \hspace{1mm} ((i,j) \in \{ (0, 0), (0, *), (*, 0), (1, 1), (1, 2), (2, 1), (2, 2) \}), respectively)$ for each $m,n$. 
\fi
			}
	\label{fig:up}
\end{figure*}
\fi
%

\section{Lower Bounds} \label{sec:low}

\subsection{Regularity} \label{sec:regularity}
\ifnum \count10 > 0
%
%
\fi
\ifnum \count11 > 0
%
%
Let us define an MCDS satisfying key properties 
to show lower bounds on the connected domination number. 
For an MCDS $D$, 
$p \in [1, m]$, 
any $q \in [2, n]$, 
$p' \in [p-1, m]$ and 
$q' \in [q-1, n]$ such that 
either $q' = q$ and $p' \ne p$ or 
$p' = p$ and $q' \ne q$, 
we say that $D$ is {\em \boldmath $(p', q')$-$(p, q)$-regular} 
(Fig.~\ref{fig:regular})
if $D$ satisfies the following conditions: 
When $q' = q$ and $p' \ne p$, 
\begin{itemize}
	\itemsep=-2.0pt
	\setlength{\leftskip}{0pt}
	\item[(Q1)]
		$(p'+1, q), \cdots, (p, q) \in D$, 
	\item[(Q2)]
		if $q \geq 2$ and $p'+1 \geq p-1$, 
		then $(p'+1, q-1), \cdots, (p-1, q-1) \notin D$, 
	\item[(Q3)]
		if $q \geq 2$ and $p = m$, 
		then $(m, q-1) \notin D$, 
	\item[(Q4)]
		if $q \geq 3$, 
		then $(p'+1, q-2), \cdots, (p, q-2) \notin D$, 
	\item[(Q5)]
		if $q \geq 4$, 
		then $(p'+1, q-3), \cdots, (m, q-3) \in D$, 
	\item[(Q6)]
		if $q \geq 4$ and $p' \geq 2$, 
		then $(p', q-3), \cdots, (p', n) \in D$, 
	\item[(Q7)]
		if $q \geq 2$ and $p' \geq 2$, 
		then $(p'-1, q-1), \cdots, (p'-1, n) \notin D$ 
		and 
		if $q \geq 3$ and $p' \geq 2$, 
		then $(p'-1, q-2), \cdots, (p'-1, n) \notin D$, 
	\item[(Q8)]
		if $q \geq 5$, 
		then $(p'+1, q-4), \cdots, (m, q-4) \notin D$ and 
	\item[(Q9)]
		if $q = n-1$, 
		then $(p'+1, n), \cdots, (p-1, n) \notin D$. 
\end{itemize}
When $p' = p$ and $q' \ne q$, 
\begin{itemize}
	\itemsep=-2.0pt
	\setlength{\leftskip}{0pt}
	\item[(P1)]
		$(p, q'+1), \cdots, (p, q) \in D$, 
	\item[(P2)]
		if $p \geq 2$ and $q'+1 \geq q-1$, 
		then $(p-1, q'+1), \cdots, (p-1, q-1) \notin D$, 
	\item[(P3)]
		if $p \geq 2$ and $q = n$, 
		then $(p-1, n) \notin D$ 
	\item[(P4)]
		if $p \geq 3$, 
		then $(p-2, q'+1), \cdots, (p-2, q) \notin D$ 
	\item[(P5)]
		if $p \geq 4$, 
		then $(p-3, q'+1), \cdots, (p-3, m) \in D$ 
	\item[(P6)]
		if $p \geq 4$, 
		then $(p-3, q'), \cdots, (m, q') \in D$, and 
		if $p = 2$, 
		then $(1, q'), \cdots, (m, q') \in D$, 
	\item[(P7)]
		if $p \geq 2$ and $q' \geq 2$, 
		then $(p-1, q'-1), \cdots, (m, q'-1) \notin D$, and 
		if $p \geq 3$ and $q' \geq 2$, 
		then $(p-2, q'-1), \cdots, (m, q'-1) \notin D$, 
	\item[(P8)]
		if $p \geq 5$, 
		then $(p-4, p'+1), \cdots, (p-4, n) \notin D$ and 
	\item[(P9)]
		if $p = m-1$, 
		then $(m, q'+1), \cdots, (m, q-1) \notin D$. 
\end{itemize}
Additionally, 
if $D \subseteq V$ is $(p', q')$-$(p, q)$-regular, 
then we say that vertices in $V$ satisfying the following conditions are {\em \boldmath $(p', q')$-$(p, q)$-regular}: 
When $q' = q$ and $p' \ne p$, 
\begin{itemize}
	\itemsep=-2.0pt
	\setlength{\leftskip}{10pt}
		\item
			$(p'+1, q), \cdots, (p, q)$, 
		\item
			if $q \geq 2$, 
			then $(p'+1, q-1), \cdots, (p-1, q-1)$, 
		\item
			if $q \geq 3$, 
			then $(p'+1, q-2), \cdots, (p, q-2)$, 
		\item
			if $q \geq 4$, 
			then for each $y = 1, \ldots, q-3$, $(p'+1, y), \cdots, (m, y)$ and 
		\item
			if $p' \geq 1$, 
			then for $y = 1, \ldots, n$, $(1, y), \cdots, (p', y)$. 
\end{itemize}
When $p' = p$ and $q' \ne q$, 
\begin{itemize}
	\itemsep=-2.0pt
	\setlength{\leftskip}{10pt}
		\item
			$(p, q'+1), \cdots, (p, q)$, 
		\item
			if $p \geq 2$, 
			then $(p-1, q'+1), \cdots, (p, q-1)$
		\item
			if $p \geq 3$, 
			then $(p-2, q'+1), \cdots, (p, q)$
		\item
			if $p \geq 4$, 
			then for each $x = 1, \ldots, p-3$, $(x, q'+1), \cdots, (x, n)$ and 
		\item
			if $q' \geq 1$, 
			then for each $x = 1, \ldots, m$, $(x, 1), \cdots, (x, q')$. 
\end{itemize}
We will drop ``$(p', q')$-$(p, q)$-'' when it is clear from the context. 
Also, 
${R}(D)$ denotes the set of $(p', q')$-$(p, q)$-regular vertices in $D$. 
We say that vertices in $V$ are not $(p', q')$-$(p, q)$-regular are {\em \boldmath $(p', q')$-$(p, q)$-irregular}. 
We will drop ``$(p', q')$-$(p, q)$-'' when it is clear from the context as well. 
${\overline{R}}(D)$ denotes the set of $(p', q')$-$(p, q)$-irregular vertices in $D$. 
That is, 
${\overline{R}}(D) = D \backslash {R}(D)$. 
\fi
\ifnum \count12 > 0
\begin{figure*}
	 \begin{center}
	  \includegraphics[width=\linewidth]{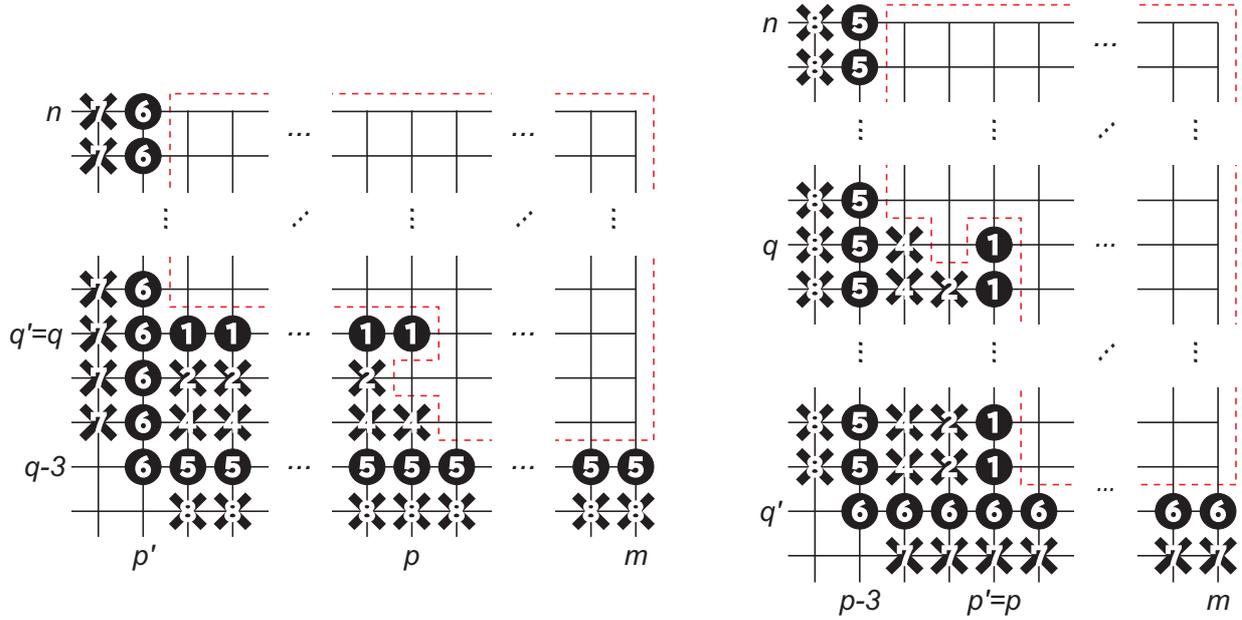}
	 \end{center}
	 \caption{
\ifnum \count10 > 0
%
%
\fi
\ifnum \count11 > 0
%
%
Suppose that $D$ is a $(p', q')$-$(p, q)$-regular MCDS. 
The left (right) figure denotes the case in which $q' = q$ and $p' \ne p$ 
($q' \ne q$ and $p' = p$). 
Circles and crosses denote regular vertices in $D$ and not in $D$, respectively. 
Vertices enclosed by dashed lines are irregular, and the others are regular. 
Numbers on circles and crosses represent the numbers of the conditions for the regularity. 
For example, 
``6'' on the vertex $(q-3, p')$ in the left figure means that the vertex satisfies the condition (Q6). 
\fi
			}
	\label{fig:regular}
\end{figure*}
\fi
%

%
\subsection{Overview of the Lower Bound Proof} \label{sec:overview}
\ifnum \count10 > 0
%
\fi
\ifnum \count11 > 0
%
%
We evaluate the connected domination number for any $m(\geq 4), n(\geq 4)$ using the regularity. 
We show the first lemma. 
\fi
\begin{LMA}\label{LMA:L1}
	\ifnum \count10 > 0
	%
	%
	\fi
	\ifnum \count11 > 0
	%
	%
	For any $m(\geq 4)$ and $n(\geq 4)$, 
	suppose that there exists an MCDS without the vertex $(1, 2)$
	of the $m \times n$ grid graph. 
	Then, 
	there exists an MCDS containing $(1, 2)$ of the $n \times m$ grid graph. 
	\fi
\end{LMA}
\begin{proof}
	\ifnum \count10 > 0
	%
	%
	\fi
	\ifnum \count11 > 0
	%
	%
	Any MCDS contains a vertex to dominate the vertex $(1, 1)$ and is connected. 
	Thus, 
	any MCDS contains either $(1, 2)$ or $(2, 1)$. 
	For some $m$ and $n$, 
	let $D$ be an MCDS without $(1, 2)$ of the $m \times n$ grid graph. 
	That is, 
	$D$ contains $(2, 1)$. 
	Now, 
	let $D'$ be the mirror image of $D$. 
	Specifically, 
	$D'$ contains a vertex $(y, x)$ 
	if and only if $D$ contains the vertex $(x, y)$. 
	$D'$ is the MCDS of the $n \times m$ grid graph 
	and contains $(1, 2)$. 
	\fi
\end{proof}
\ifnum \count10 > 0
%
%
\fi
\ifnum \count11 > 0
%
%
By symmetry, 
the connected domination number of an $m \times n$ grid graph is equal to that of an $n \times m$ grid graph, that is, $\gamma_{m,n} = \gamma_{n,m}$. 
Hence, 
for all MCDSs containing $(1, 2)$, deriving lower bounds on their elements, we can obtain the lower bounds for all MCDSs. 
Then, 
in what follows, 
we consider only MCDSs containing $(1, 2)$. 
Let us consider an MCDS containing the vertex $(1, 2)$. 
The MCDS is $(0, 2)$-$(1, 2)$-regular by definition. 
Starting from the MCDS, 
we increase regular vertices in the MCDS sequentially according to some routine, 
which is defined later. 
Specifically, 
the routine constructs a new vertex set $D'$ by removing some irregular vertices from an MCDS $D$ and adding vertices to $D$ 
such that the number of added vertices is equal to the removed vertices but the number of regular vertices in $D'$ is larger than that in $D$, and $D'$ is a CDS. 
That is, 
$D'$ is an MCDS with more regular vertices than $D$. 
The routine repeats this construction according to some rules 
until it cannot increase regular vertices, 
and then it is finished.  
Let us denote an MCDS at the time when the routine is finished as $D_{\tau}$, 
which is referred to as the ``completely regular MCDS'' in Sec.~\ref{sec:intro}. 
We give two definitions. 
For a $(p', q')$-$(p, q)$-regular MCDS $D$, $x \in [2, m-1]$ and $y \in [2, n-1]$, 
a {\em connector in $D$} is called a vertex $(x, y) \in {R}(D)$ satisfying one of the following two conditions: 
\begin{itemize}
	\itemsep=-2.0pt
	\setlength{\leftskip}{0pt}
	\item[(c-i)]
		$(x+1, y), (x-1, y), (x-1, y-1), (x-1, y+1) \in {R}(D)$
		and 
		$(x, y-1) \notin {R}(D)$ or 

	\item[(c-ii)]
		$(x, y+1), (x, y-1), (x-1, y-1), (x+1, y-1) \in {R}(D)$
		and 
		$(x-1, y) \notin {R}(D)$. 
\end{itemize}
Also, 
a {\em pre-connector in $D$} is called a vertex $(x, y) \in {R}(D)$ satisfying one of the following two conditions: 
\begin{itemize}
	\itemsep=-2.0pt
	\setlength{\leftskip}{0pt}
	\item[(p-i)]
		$(x-1, y), (x-1, y-1), (x-1, y+1) \in {R}(D)$
		and 
		$(x+1, y)$
		or 

	\item[(p-ii)]
		$(x, y+1), (x, y-1), (x-1, y-1), (x+1, y-1) \in {R}(D)$
		and 
		$(x, y+1)$. 
\end{itemize}
We call a vertex in ${R}(D)$ which is neither a connector nor a pre-connector a {\em dominator in $D$}. 
\fi
\ifnum \count12 > 0
\begin{figure*}
	 \begin{center}
	  \includegraphics[width=\linewidth]{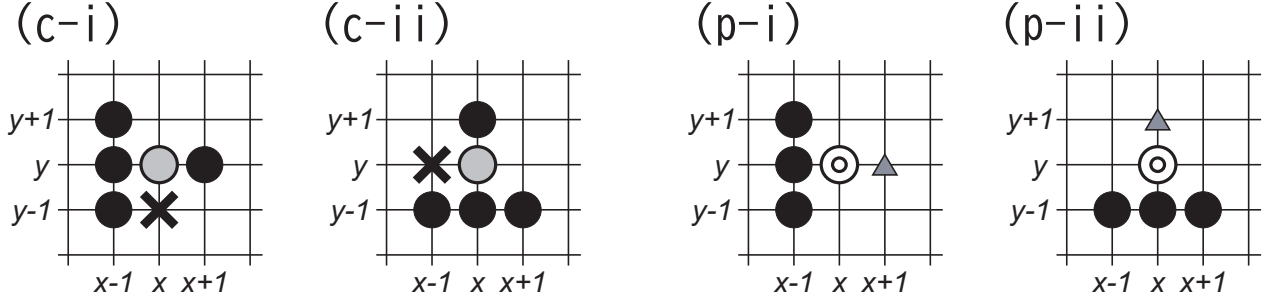}
	 \end{center}
	 \caption{
\ifnum \count10 > 0
%
%
\fi
\ifnum \count11 > 0
%
%
Black circles denote vertices in ${R}(D)$, 
crosses denote vertices not in ${R}(D)$ and 
gray triangles denote irregular vertices. 
Vertices satisfying (c-i) or (c-ii) are connectors, 
denoted by gray circles, and 
vertices satisfying (p-i) or (p-ii) are pre-connectors, 
denoted by double circles. 
\fi
			}
	\label{fig:cr}
\end{figure*}
\fi
\ifnum \count10 > 0
%
%
\fi
\ifnum \count11 > 0
%
%
In Lemma~\ref{LMA:PCR}, 
we show that $D_{\tau}$ does not contain any pre-connector. 
In Lemma~\ref{LMA:CR0}, 
we show that for any connector, vertices dominated by the connector are dominated by at least one dominator. 
These two lemmas indicate that 
the number of vertices by regular vertices in $D_{\tau}$ is equal to that of vertices dominated by dominators in $D_{\tau}$. 
Furthermore, 
in Lemma~\ref{LMA:DR}, 
we show that 
dominators dominate $3 {d}_{m,n}$ vertices, 
in which ${d}_{m,n}$ denotes the number of dominators in $D_{\tau}$. 
Since the number of vertices in an $m \times n$ grid graph is $mn$, 
we have 
\begin{equation} \label{EQ:sec:overview.eq1}
	mn = 3 {d}_{m,n} + {a}_{m,n}, 
\end{equation}
in which 
${a}_{m,n}$ denotes the number of vertices which are not dominated by regular vertices in $D_{\tau}$ 
at the time when the routine is finished. 
Let $\bar{r}_{m,n}$ denote the number of irregular vertices in $D_{\tau}$, 
that is, 
$\bar{r}_{m,n} = |{\overline{R}}(D_{\tau})|$. 
Let ${c}_{m,n}$ denote the number of connectors in $D_{\tau}$. 
Since regular vertices in $D_{\tau}$ are connectors or dominators by Lemma~\ref{LMA:PCR},
we have 
\begin{equation} \label{EQ:sec:overview.eq2}
	\gamma_{m,n} = {d}_{m,n} + {c}_{m,n} + \bar{r}_{m,n}. 
\end{equation}
In Lemma~\ref{LMA:IR}, we show that 
\begin{equation} \label{EQ:sec:overview.eq3}
	a_{m,n} = (m \mod 3) \cdot (n \mod 3)
\end{equation}
and 
\begin{equation} \label{EQ:sec:overview.eq4}
	\bar{r}_{m,n} = 
		\begin{cases}
			3 & \hspace{10mm} (m \mod 3) \cdot (n \mod 3) = 4 \\
			2 & \hspace{10mm} (m \mod 3) \cdot (n \mod 3) = 2 \\
			1 & \hspace{10mm} (m \mod 3) \cdot (n \mod 3) = 1 \\
			0 & \hspace{10mm} (m \mod 3) \cdot (n \mod 3) = 0. 
		\end{cases}
\end{equation}
Also, in Lemma~\ref{LMA:CR1}, 
we show that 
\begin{equation} \label{EQ:sec:overview.eq5}
	{c}_{m,n} \geq 
		\begin{cases}
			\min\{ \frac{m}{3}, \frac{n}{3} \} & (m \mod 3) = 0   \mbox{ and } (n \mod 3) = 0 \\
			\frac{m}{3} & (m \mod 3) = 0   \mbox{ and } (n \mod 3) \ne 0 \\
			\frac{n}{3} & (m \mod 3) \ne 0 \mbox{ and } (n \mod 3) = 0 \\
			\lfloor \frac{m}{3} \rfloor + \lfloor \frac{n}{3} \rfloor - 1 & (m \mod 3) \ne 0 \mbox{ and } (n \mod 3) \ne 0. \\
		\end{cases}
\end{equation}
By Eqs~(\ref{EQ:sec:overview.eq1}),(\ref{EQ:sec:overview.eq2}),(\ref{EQ:sec:overview.eq3}),(\ref{EQ:sec:overview.eq4}) and (\ref{EQ:sec:overview.eq5}),
we have the lower bound lemma: 
\fi
\begin{LMA} \label{LMA:low}
	\ifnum \count10 > 0
	%
	%
	\fi
	\ifnum \count11 > 0
	%
	%
	For any $m \geq 4$ and $n \geq 4$, 
	\[		
		\gamma_{m, n} \geq \frac{mn - {a'}_{m,n} }{3} + \bar{r}'_{m,n} + {c'}_{m,n}, 
	\]
	in which 
	\[
		{a'}_{m,n} = (m \mod 3) \cdot (n \mod 3), 
	\]
	\[
		\bar{r}'_{m,n} = 
			\begin{cases}
					3 & \hspace{10mm} (m \mod 3) \cdot (n \mod 3) = 4 \\
					2 & \hspace{10mm} (m \mod 3) \cdot (n \mod 3) = 2 \\
					1 & \hspace{10mm} (m \mod 3) \cdot (n \mod 3) = 1 \\
					0 & \hspace{10mm} (m \mod 3) \cdot (n \mod 3) = 0
			\end{cases}	
	\]
	and 
	\[
		{c'}_{m,n} = 
			\begin{cases}
				\min\{ \frac{m}{3}, \frac{n}{3} \} & (m \mod 3) = 0   \mbox{ and } (n \mod 3) = 0 \\
				\frac{m}{3} & (m \mod 3) = 0   \mbox{ and } (n \mod 3) \ne 0 \\
				\frac{n}{3} & (m \mod 3) \ne 0 \mbox{ and } (n \mod 3) = 0 \\
				\lfloor \frac{m}{3} \rfloor + \lfloor \frac{n}{3} \rfloor - 1 & (m \mod 3) \ne 0 \mbox{ and } (n \mod 3) \ne 0. \\
			\end{cases}
	\]
	\fi
\end{LMA}
\ifnum \count10 > 0
%
\fi
\ifnum \count11 > 0
%
%

%
\fi
%

%
\subsection{Regularization Routine} \label{sec:routine}
\ifnum \count10 > 0
\fi
\ifnum \count11 > 0
%
%
We give some definitions to define the routine. 
Suppose that $D$ is a $(p', q')$-$(p, q)$-regular CDS. 
Also, 
suppose that $D'$ is a CDS such that 
$D' = (D \backslash U) \cup U' = D \backslash U \cup U'$, 
in which 
$U, U' \subseteq {\overline{R}}(D)$, 
and 
$|U| = |U'|$, 
that is, 
all the vertices in $U$ and $U'$ are irregular 
and 
the numbers of vertices in the two sets are equal. 
Then, 
we say that {\em $D$ is regularized into $D'$}.
Furthermore, 
if $D'$ is a $(\hat{p}', \hat{q}')$-$(\hat{p}, \hat{q})$-regular CDS, 
then we say that {\em $D$ is $(\hat{p}', \hat{q}')$-$(\hat{p}, \hat{q})$-regularized into $D'$}. 
Note that if $D$ is an MCDS, then $D'$ is also an MCDS. 
In the routine, 
note that the conditions of Cases~3.2 and 3.3 are not exclusive. 
Specifically, 
if $p' \ne 0$, $q' \leq n-4$ and $p' \leq m-4$, 
then the routine executes only one case of Cases~3.2 and 3.3. 
In the beginning, 
the routine initializes the variable {\tt D} as an MCDS containing $(1, 2)$. 
The routine repeatedly executes Cases~1, 2, 3.1, 3.2 and 3.3 depending on the MCDS of {\tt D} 
and regularizes it 
until the routine executes Case~3.4 and is finished. 
\fi
\ifnum \count10 > 0
\fi
\ifnum \count11 > 0
%
%
\noindent\vspace{-1mm}\rule{\textwidth}{0.5mm} 
\vspace{-3mm}
{\sc Regularization Routine}\\
\rule{\textwidth}{0.1mm}
	{\bf Initialization:}
	{\tt D} $:=$ a $(0, 2)$-$(1, 2)$-regular MCDS (see Fig. ~\ref{fig:routine}). \\
	Suppose that {\tt D} is $(p', q')$-$(p, q)$-regular. 
	Execute one of the following cases: 
	\\
	{\bf\boldmath Case~1 ($q' = q$ and $p \leq m-1$):}\\
\hspace*{1mm}
		{\tt D} $:=$ an MCDS into which {\tt D} is $(p', q')$-$(p+1, q)$-regularized 
		(guaranteed in Lemma~\ref{LMA:L4}). 
		\\
	{\bf\boldmath Case~2 ($p' = p$ and $q \leq n-1$):}\\
\hspace*{1mm}
		{\tt D} $:=$ an MCDS into which {\tt D} is $(p', q')$-$(p, q+1)$-regularized 
		(guaranteed in Lemma~\ref{LMA:L4}). 
		\\
	{\bf\boldmath Case~3 (either $q'=q$ and $p=m$ or $p'=p$ and $q=n$):}
		Execute one of the four cases: \\
\hspace*{2mm}
			{\bf\boldmath Case~3.1 ($p' = 0$):}\\
\hspace*{3mm}
				{\tt D} $:=$ an MCDS into which {\tt D} is $(2, 2)$-$(2, 3)$-regularized 
				(guaranteed in Lemma~\ref{LMA:L6}). 
				\\
\hspace*{2mm}
			{\bf\boldmath Case~3.2 ($p' \ne 0$, $p' \leq m-2$ and $q' \leq n-4$):}\\
\hspace*{3mm}
				{\tt D} $:=$ an MCDS into which {\tt D} is $(p', q'+3)$-$(p'+1, q'+3)$-regularized 
				(guaranteed in Lemma~\ref{LMA:C9}). 
				\\
\hspace*{2mm}
			{\bf\boldmath Case~3.3 ($p' \ne 0$, $p' \leq m-4$ and $q' \leq n-2$):}\\
\hspace*{3mm}
				{\tt D} $:=$ an MCDS into which {\tt D} is $(p'+3, q')$-$(p'+3, q'+1)$-regularized 
				(guaranteed in Lemma~\ref{LMA:C9}). 
				\\
\hspace*{2mm}
			{\bf\boldmath Case~3.4 
			(either $p' = p = m-1$, 
			$q' = q = n-1$ or 
			both $m-3 \leq p' \leq m-2$ and $n-3 \leq q' \leq n-2$):}\\
\hspace*{3mm}
				Finish. 
				\\
\noindent\vspace{-1mm}\rule{\textwidth}{0.5mm} 
\fi
\ifnum \count12 > 0
\begin{figure*}
	 \begin{center}
	  \includegraphics[width=\linewidth]{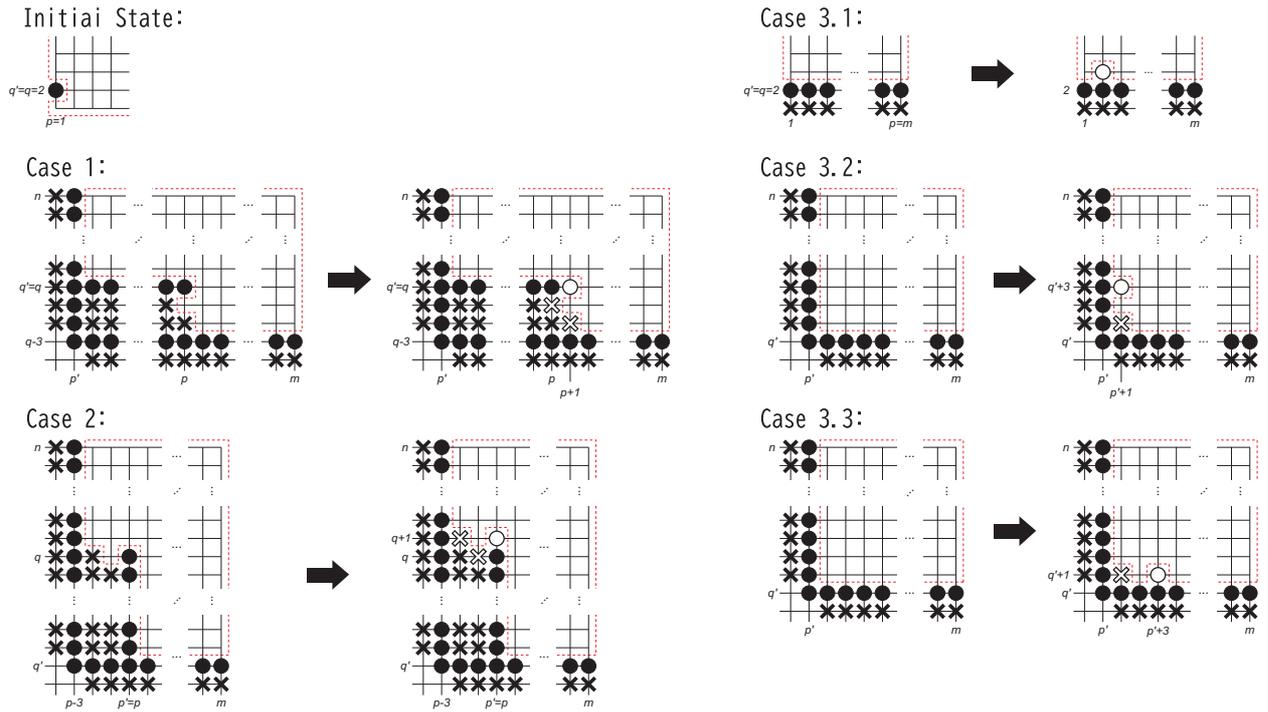}
	 \end{center}
	 \caption{
\ifnum \count10 > 0
\fi
\ifnum \count11 > 0
%
%
Regularization routine. 
Black circles and crosses denote regular vertices in {\tt D} and not in {\tt D}, respectively. 
White circles and crosses denote regular vertices which are newly added into {\tt D} and are not added, respectively, when each case is executed. 
Vertices enclosed by dashed lines are irregular and the other vertices are regular. 
For each case,
the left figure on the left (right) of the arrow denotes the situation of {\tt D} immediately before (after) the routine execution. 
\fi
			}
	\label{fig:routine}
\end{figure*}
\fi
\ifnum \count10 > 0
%
%
%

%
\fi
\ifnum \count11 > 0
%
%
If each case is feasible, 
then the routine necessarily is finished
because the number of regular vertices is monotone increasing. 
\fi
%

%
\subsection{Regularity Property} \label{sec:regularityproperty}
\ifnum \count10 > 0
\fi
\ifnum \count11 > 0
%
%
In this section, 
we show some properties of MCDSs processed by the routine 
if the routine is feasible. 
First, 
we give some definitions. 
Suppose that the routine is feasible.
Then, 
let $\tau$ be the number of times that the routine is executed by the time that the routine is finished. 
Let $D_{0}$ be an MCDS which is given to the routine at initialization. 
$D_{i}$ denotes the value of the variable {\tt D} in the routine 
immediately after the routine is executed for the $i (= 1, 2, \ldots, \tau)$-th time. 
Let $p'_{i}, q'_{i}), p_{i}$ and $q_{i})$ be integers 
such that $D_{i}$ is $(p'_{i}, q'_{i})$-$(p_{i}, q_{i})$-regular. 
Also, 
we define $p'_{0} = 0, p_{0} = 1$
and 
$q'_{0} = q_{0} = 2$. 
$D_{0}$ is $(0, 2)$-$(1, 2)$-regular by definition, 
that is, $(p'_{0}, q'_{0})$-$(p_{0}, q_{0})$-regular. 
Let $k$ be the number of times that the routine executes Case~3.1, 3.2, 3.3 or 3.4. 
Let $t_{1}, \ldots, t_{k}$ be positive integers such that 
the routine executes Case~3.1, 3.2, 3.3 or 3.4 for the $t_{1}, \ldots, t_{k}$-th time and 
$t_{1} < t_{2} < \cdots < t_{k}$. 
When the routine executes Case~3.4, 
it is finished. 
Thus, 
the case which the routine executes for the $t_{k}$-th time is Case~3.4. 
That is, 
$t_{k} = \tau$. 
\fi
\begin{LMA}\label{LMA:PCR}
	\ifnum \count10 > 0
	\fi
	\ifnum \count11 > 0
	%
	%
	Suppose that the routine is feasible. 
	Then, 
	$D_{\tau}$ does not contain any pre-connector. 
	\fi
\end{LMA}
\begin{proof}
	\ifnum \count10 > 0
	%
	%
	\fi
	\ifnum \count11 > 0
	%
	%
	By the supposition of this lemma, 
	the routine is feasible. 
	In the following, 
	we will show the following two properties by induction on the number $j \in [1, \tau]$ of times that the routine is executed: 
	\begin{itemize}
		\itemsep=-2.0pt
		\setlength{\leftskip}{0pt}
	\item[(a)]
	If the routine executes Case~3.1, 3.2 or 3.3 for the $j$-th execution, 
	then $D_{j}$ contains one pre-connector. 
	\item[(b)]
	If the routine executes Case~1, 2 or 3.4 for the $j$-th execution, 
	then $D_{j}$ does not contain any pre-connector. 
	\end{itemize}
	If we show that both (a) and (b) hold, 
	it implies that the statement of this lemma holds 
	because the routine is finished at the execution of Case~3.4 
	and (b) shows that the MCDS $D_{\tau}$ does not contain any pre-connector. 
	When $j = 0$, 
	that is, 
	at the initialization of the routine, 
	$D_{0} = \{ (1, 2) \}$ holds. 
	By the definition of a pre-connector, 
	$D_{0}$ does not contain a pre-connector
	(see Fig.~\ref{fig:ldr}(1)). 
	We assume that (a) and (b) hold for the $j = t$-th execution of the routine  
	and prove that they also hold for the $t+1$-st execution. 
	We discuss each case depending on which Cases the routine executes for the $t$-th and $t+1$-st times. 
	Note that only irregular vertices are handled when $D_{t}$ is regularized into  $D_{t+1}$, 
	and regular vertices are not removed from $D_{t}$ when the routine constructs $D_{t+1}$. 
	\noindent
	{\bf\boldmath $t = 0$ or Case~1:} 
	If either $t = 0$ or 
	the routine executes Case~1 for the $t$-th execution,  
	then $D_{t}$ does not contain any pre-connector 
	immediately before the $t+1$-st execution by (b) in the induction hypothesis. 
	$q'_{t} = q_{t}$ and thus, 
	if $p_{t} \leq m-1$, then the routine executes Case~1 for the $t+1$-st execution and 
	if $p_{t} = m$, the routine executes Case~3.1, 3.2, 3.3 or 3.4. 
	That is, 
	the routine can execute Case~1, 3.1, 3.2, 3.3 or 3.4 for the $t+1$-st execution. 
	If the routine executes Case~1, 
	the regular vertices newly added into $D_{t+1}$ satisfy the definitions of neither connectors nor pre-connectors, 
	that is, the vertices are dominators
	((2),(3),(10),(11) and (12) Fig.~\ref{fig:ldr}).
	Hence, 
	in these cases, the number of pre-connectors does not change and 
	$D_{t+1}$ does not contain any pre-connector, 
	which means that (b) is true. 
	If the routine executes Case~3.1, 3.2 or 3.3, 
	the regular vertex newly added into $D_{t+1}$ is only one pre-connector ((4),(9) and (13) in Fig.~\ref{fig:ldr}) 
	and (a) is true. 
	If the routine executes Case~3.4, 
	no vertex is added into $D_{t+1}$ and (b) is true. 
	\noindent
	{\bf Case~2:} 
	We can show this case similarly to Case~1. 
	If the routine executes Case~2 for the $t$-th execution, 
	then $D_{t}$ does not contain any pre-connector immediately before the $t+1$-st execution by (b) in the induction hypothesis.  
	That is, 
	the routine can execute Case~2, 3.1, 3.2, 3.3 or 3.4 for the $t+1$-st execution. 
	If the routine executes Case~3.1, 3.2 or 3.3, 
	the regular vertex newly added into $D_{t+1}$ is only one pre-connector ((9),(13) in Fig.~\ref{fig:ldr}) and (a) is true. 
	If the routine executes Case~2, 
	the regular vertex newly added into $D_{t+1}$ is a dominator
	((5),(6),(7),(14),(15),(16) in Fig.~\ref{fig:ldr}).
	If the routine executes Case~3.4, 
	no vertex is added into $D_{t+1}$. 
	Hence, 
	the number of pre-connectors does not change in these cases and 
	$D_{t+1}$ does not contain any pre-connector. 
	Thus, 
	(b) is true. 
	\noindent
	{\bf Case~3.1, 3.2, 3.3:} 
	If the routine executes Case~3.1, 3.2 or 3.3 for the $t$-th execution, 
	$D_{t}$ contains one pre-connector 
	immediately before the $t+1$-st execution by (a) in the induction hypothesis. 
	If the routine executes Case~1 or 2 for the $t+1$-st execution, 
	the pre-connector in $D_{t}$ becomes a connector in $D_{t+1}$
	((5),(10),(14) in Fig.~\ref{fig:ldr}). 
	Thus, 
	$D_{t+1}$ contains any pre-connector. 
	In the following, 
	we show that the routine does not execute Case~3 (3.1, 3.2, 3.3) for the $t+1$-st execution. 
	If the routine executes Case~3.1 for the $t$-th execution, 
	then $(2, 3)$ in $D_{t}$ is a pre-connector. 
	Since we discuss the case of $n \geq 4$, 
	the routine executes Case~2 but not Case~3 for the $t+1$-st execution. 
	If the routine executes Case~3.2 for the $t$-th, 
	then $p'_{t} \leq m-2$ by the condition of Case~3.2. 
	Then, 
	the routine executes Case~1 but not Case~3 for the $t+1$-st 
	because $p_{t} = p'_{t+1} \leq m-1$. 
	Similarly, 
	if the routine executes Case~3.3 for the $t$-th, 
	then $q'_{t} \leq n-2$ by the condition of Case~3.3. 
	Then, 
	the routine executes Case~2 but not Case~3 for the $t+1$-st 
	because $q_{t} = q'_{t+1} \leq n-1$. 
	\noindent
	{\bf Case~3.4:} 
	If the routine executes Case~3.4 for the $t$-th execution, 
	then (a) and (b) hold by the induction hypothesis. 
	We have shown the (a) and (b) are true. 
	\fi
\end{proof}

\ifnum \count12 > 0
\begin{landscape}
\begin{figure*}
	 \begin{center}
	  \includegraphics[width=\linewidth]{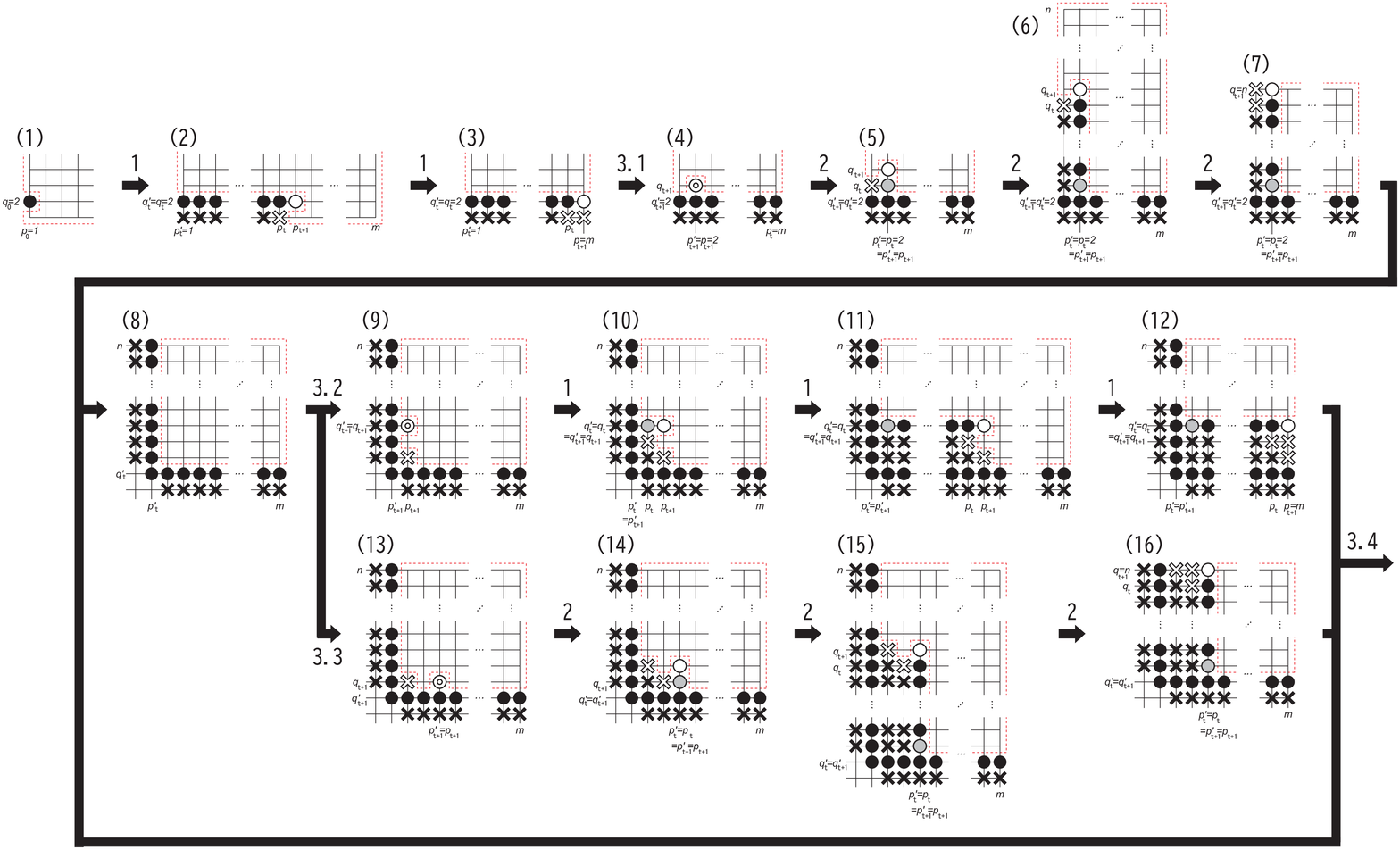}
	 \end{center}
	 \caption{
\ifnum \count10 > 0
\fi
\ifnum \count11 > 0
%
%
Situations of MCDS $D_{t+1}$. 
Circles and crosses denote regular vertices in $D_{t+1}$ and not in $D_{t+1}$, respectively. 
Vertices not enclosed by dashed lines are regular. 
White circles, double circles and white crosses become newly regular for the $t+1$-st execution. 
Double circles and gray circles denote pre-connectors and connectors, respectively. 
Each number above an arrow implies which Case the routine executes. 
For example, 
``3.1'' above the arrow between Figures~(3) and (4) means 
that the routine executes Case~3.1 in the situation of (3) and 
after that, 
$D_{t+1}$ becomes the situation of (4). 
The situations of $D_{t+1}$ when Case~3.4 is executed for the $t+1$-st execution are illustrated in Fig.~\ref{fig:lir}. 
\fi
			}
	\label{fig:ldr}
\end{figure*}
\end{landscape}
\fi

%
\begin{LMA}\label{LMA:DR}
	\ifnum \count10 > 0
	%
	%
	\fi
	\ifnum \count11 > 0
	%
	%
	Suppose that the routine is feasible. 
	Then, 
	the number of vertices dominated by dominators in $D_{\tau}$ is $3d_{m,n}$. 
	\fi
\end{LMA}
\begin{proof}
	\ifnum \count10 > 0
	%
	%
	\fi
	\ifnum \count11 > 0
	%
	%
	For any $i = 0, \ldots, \tau$, 
	$d_{i}$ denotes the number of dominators in $D_{i}$. 
	By the supposition of this lemma, 
	the routine is feasible. 
	In what follows,
	we will prove the following two properties by induction on the number $j \in [1, \tau]$ of the times that the routine is executed: 
	\begin{itemize}
		\itemsep=-2.0pt
		\setlength{\leftskip}{0pt}
	\item[(a)]
	If either $p'_{j} = p_{j}$ and $q_{j} \leq n-1$ or 
	$q'_{j} = q_{j}$ or $p_{j} \leq m-1$, 
	then the number of vertices dominated by either dominators or pre-connectors in $D_{j}$ is $3d_{j}+1$. 
	\item[(b)]
	If either $p'_{j} = p_{j}$ and $q_{j} = n$ or 
	$q'_{j} = q_{j}$ and $p_{j} = m$, 
	then the number of vertices dominated by either dominators or pre-connectors in $D_{j}$ is $3d_{j}$. 
	\end{itemize}
	When the routine executes Case~3.4 and is finished, 
	either $q'_{\tau} = q_{\tau}$ and $p_{\tau} = m$ 
	or 
	$p'_{\tau} = p_{\tau}$ and $q_{\tau} = n$
	by the condition of Case~3. 
	$D_{\tau}$ does not contain any pre-connector 
	by Lemma~\ref{LMA:PCR}. 
	Hence, 
	if these properties hold, 
	then (b) holds when the routine is finished, and 
	the number of vertices dominated by dominators in $D_{\tau}$ is $3d_{\tau}$.  
	By the definition of $d_{m,n}$, 
	the number of dominators at the time when the routine is finished is $d_{m,n}=d_{\tau}$ 
	and we can show the statement of the lemma. 
	When $j = 0$, that is, 
	at the time of the initialization, 
	$D_{0}$ is $(0, 2)$-$(1, 2)$-regular
	((1) in Fig.~\ref{fig:ldr}). 
	Since $D_{0} = \{ (1,2) \}$, $d_{0} = 1$. 
	The vertex $(1,2)$ dominates itself and three adjacent vertices $(1, 1), (1, 3)$ and $(2, 2)$, and the number of vertices dominated by a dominator in $D_{0}$ is four. 
	Hence, 
	when $j=0$, the property is true. 
	We assume that 
	(a)	and (b) hold when the routine is executed for the $j=t$-th execution, and 
	we show that they also hold for the $t+1$-st execution. 
	Note that only irregular vertices are handled at the regularization, 
	and regular vertices are not removed from $D_{t}$ when the routine constructs $D_{t+1}$. 
	\noindent
	{\bf Case~1:} 
	If the routine executes Case~1 
	for the $t+1$-st execution, 
	then $D_{t+1}$ is $(p'_{t}, q'_{t})$-$(p_{t}+1, q_{t})$-regular. 
	By definition, 
	the set of regular vertices in $D_{t+1}$ is the set of all the vertices in $D_{t}$ plus the vertex $(p_{t}+1, q_{t})$. 
	Since $(p_{t}, q_{t}-1) \notin R(D_{t+1})$ by the condition (Q2) of the regularity, 
	$(p_{t}+1, q_{t})$ satisfies the definition of neither connectors nor pre-connectors. 
	Thus, 
	$(p_{t}+1, q_{t})$ is a dominator and 
	we have 
	\begin{equation} \label{LMA:DR:eq.1}
		d_{t+1} = d_{t} + 1. 
	\end{equation}
	First, 
	we consider the case in which 
	$t = 0$ or 
	the routine executes Case~3.2 for the $t$-th execution, 
	that is, 
	$p_{t} = p'_{t}+1$ 
	(i.e., $p_{t+1} = p'_{t}+2$) 
	and 
	$p_{t+1} = p_{t}+1 \leq m-1$ 
	((10) in Fig.~\ref{fig:ldr}). 
	In this case, 
	it is sufficient that we prove (a) is true 
	because $p_{t}+1 \leq m-1$. 
	$(p_{t}+1, q_{t})$ dominates the five vertices $(p_{t}, q_{t}),(p_{t}+1, q_{t}),(p_{t}+1, q_{t}-1),(p_{t}+1, q_{t}+1)$ and $(p_{t}+2, q_{t})$. 
	By the definition of $(p'_{t}, q'_{t})$-$(p_{t}, q_{t})$-regularity, 
	vertices in $R(D_{t})$ do not dominate none of $(p_{t}+1, q_{t}-1),(p_{t}+1, q_{t}+1)$ and $(p_{t}+2, q_{t})$. 
	Thus, 
	the number of vertices which vertices in $R(D_{t+1})$ dominate increases by three vertices. 
	Since (a) is true for the $t$-th execution by the induction hypothesis, 
	dominators and a pre-connector in $D_{t}$ dominate $3 d_{t} +1$ vertices. 
	Thus, 
	we have using Eq.~(\ref{LMA:DR:eq.1}), 
	\[
		3 d_{t} + 1 + 3 = 3(d_{t} + 1) + 1 = 3 d_{t+1} + 1, 
	\]
	which implies that 
	(a)	is also true for $j = t+1$ 
	if $t = 0$ or 
	the routine executes Case~3.2 for the $t$-th execution. 
	Next, 
	we consider the case in which 
	the routine executes Case~1 for the $t$-th and $t+2$-nd executions, 
	that is, 
	$p_{t} > p'_{t}+1$ 
	(i.e., $p_{t+1} > p'_{t}+2$) 
	and 
	$p_{t+1} = p_{t}+1 \leq m-1$
	(Fig.~\ref{fig:ldr}(2),(3),(11)). 
	In this case, 
	it suffices to prove that (a) is true 
	because $p_{t}+1 \leq m-1$. 
	$(p_{t}+1, q_{t})$ dominates the five vertices $(p_{t}, q_{t}),(p_{t}+1, q_{t}),(p_{t}+1, q_{t}-1),(p_{t}+1, q_{t}+1)$ and $(p_{t}+2, q_{t})$, and 
	vertices in $R(D_{t})$ dominate none of $(p_{t}+1, q_{t}-1),(p_{t}+1, q_{t}+1)$ and $(p_{t}+2, q_{t})$ by the $(p'_{t}, q'_{t})$-$(p_{t}, q_{t})$-regularity. 
	Thus, 
	the number of vertices dominated by vertices in $R(D_{t+1})$ increase by three vertices 
	and by Eq.~(\ref{LMA:DR:eq.1}), 
	\[
		3 d_{t} + 1 + 3 = 3(d_{t} + 1) + 1 = 3 d_{t+1} + 1. 
	\]
	(a)	is true for $j = t+1$
	if the routine executes Case~1 for the $t$-th and $t+2$-nd executions. 
	Finally, 
	we consider the case in which 
	the routine executes Case~1 for the $t$-th execution and 
	does not execute Case~1 for the $t+2$-nd execution, 
	that is, 
	$p_{t+1} = p_{t}+1 = m$
	(Fig.~\ref{fig:ldr}(3),(12)). 
	In this case, 
	it suffices to prove that (b) is true by the condition. 
	$(p_{t}+1, q_{t})$ dominates the four vertices $(p_{t}, q_{t}),(m, q_{t}),(m, q_{t}-1)$ and $(m, q_{t}+1)$ and 
	vertices in $R(D_{t})$ dominate neither $(m, q_{t}-1)$ nor $(m, q_{t}+1)$
	by the $(p'_{t}, q'_{t})$-$(p_{t}, q_{t})$-regularity. 
	Thus, 
	the number of vertices dominated by vertices in $R(D_{t+1})$ increases by two vertices. 
	Using~Eq.~(\ref{LMA:DR:eq.1}), 
	we have 
	\[
		3 d_{t} + 1 + 2 = 3(d_{t} + 1) = 3 d_{t+1}, 
	\]
	which implies 
	(b) is also true in this case when $j = t+1$. 
	\noindent
	{\bf Case~2: } 
	We can prove this case similarly to the proof of Case~1. 
	If the routine executes Case~2 
	for the $t+1$-st execution, 
	then $D_{t+1}$ is $(p'_{t}, q'_{t})$-$(p_{t}, q_{t}+1)$-regular. 
	The set of regular vertices in $D_{t+1}$ is the set of all the vertices in $D_{t}$ plus the vertex $(p_{t}, q_{t}+1)$. 
	Since $(p_{t}-1, q_{t}) \notin R(D_{t+1})$ by the condition (P2) of the regularity, 
	$(p_{t}, q_{t}+1)$ satisfies the definition of neither connectors nor pre-connectors. 
	Thus, 
	$(p_{t}, q_{t}+1)$ is a dominator and
	we have 
	\begin{equation} \label{LMA:DR:eq.2}
		d_{t+1} = d_{t} + 1. 
	\end{equation}
	First, 
	we consider the case in which 
	the routine executes Case~3.3 for the $t$-th execution, 
	that is, 
	$q_{t} = q'_{t}+1$ 
	(i.e., $q_{t+1} = q'_{t}+2$) 
	and 
	$q_{t+1} = q_{t}+1 \leq n-1$ 
	(Fig.~\ref{fig:ldr}(5),(14)).
	In this case, 
	it is sufficient that we prove (a) is true 
	because $q_{t}+1 \leq n-1$. 
	$(p_{t}, q_{t}+1)$ dominates the five vertices $(p_{t}, q_{t}),(p_{t}, q_{t}+1),(p_{t}-1, q_{t}+1),(p_{t}+1, q_{t}+1)$ and $(p_{t}, q_{t}+2)$. 
	By the definition of $(p'_{t}, q'_{t})$-$(p_{t}, q_{t})$-regularity, 
	vertices in $R(D_{t})$ do not dominate none of $(p_{t}-1, q_{t}+1),(p_{t}+1, q_{t}+1)$ and $(p_{t}, q_{t}+2)$. 
	Thus, 
	the number of vertices which vertices in $R(D_{t+1})$ dominate increases by three vertices. 
	Since (a) is true for the $t$-th execution by the induction hypothesis, 
	dominators and a pre-connector in $D_{t}$ dominate $3 d_{t} +1$ vertices. 
	Thus, 
	we have using Eq.~(\ref{LMA:DR:eq.2}), 
	\[
		3 d_{t} + 1 + 3 = 3(d_{t} + 1) + 1 = 3 d_{t+1} + 1, 
	\]
	which implies that 
	(a)	is also true for $j = t+1$ 
	if the routine executes Case~3.3 for the $t$-th execution. 
	Next, 
	we consider the case in which 
	the routine executes Case~2 for the $t$-th and $t+2$-nd executions, 
	that is, 
	$q_{t} > q'_{t}+1$, 
	($q_{t+1} > q'_{t}+2$)
	and 
	$q_{t+1} = q_{t}+1 \leq n-1$
	(Fig.~\ref{fig:ldr}(6),(15)). 
	In this case, 
	it suffices to prove that (a) is true 
	because $q_{t}+1 \leq n-1$. 
	$(p_{t}, q_{t}+1)$ newly dominates the three vertices $(p_{t}-1, q_{t}+1),(p_{t}+1, q_{t}+1)$ and $(p_{t}, q_{t}+2)$, and 
	vertices in $R(D_{t})$ do not dominate them. 
	Thus, 
	the number of vertices dominated by vertices in $R(D_{t+1})$ increase by three vertices and by Eq.~(\ref{LMA:DR:eq.2}), 
	\[
		3 d_{t} + 1 + 3 = 3(d_{t} + 1) + 1 = 3 d_{t+1} + 1. 
	\]
	(a)	is true for $j = t+1$
	if the routine executes Case~2 for the $t$-th and $t+2$-nd executions. 
	Finally, 
	we consider the case in which 
	the routine executes Case~2 for the $t$-th execution and 
	does not execute Case~2 for the $t+2$-nd execution, 
	that is, 
	$q_{t+1} = q_{t}+1 = n$
	(Fig.~\ref{fig:ldr}(7),(16)). 
	$(p_{t}, q_{t}+1)$ newly dominates the two vertices $(p_{t}-1, n)$ and $(p_{t}+1, n)$, and 
	vertices in $R(D_{t})$ do not dominate them
	by the $(p'_{t}, q'_{t})$-$(p_{t}, q_{t})$-regularity. 
	Thus, 
	the number of vertices dominated by vertices in $R(D_{t+1})$ increases by two vertices. 
	Using~Eq.~(\ref{LMA:DR:eq.2}), 
	we have 
	\[
		3 d_{t} + 1 + 2 = 3(d_{t} + 1) = 3 d_{t+1}, 
	\]
	which implies 
	(b) is also true in this case when $j = t+1$. 
	\noindent
	{\bf Case~3.1:} 
	If the routine executes Case~3.1, 
	then a vertex which is not in $D_{t}$ but in $D_{t+1}$ is only $(2, 3)$, 
	which is a pre-connector, and dominators do not change
	(Fig.~\ref{fig:ldr}(4)). 
	Thus, 
	\begin{equation} \label{LMA:DR:eq.3}
		d_{t+1} = d_{t}. 
	\end{equation}
	Also, 
	$(2, 3)$ dominates $(2, 4)$, 
	which is not dominated by vertices in $R(D_{t})$. 
	Since (b) is true for the $t$-th execution by the induction hypothesis, 
	dominators in $D_{t}$ dominate $3 d_{t}$ vertices. 
	Together with Eq.~(\ref{LMA:DR:eq.3}), 
	we have 
	\[
		3 d_{t} + 1 = 3 d_{t+1} + 1, 
	\]
	which implies that 
	(a)	is true for $j = t+1$ in this case. 
	\noindent
	{\bf Case~3.2:} 
	We can prove this case similarly to the previous case
	(Fig.~\ref{fig:ldr}(9)). 
	If the routine executes Case~3.2, 
	then a vertex which is not in $D_{t}$ but in $D_{t+1}$ is only $(p_{t}, q_{t})$, 
	which is a pre-connector, and dominators do not change. 
	Thus, 
	\begin{equation} \label{LMA:DR:eq.4}
		d_{t+1} = d_{t}. 
	\end{equation}
	Also, 
	$(p_{t}, q_{t})$ dominates $(p_{t}+1, q_{t})$, 
	which is not dominated by vertices in $R(D_{t})$. 
	Since (b) is true for the $t$-th execution by the induction hypothesis, 
	dominators in $D_{t}$ dominate $3 d_{t}$ vertices. 
	Together with Eq.~(\ref{LMA:DR:eq.4}), 
	we have 
	\[
		3 d_{t} + 1 = 3 d_{t+1} + 1, 
	\]
	which implies that 
	(a)	is true for $j = t+1$ in this case. 

	\noindent
	{\bf Case~3.3:} 
	We omit the proof of this case 
	because we prove it similarly to that of Case~3.2
	(Fig.~\ref{fig:ldr}(13)). 
	\noindent
	{\bf Case~3.4:} 
	The routine does not do anything in Case~3.4, and $D_{t+1} = D_{t}$ hold. 
	Hence, 
	(a)	and (b) hold by the induction hypothesis. 
	We have shown that the both properties hold for the $t+1$-st execution. 
	\fi
\end{proof}

\ifnum \count10 > 0
%
%

%
\fi
\ifnum \count11 > 0
%
%

%
\fi
%

%
\begin{LMA}\label{LMA:PQ}
	\ifnum \count10 > 0
	%
	%
	\fi
	\ifnum \count11 > 0
	%
	%
	Suppose that the routine is feasible. 
	Then, 
	the following properties hold: 
	\begin{itemize}
	\itemsep=-2.0pt
	\setlength{\leftskip}{0pt}
		\item[(i)]
			For any $i \in [1, t_{1}-1]$, 
			$p'_{i} = 0$
			and 
			$q'_{i} = 2$.
		\item[(ii)]
			For any $i \geq t_{1}$, 
			$(p'_{i} \mod 3) = 2$
			and 
			$(q'_{i} \mod 3) = 2$. 
	%
	\end{itemize}
	\fi
\end{LMA}
\begin{proof}
	\ifnum \count10 > 0
	%
	%
	\fi
	\ifnum \count11 > 0
	%
	%
	By definition, 
	$p'_{0} = 0, q'_{0} = 2, p_{0} = 1$ and $q_{0} = 2$
	(Fig.~\ref{fig:ldr}(1)). 
	We can prove the following by induction on the number $i$ of times that the routine is executed:  
	For each $i = 1, \ldots, t_{1}-1$, 
	$q'_{i-1} = q_{i-1}$
	and 
	$p_{i} \leq m-1$. 
	Also, 
	the routine executes Case~1 for the $i$-th execution, 
	$p'_{i} = p'_{0} = 0$, 
	$p_{i} = p_{i-1} + 1$ and 
	$q'_{i} = q_{i} = q'_{0} = 2$
	(Fig.~\ref{fig:ldr}(2),(3)). 
	Thus, 
	(i) in the statement of this lemma holds. 
	Since $p'_{t_{1}-1} = p'_{0} = 0$, 
	$p_{t_{1}-1} = m$
	and 
	$q'_{t_{1}-1} = q_{t_{1}-1} = q'_{0} = 2$, 
	the routine executes Case~3.1 for the $t_{1}$-the execution
	(Fig.~\ref{fig:ldr}(4)). 
	Hence, 
	$p'_{t_{1}} = p_{t_{1}} = 2$, 
	$q'_{t_{1}} = 2$
	and
	$q_{t_{1}} = 3$. 
	We can prove the following by induction on the number $i$ of times that the routine is executed:  
	For each $i = t_{1} + 1, \ldots, t_{2}-1$, 
	the routine executes Case~2, 
	$p'_{i} = p_{i} = 2$, 
	$q'_{i} = 2$
	and 
	$q_{i} = q_{i-1} + 1$
	(Fig.~\ref{fig:ldr}(5),(6),(7)). 
	Since $p'_{t_{2}-1} = p_{t_{2}-1} = 2$, 
	$q'_{t_{2}-1} = 2$
	and
	$q_{t_{2}-1} = n$, 
	the routine executes either Case~3.2 or 3.3 for the $t_{2}$-th execution
	(Fig.~\ref{fig:ldr}(8)). 
	First, 
	let us consider the case in which 
	the routine executes Case~3.2 for the $t_{j} \hspace{1mm} (j = 2, \ldots, k-1)$-th execution. 
	Then, 
	$p'_{t_{j}} = p'_{t_{j}-1}$, 
	$p_{t_{j}} = p'_{t_{j}-1} + 1$
	and 
	$q'_{t_{j}} = q_{t_{j}} = q'_{t_{j}-1} + 3$
	(Fig.~\ref{fig:ldr}(9)). 
	We can prove by induction on $i$ the following: 
	the routine executes Case~1 for the $i = t_{j} + 1, \ldots, t_{j+1}-2$-th execution, 
	$p'_{i} = p'_{t_{j}}$, 
	$p_{i} = p_{i-1} + 1$
	and 
	$q'_{i} = q_{i} = q'_{t_{j}}$
	(Fig.~\ref{fig:ldr}(10),(11)).
	Moreover, 
	the routine executes Case~1 for the $t_{j+1}-1$-th execution, 
	$p'_{t_{j+1}-1} = p'_{t_{j}}$, 
	$p_{t_{j+1}-1} = m$
	and 
	$q'_{t_{j+1}-1} = q_{t_{j+1}-1} = q'_{t_{j}}$
	(Fig.~\ref{fig:ldr}(12)).
	Next, 
	we consider the case in which 
	the routine executes Case~3.3 for the $t_{j} \hspace{1mm} (j = 2, \ldots, k-1)$-th execution. 
	Then, 
	$q'_{t_{j}} = q'_{t_{j}-1}$, 
	$q_{t_{j}} = q'_{t_{j}-1} + 1$
	and 
	$p'_{t_{j}} = p_{t_{j}} = p'_{t_{j}-1} + 3$
	(Fig.~\ref{fig:ldr}(13)). 
	We can prove by induction on $i$ the following: 
	the routine executes Case~2 for the $i = t_{j} + 1, \ldots, t_{j+1}-2$-th execution, 
	$q'_{i} = q'_{t_{j}}$, 
	$q_{i} = q_{i-1} + 1$
	and 
	$p'_{i} = p_{i} = p'_{t_{j}}$
	(Fig.~\ref{fig:ldr}(14),(15)).
	Moreover, 
	the routine executes Case~2 for the $t_{j+1}-1$-th execution, 
	$q'_{t_{j+1}-1} = q'_{t_{j}}$, 
	$q_{t_{j+1}-1} = n$
	and 
	$p'_{t_{j+1}-1} = p_{t_{j+1}-1} = p'_{t_{j}}$
	(Fig.~\ref{fig:ldr}(16)). 
	By the above argument, 
	$p'_{t_{1}} = 2$
	and 
	$q'_{t_{1}} = 2$. 
	Furthermore, 
	for any $i \geq t_{1} + 1$, 
	if the routine executes Case~3.2, 
	then the value of $q'_{i}$ increases by three. 
	Otherwise, 
	its value does not change. 
	Similarly, 
	if the routine executes Case~3.3, 	
	then the value of $p'_{i}$ increases by three, and otherwise, 
	its value does not change.
	Therefore, 
	(ii) in the statement holds, 
	which completes the proof. 
	\fi
\end{proof}
\ifnum \count10 > 0
%
%

%
\fi
\ifnum \count11 > 0
%
%

%
\fi
%

%
\begin{LMA}\label{LMA:CR0}
	\ifnum \count10 > 0
	%
	%
	\fi
	\ifnum \count11 > 0
	%
	%
	Suppose that the routine is feasible. 	
	Then, 
	for any connector in $D_{\tau}$, 
	vertices dominated by the connector are dominated by at least one dominator in $D_{\tau}$. 
	\fi
\end{LMA}
\begin{proof}
	\ifnum \count10 > 0
	%
	%
	\fi
	\ifnum \count11 > 0
	%
	%
	Suppose that the routine is feasible. 
	We prove the statement of this lemma by induction on the number $j$ of times that the routine is executed. 
	When $j = 0$, 
	that is, 
	at the initialization of the routine, 
	$D_{0}$ is $(0, 2)$-$(1, 2)$-regular. 
	At this point, 
	only $(2, 1)$ is regular and $D_{0}$ does not contain a connector. 
	Thus, 
	the statement is true
	(Fig.~\ref{fig:ldr}(1)). 
	We assume that 
	the statement is true for the $j = t$-th execution of the routine and 
	show that it is also true for the $t+1$-st execution. 
	We will discuss each Case executed by the routine for the $t+1$-st execution. 
	Note that only irregular vertices are handled at the regularization, 
	and regular vertices are not removed from $D_{t}$ when the routine constructs $D_{t+1}$. 
	\noindent
	{\bf Case~1:} 
	First, 
	we consider the case in which the routine executes Case~1 for the $t+1$-st execution. 
	In this case, 
	$D_{t+1}$ is $(p'_{t}, q'_{t})$-$(p_{t}+1, q_{t})$-regular. 
	If either $t = 0$ or 
	$p_{t} > p'_{t}+1$, 
	then connectors in $D_{t+1}$ are equal to those in $D_{t}$ 
	and thus, 
	the statement is true by the induction hypothesis. 
	Second, 
	we consider the case in which 
	$t \ne 0$
	and 
	$p_{t} = p'_{t}+1$, 
	that is, 
	$p_{t+1} = p'_{t}+2$
	(Fig.~\ref{fig:ldr}(10)). 
	By definition, 
	$(p_{t}, q_{t})$ is a pre-connector in $D_{t}$ and is a connector in $D_{t+1}$. 
	$(p_{t}, q_{t})$ dominates the five vertices $(p_{t}, q_{t}), (p_{t}-1, q_{t}), (p_{t}, q_{t}-1), (p_{t}, q_{t}+1)$ and $(p_{t}+1, q_{t})$. 
	Also, 
	$(p_{t}, q_{t})$ is a connector and 
	the four vertices $(p_{t}+1, q_{t}), (p_{t}-1, q_{t}), (p_{t}-1, q_{t}-1)$ and $(p_{t}-1, q_{t}+1)$ belong to $D_{t+1}$. 
	The four vertices dominate all the vertices which $(p_{t}, q_{t})$ dominates. 
	Thus, 
	it suffices to show that 
	the four vertices are dominators. 
	Since $(p_{t}, q_{t}-1) \notin R(D_{t+1})$ by (Q2) in the regularity condition, 
	$(p_{t}+1, q_{t})$ satisfies the condition of neither a connector nor a pre-connector, and 
	is neither of them. 
	Similarly, 
	$(p_{t}-1, q_{t}+1)$ ($(p_{t}-1, q_{t}), (p_{t}-1, q_{t}-1)$) is neither a connector nor a pre-connector 
	because $(p_{t}-2, q_{t}), (p_{t}-2, q_{t}-1), (p_{t}-2, q_{t}-2) \notin R(D_{t+1})$
	by the condition (Q7). 
	Thus, 
	vertices which $(p_{t}, q_{t})$ dominates are dominated by dominators. 
	The other connectors except for $(p_{t}, q_{t})$ do not change, 
	and thus, the statement is true by the induction hypothesis. 
	\noindent
	{\bf Case~2:} 
	We prove this case similarly to the proof of Case~1. 
	When the routine executes Case~2 for the $t+1$-st execution, 
	$D_{t+1}$ is $(p'_{t}, q'_{t})$-$(p_{t}, q_{t}+1)$-regular. 
	If $q_{t} > q'_{t}+1$, 
	then connectors in $D_{t+1}$ are equal to those in $D_{t}$ 
	and thus, 
	the statement is true by the induction hypothesis. 
	Second, 
	we consider the case in which 
	$q_{t} = q'_{t}+1$, 
	that is, 
	$q_{t+1} = q'_{t}+2$
	(Fig.~\ref{fig:ldr}(14)). 
	By definition, 
	$(p_{t}, q_{t})$ is a pre-connector in $D_{t}$ and is a connector in $D_{t+1}$. 
	$(p_{t}, q_{t})$ dominates the five vertices $(p_{t}, q_{t}), (p_{t}-1, q_{t}), (p_{t}, q_{t}-1), (p_{t}, q_{t}+1)$ and $(p_{t}+1, q_{t})$. 
	Also, 
	$(p_{t}, q_{t})$ is a connector and 
	the four vertices $(p_{t}+1, q_{t}), (p_{t}-1, q_{t}), (p_{t}-1, q_{t}-1)$ and $(p_{t}-1, q_{t}+1)$ belong to $D_{t+1}$. 
	The four vertices dominate all the vertices which $(p_{t}, q_{t})$ dominates. 
	Thus, 
	it suffices to show that 
	the four vertices are dominators. 
	Since $(p_{t}-1, q_{t}) \notin R(D_{t+1})$ by (P2) in the regularity condition, 
	$(p_{t}, q_{t}+1)$ satisfies the condition of neither a connector nor a pre-connector, and 
	is neither of them. 
	Similarly, 
	$(p_{t}+1, q_{t}-1)$ ($(p_{t}, q_{t}-1), (p_{t}-1, q_{t}-1)$) is neither a connector nor a pre-connector 
	because $(p_{t}, q_{t}-2), (p_{t}-1, q_{t}-2), (p_{t}-2, q_{t}-2) \notin R(D_{t+1})$ 
	by the condition (P7). 
	Thus, 
	vertices which $(p_{t}, q_{t})$ dominates are dominated by dominators. 
	The other connectors except for $(p_{t}, q_{t})$ do not change, 
	and thus, the statement is true by the induction hypothesis. 
	\noindent
	{\bf Case~3.1, 3.2 or 3.3:} 
	If the routine executes Case~3.1, 3.2 or 3.3, 
	then a regular vertex newly add into $D_{t+1}$ is only a pre-connector and dominators do not change (Fig.~\ref{fig:ldr}(4),(9),(13)). 
	Thus, 
	the statement is true by the induction hypothesis. 
	\noindent
	{\bf Case~3.4:} 
	If the routine executes Case~3.4, 
	then the routine does nothing and thus, 
	the statement is true by the induction hypothesis. 
	We have shown that the statement is also true for the $t+1$-st execution. 
	\fi
\end{proof}
\ifnum \count10 > 0
%
%
%
\fi
\ifnum \count11 > 0
%
%

%
\fi
%

%
\begin{LMA}\label{LMA:IR}
	\ifnum \count10 > 0
	%
	%
	\fi
	\ifnum \count11 > 0
	%
	%
	Suppose that the routine is feasible. 
	Then, 
	\[
		{a}_{m,n} = (m \mod 3) \cdot (n \mod 3)
	\]
	and 
	\[
		\bar{r}_{m,n} = 
			\begin{cases}
					3 & \hspace{10mm} (m \mod 3) \cdot (n \mod 3) = 4 \\
					2 & \hspace{10mm} (m \mod 3) \cdot (n \mod 3) = 2 \\
					1 & \hspace{10mm} (m \mod 3) \cdot (n \mod 3) = 1 \\
					0 & \hspace{10mm} (m \mod 3) \cdot (n \mod 3) = 0. 
			\end{cases}
	\]
	\fi
\end{LMA}
\begin{proof}
	\ifnum \count10 > 0
	%
	%
	\fi
	\ifnum \count11 > 0
	%
	%
	Suppose that the routine is feasible. 
	If the routine executes Case~3.4 and is finished, 
	then one of the following conditions holds at that time: 
	(a) 
	$p'_{\tau} = p_{\tau} = m-1$, 
	(b)
	$q'_{\tau} = q_{\tau} = n-1$ and 
	(c)
	both $m-3 \leq p'_{\tau} \leq m-2$
	and 
	$n-3 \leq q'_{\tau} \leq n-2$. 
	In what follows, 
	we evaluate the values of $a_{m,n}$ and $\bar{r}_{m,n}$ 
	for each of these three cases. 
	\noindent
	{\bf (a):}
	First, 
	we consider the case of $p'_{\tau} = p_{\tau} = m-1$. 
	It follows from (ii) in Lemma~\ref{LMA:PQ} 
	that $(p'_{\tau} \mod 3) = 2$, 
	that is, 
	$(p'_{\tau} + 1 \mod 3) = 0$
	(Fig.~\ref{fig:lir}). 
	Thus, 
	$(m \mod 3) = 0$ by the above equalities. 
	That is, 
	\[
		(m \mod 3) \cdot (n \mod 3) = 0. 
	\]
	On the other hand, 
	no irregular vertex exists when the routine is finished 
	because $D_{\tau}$ is $(m-1, q'_{\tau})$-$(m-1, n)$-regular. 
	Hence, 
	\[
		\bar{r}_{m,n} = 0. 
	\]
	Also, 
	since all the vertices are dominated by regular vertices in $D_{\tau}$, 
	we have 
	\[
		a_{m,n} = 0, 
	\]
	which implies that 
	\[
		a_{m,n} = (m \mod 3) \cdot (n \mod 3). 
	\]
	\noindent
	{\bf (b):}
	We can prove the case (b) similarly to the case (a). 
	It follows from (ii) in Lemma~\ref{LMA:PQ} 
	that $(q'_{\tau} \mod 3) = 2$, 
	that is, 
	$(q'_{\tau} + 1 \mod 3) = 0$. 
	Thus, 
	$(n \mod 3) = 0$ by the above equalities. 
	That is, 
	\[
		(m \mod 3) \cdot (n \mod 3) = 0. 
	\]
	On the other hand, 
	no irregular vertex exists when the routine is finished 
	because $D_{\tau}$ is $(p'_{\tau}, n-1)$-$(m, n-1)$-regular. 
	Hence, 
	\[
		\bar{r}_{m,n} = 0. 
	\]
	Also, 
	since all the vertices are dominated by regular vertices in $D_{\tau}$, 
	we have 
	\[
		a_{m,n} = 0, 
	\]
	which implies that 
	\[
		a_{m,n} = (m \mod 3) \cdot (n \mod 3). 
	\]
	\noindent
	{\bf (c):}
	Since $(p'_{\tau} \mod 3) = 2$ by (ii) in Lemma~\ref{LMA:PQ},  
	$(p'_{\tau} + 1 \mod 3) = 0$. 
	Hence, 
	if $p'_{\tau} = m-3$, 
	then 
	$(m-3 + 1 \mod 3) = (m-2 \mod 3) = 0$, 
	that is, 
	$(m \mod 3) = 2$. 
	If $p'_{\tau} = m-2$, 
	then
	$(m \mod 3) = 1$. 
	Also, 
	$(q'_{\tau} + 1 \mod 3) = 0$ by (ii) in Lemma~\ref{LMA:PQ}. 
	Hence, 
	if $q'_{\tau} = n-3$, 
	then 
	$(n \mod 3) = 2$. 
	If $q'_{\tau} = n-2$, 
	then
	$(n \mod 3) = 1$. 
	Let us call these facts Fact~(A). 
	In what follows, 
	we discuss each case for $m$ and $n$. 
	\noindent
	{\bf\boldmath (c-1) $p'_{\tau} = m-3$ and $q'_{\tau} = n-3$:}
	By Fact~(A), 
	\[
		(m \mod 3) \cdot (n \mod 3) = 2 \cdot 2 = 4. 
	\]
	Then, 
	$D_{\tau}$ is 
	$(m-3, n-3)$-$(m-3, n)$-regular or 
	$(m-3, n-3)$-$(m, n-3)$-regular. 
	Thus, 
	the irregular vertices in $D_{\tau}$ are 
	$
		(m-2, n-2), (m-2, n-1), (m-2, n), 
		(m-1, n-2), (m-1, n-1), (m-1, n), 
		(m, n-2), (m, n-1)$ and $(m, n)
	$ 
	by definition. 
	The vertices which are not dominated by regular vertices in $D_{\tau}$ are 
	the four vertices $(m-1, n-1), (m-1, n), (m, n-1)$ and $(m, n)$, 
	which implies that 
	\[
		a_{m,n} = 4. 
	\]
	Hence, 
	\[
		a_{m,n} = (m \mod 3) \cdot (n \mod 3). 
	\]
	Additionally, 
	$(m-3,n-3), 
	(m-3,n-2), (m-3,n-1), (m-3,n), 
	(m-2,n-3), (m-1,n-3), (m,n-3) 
	\in D_{\tau}$ 
	because $D_{\tau}$ is 
	$(m-3, n-3)$-$(m-3, n)$-regular or 
	$(m-3, n-3)$-$(m, n-3)$-regular. 
	We need three vertices to dominate 
	the vertices $(m-1, n-1), (m-1, n), (m, n-1)$ and $(m, n)$, 
	and hence 
	\[
		\bar{r}_{m,n} = 3. 
	\]
	\noindent
	{\bf\boldmath (c-2) $p'_{\tau} = m-3$ and $q'_{\tau} = n-2$:}
	By Fact~(A), 
	\[
		(m \mod 3) \cdot (n \mod 3) = 2 \cdot 1 = 2. 
	\]
	Then, 
	$D_{\tau}$ is 
	$(m-3, n-2)$-$(m-3, n)$-regular or 
	$(m-3, n-2)$-$(m, n-2)$-regular. 
	Thus, 
	the irregular vertices in $D_{\tau}$ are 
	$ 
		(m-2, n-1), (m-2, n), 
		(m-1, n-1), (m-1, n), 
		(m, n-1)$ and $(m, n)
	$ 
	by definition. 
	The vertices which are not dominated by regular vertices in $D_{\tau}$ are 
	the two vertices $(m-1, n)$ and $(m, n)$, 
	which implies that 
	\[
		a_{m,n} = 2. 
	\]
	Hence, 
	\[
		a_{m,n} = (m \mod 3) \cdot (n \mod 3). 
	\]
	Additionally, 
	$(m-3,n-2), 
	(m-3,n-1), (m-3,n), 
	(m-2,n-2), (m-1,n-2), (m,n-2) 
	\in D_{\tau}$ 
	because $D_{\tau}$ is 
	$D_{\tau}$ is 
	$(m-3, n-2)$-$(m-3, n)$-regular or 
	$(m-3, n-2)$-$(m, n-2)$-regular. 
	We need two vertices to dominate 
	the vertices $(m-1, n)$ and $(m, n)$, 
	and hence 
	\[
		\bar{r}_{m,n} = 2. 
	\]
	\noindent
	{\bf\boldmath (c-3) $p'_{\tau} = m-2$ and $q'_{\tau} = n-3$:}
	By Fact~(A), 
	\[
		(m \mod 3) \cdot (n \mod 3) = 1 \cdot 2 = 2. 
	\]
	Then, 
	$D_{\tau}$ is 
	$(m-2, n-3)$-$(m-2, n)$-regular or 
	$(m-2, n-3)$-$(m, n-3)$-regular. 
	Thus, 
	the irregular vertices in $D_{\tau}$ are 
	$ 
		(m-1, n-2), (m-1, n-1), (m-1, n), 
		(m, n-2), (m, n-1)$ and $(m, n)
	$ 
	by definition. 
	The vertices which are not dominated by regular vertices in $D_{\tau}$ are 
	the two vertices $(m, n-1)$ and $(m, n)$, 
	which implies that 
	\[
		a_{m,n} = 2. 
	\]
	Hence, 
	\[
		a_{m,n} = (m \mod 3) \cdot (n \mod 3). 
	\]
	Additionally, 
	$(m-2,n-3), 
	(m-2,n-2), (m-2,n-1), (m-2,n), 
	(m-2,n-3), (m-1,n-3), (m,n-3)
	\in D_{\tau}$ 
	because $D_{\tau}$ is 
	$D_{\tau}$ is 
	$(m-2, n-3)$-$(m-2, n)$-regular or 
	$(m-2, n-3)$-$(m, n-3)$-regular. 
	We need two vertices to dominate 
	the vertices $(m, n-1)$ and $(m, n)$, 
	and hence 
	\[
		\bar{r}_{m,n} = 2. 
	\]
	\noindent
	{\bf\boldmath (c-4) $p'_{\tau} = m-2$ and $q'_{\tau} = n-2$:}
	By Fact~(A), 
	\[
		(m \mod 3) \cdot (n \mod 3) = 1 \cdot 1 = 1. 
	\]
	Then, 
	$D_{\tau}$ is 
	$(m-2, n-2)$-$(m-2, n)$-regular or 
	$(m-2, n-2)$-$(m, n-2)$-regular. 
	Thus, 
	the irregular vertices in $D_{\tau}$ are 
	$
		(m-1, n-1), (m-1, n), 
		(m, n-1)$ and $(m, n)
	$ 
	by definition. 
	The vertex which are not dominated by regular vertices in $D_{\tau}$ is 
	the one vertex $(m, n)$, 
	which implies that 
	\[
		a_{m,n} = 1. 
	\]
	Hence, 
	\[
		a_{m,n} = (m \mod 3) \cdot (n \mod 3). 
	\]
	Additionally, 
	$(m-2,n-2), 
	(m-2,n-1), (m-2,n), 
	(m-1,n-2), (m,n-2) 
	\in D_{\tau}$ 
	because $D_{\tau}$ is 
	$D_{\tau}$ is 
	$(m-2, n-2)$-$(m-2, n)$-regular or 
	$(m-2, n-2)$-$(m, n-2)$-regular. 
	We need one vertex to dominate 
	the vertex $(m, n)$, 
	and hence 
	\[
		\bar{r}_{m,n} = 1. 
	\]
	\fi
\end{proof}
\ifnum \count12 > 0
\begin{figure*}
	 \begin{center}
	  \includegraphics[width=\linewidth]{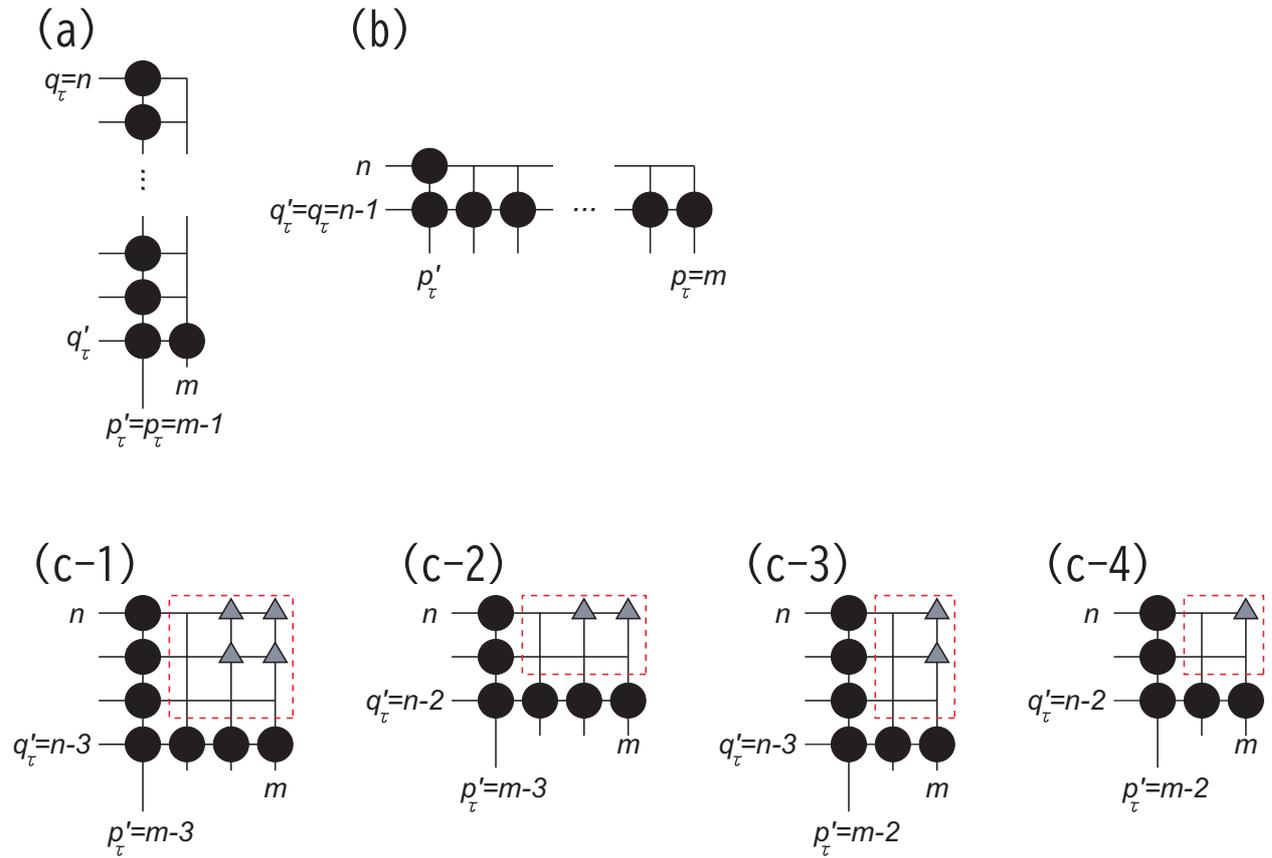}
	 \end{center}
	 \caption{
\ifnum \count10 > 0
\fi
\ifnum \count11 > 0
%
%
Vertices which are not dominated by regular vertices and irregular vertices when the routine is finished. 
Black circles denote vertices in the MCDS $D_{\tau}$. 
Vertices enclosed by dashed lines are irregular and the other vertices are regular. 
In (a) and (b), there do not exist irregular vertices. 
Gray triangles denote vertices which are not dominated by regular vertices in $D_{\tau}$. 
\fi
			}
	\label{fig:lir}
\end{figure*}
\fi
\ifnum \count10 > 0
%
%

%
\fi
\ifnum \count11 > 0
%
%

%
\fi
%
\begin{LMA}\label{LMA:CR1}
	\ifnum \count10 > 0
	\fi
	\ifnum \count11 > 0
	%
	%
	If the routine is feasible, then 
	\[
		{c}_{m,n} \geq 
			\begin{cases}
				\min\{ \frac{m}{3}, \frac{n}{3} \} & (m \mod 3) = 0   \mbox{ and } (n \mod 3) = 0 \\
				\frac{m}{3} & (m \mod 3) = 0   \mbox{ and } (n \mod 3) \ne 0 \\
				\frac{n}{3} & (m \mod 3) \ne 0 \mbox{ and } (n \mod 3) = 0 \\
				\lfloor \frac{m}{3} \rfloor + \lfloor \frac{n}{3} \rfloor - 1 & (m \mod 3) \ne 0 \mbox{ and } (n \mod 3) \ne 0. \\
			\end{cases}
	\]
	\fi
\end{LMA}
\begin{proof}
	\ifnum \count10 > 0
	%
	%
	\fi
	\ifnum \count11 > 0
	%
	%
	Suppose that the routine is feasible. 
	The routine executes Case~3.1, 3.2 or 3.3 for the $i = t_{j} \hspace{1mm} (j = 1, \ldots, k-1)$-th execution by definition. 
	If the routine executes Case~3.1, 
	then the vertex $(p'_{t_{1}}, q'_{t_{1}}+1) (= (2, 3))$ becomes a pre-connector
	(Fig.~\ref{fig:ldr}(4)). 
	If the routine executes Case~3.2, 
	then $(p'_{t_{j}}+1, q'_{t_{j}}+3)$ becomes a pre-connector
	(Fig.~\ref{fig:ldr}(9)). 
	If the routine executes Case~3.3, 
	then $(p'_{t_{j}}+3, q'_{t_{j}}+1)$ becomes a pre-connector
	(Fig.~\ref{fig:ldr}(13)). 
	Also, 
	the routine executes Case~1 or 2 for the $i = t_{j} + 1\hspace{1mm} (j = 1, \ldots, k-1)$-st execution. 
	If the routine executes Case~1 for $t_{j} + 1$-st, 
	then $(p_{t_{j}}, q_{t_{j}}) (= (p_{t_{j}+1}-1, q_{t_{j}+1}))$ which is a pre-connector becomes a connector. 
	If the routine executes Case~2, 
	$(p_{t_{j}}, q_{t_{j}}) (= (p_{t_{j}+1}, q_{t_{j}+1}-1))$ which is a pre-connector becomes a connector. 
	Hence, 
	the number of connectors at the time when the routine is finished is equal to the number of times that the routine executes Case~3.1, 3.2 or 3.3. 
	Let $c_{1}$ ($c_{2},c_{3}$, respectively) denote the number of times that the routine executes Case~3.1 (3.2, 3.3, respectively) 
	by the time when the routine is finished. 
	By definition,  
	\[
		c_{m,n} = c_{1} + c_{2} + c_{3}. 
	\]
	In what follows, 
	we show the inequalities in the statement by evaluating lower bounds on $c_{1},c_{2}$ and $c_{3}$. 
	The routine executes Case~1 for the $i \in [1, t_{1}-1]$-th execution. 
	Since we assume that $m \geq 4$ and $n \geq 4$ throughout this paper, 
	the routine executes Case~3.1 for the $t_{1}$-th execution. 
	Also, 
	the routine does not execute Case~3.1 after the $t_{1}$-th execution 
	because $(p'_{i} \mod 3) = 2$ for any $i > t_{1}$ by (ii) in Lemma~\ref{LMA:PQ}.
	Thus, 
	$c_{1} = 1$ holds. 
	At the time when the routine executes Case~3.4, 
	one of the following three conditions holds: 
	(a) 
	$p'_{\tau} = p_{\tau} = m-1$, 
	(b)
	$q'_{\tau} = q_{\tau} = n-1$ and 
	(c)
	both $m-3 \leq p'_{\tau} \leq m-2$
	and 
	$n-3 \leq q'_{\tau} \leq n-2$. 
	We evaluate lower bounds on $c_{2}$ and $c_{3}$ for each case of the three cases. 
	\noindent
	{\bf (a):}
	When the routine executes Case~3.1 for the $t_{1}$-th execution, 
	$p'_{t_{1}} = 2$ by the condition of the routine. 
	Using (ii) in Lemma~\ref{LMA:PQ}, 
	we have for any $i > t_{1}$, 
	$(p'_{i} \mod 3) = 2$. 
	Only if the routine executes Case~3.3, 
	the value of $p'_{i}$ increases by three. 
	For the execution of the other cases, 
	it does not change. 
	Moreover, 
	$p'_{\tau} = m-1$
	by the condition of (a). 
	Hence, 
	the number of times that the routine executes Case~3.3 is 
	\[
		c_{3} = \frac{m-1 - 2}{3} = \frac{m}{3} - 1.
	\]
	Therefore, 
	\begin{equation} \label{LMA:CR1:eq.1}
		c_{m,n} \geq c_{1} + c_{3} = 1 + \frac{m}{3} - 1 = \frac{m}{3}. 
	\end{equation}
	\noindent
	{\bf (b):}
	We can prove this case similarly to that of (a). 
	When the routine executes Case~3.1 for the $t_{1}$-th execution, 
	$q'_{t_{1}} = 2$ by the condition of the routine. 
	Using (ii) in Lemma~\ref{LMA:PQ}, 
	we have for any $i > t_{1}$, 
	Only if the routine executes Case~3.2, 
	the value of $q'_{i}$ increases by three. 
	For the execution of the other cases, 
	it does not change. 
	Moreover, 
	$q'_{\tau} = n-1$
	by the condition of (b). 
	Hence, 
	the number of times that the routine executes Case~3.2 is 
	\[
		c_{2} = \frac{n-1 - 2}{3} = \frac{n}{3} - 1. 
	\]
	Therefore, 
	\begin{equation} \label{LMA:CR1:eq.2}
		c_{m,n} \geq c_{1} + c_{2} = 1 + \frac{n}{3} - 1 = \frac{n}{3}. 
	\end{equation}
	\noindent
	{\bf (c):}
	If $p'_{\tau} = m-3$ by the condition of (c),
	we can evaluate the number of times that the routine executes Case~3.3 similarly to the case (a): 
	\[
		c_{3} = \frac{m-3 - 2}{3} = \frac{m-2}{3} - 1.
	\]
	Since $(p'_{\tau}-2 \mod 3) = 0$ by (ii) in Lemma~\ref{LMA:PQ}, 
	$\frac{m-2}{3} = \lfloor \frac{m}{3} \rfloor$, 
	which implies that 
	\begin{equation} \label{LMA:CR1:eq.3}
		c_{3} = \frac{m-2}{3} - 1 = \lfloor \frac{m}{3} \rfloor  -  1. 
	\end{equation}
	Similarly, 
	if $p'_{\tau} = m-2$, 
	then 
	\begin{equation} \label{LMA:CR1:eq.4}
		c_{3} = \frac{m-2 - 2}{3} - 1 = \lfloor \frac{m}{3} \rfloor  -  1.
	\end{equation}
	If $q'_{\tau} = n-3$ by the condition of (c),
	we can evaluate the number of times that the routine executes Case~3.2 
	similarly to the case (b) and 
	it is 
	\[
		c_{2} = \frac{n-3 - 2}{3} = \frac{n-2}{3} - 1.
	\]
	Since $(q'_{\tau}-2 \mod 3) = 0$ by (ii) in Lemma~\ref{LMA:PQ}, 
	\begin{equation} \label{LMA:CR1:eq.5}
		c_{2} = \frac{n-2}{3} - 1 = \lfloor \frac{n}{3} \rfloor  -  1.
	\end{equation}
	Similarly, 
	if $q'_{\tau} = n-2$, 
	then 
	\begin{equation} \label{LMA:CR1:eq.6}
		c_{2} = \frac{n-2 - 2}{3} - 1 = \lfloor \frac{n}{3} \rfloor  -  1.
	\end{equation}
	By Eqs.~(\ref{LMA:CR1:eq.3}), (\ref{LMA:CR1:eq.4}), (\ref{LMA:CR1:eq.5}) and (\ref{LMA:CR1:eq.6}), 
	\begin{equation} \label{LMA:CR1:eq.7}
		c_{m,n} = c_{1} + c_{2} + c_{3} 
			= 1 + \lfloor \frac{m}{3} \rfloor  -  1 + \lfloor \frac{n}{3} \rfloor  -  1 
			= \lfloor \frac{m}{3} \rfloor + \lfloor \frac{n}{3} \rfloor - 1. 
	\end{equation}
	By the above argument, 
	if $(m \mod 3) = 0$ and $(n \mod 3) = 0$, 
	then $(m-1 \mod 3) = 2$ and $(n-1 \mod 3) = 2$. 
	Also, 
	either $p'_{\tau} = m-1$ or 
	$q'_{\tau} = n-1$ 
	because $(p'_{\tau} \mod 3) = 2$ and $(q'_{\tau} \mod 3) = 2$ by (ii) in Lemma~\ref{LMA:PQ}. 	
	Thus,  
	either (a) or (b) is true
	Then, 
	by Eqs.~(\ref{LMA:CR1:eq.1}) and (\ref{LMA:CR1:eq.2}), 
	\[
		c_{m,n} \geq \min \left \{ \frac{m}{3}, \frac{n}{3} \right \}.
	\]
	In the case in which both $(m \mod 3) = 0$ and $(n \mod 3) \ne 0$, 
	$p'_{\tau} = m-1$ holds, 
	which satisfies the condition of (a). 
	Thus, 
	by Eq.~(\ref{LMA:CR1:eq.1}), 
	\[
		c_{m,n} \geq \frac{m}{3}.
	\]
	In the case in which both $(m \mod 3) \ne 0$ and $(n \mod 3) = 0$, 
	$q'_{\tau} = n-1$ holds, 
	which satisfies the condition of (b). 
	Thus, 
	by Eq.~(\ref{LMA:CR1:eq.2}), 
	\[
		c_{m,n} \geq \frac{n}{3}.
	\]
	The case in which $(m \mod 3) \ne 0$ and $(n \mod 3) \ne 0$ 
	satisfies the condition of (c). 
	Thus, 
	by Eq.~(\ref{LMA:CR1:eq.7}), 
	\[
		c_{m,n} = \lfloor \frac{m}{3} \rfloor + \lfloor \frac{n}{3} \rfloor - 1.
	\]
	\fi
\end{proof}
\ifnum \count10 > 0
%
%

%
\fi
\ifnum \count11 > 0
%
%

%
\fi
%

%
\subsection{Routine Feasibility} \label{sec:routinefeasibility}
\ifnum \count10 > 0
\fi
\ifnum \count11 > 0
%
%
In this section, 
we show that the routine is feasible and can obtain the MCDS $D_{\tau}$, 
of which properties we have shown in the previous section. 
Specifically, 
we will complete the proof of the lower bound lemma by showing that 
the routine can conduct regularization in each of Cases~1, 2, 3.1, 3.2 and 3.3. 
We give a definition for the following lemmas. 
For a CDS $D$, 
a vertex $(x, y) \in {\overline{R}}(D)$ is called a {\em mobile in $D$} 
if the vertex satisfies at least one of the following five conditions 
(see Fig.~\ref{fig:mv}): 
\begin{itemize}
	\itemsep=-2.0pt
	\setlength{\leftskip}{0pt}
	\item[(i)]
		$x \in [2, m-1]$, 
		$y \in [2, n-1]$
		and 
		$(x+1, y), (x-1, y), (x-1, y-1), (x-1, y+1) \in D$, 
	\item[(ii)]
		$x \in [2, m-1]$ 
		and 
		$(x+1, n), (x-1, n), (x-1, n-1) \in D$, 
	\item[(iii)]
		$x \in [2, m-1]$, 
		$y \in [2, n-1]$
		and 
		$(x, y+1), (x, y-1), (x-1, y-1), (x+1, y-1) \in D$, 
	\item[(iv)]
		$y \in [2, n-1]$
		and 
		$(m, y+1), (m, y-1), (m-1, y-1) \in D$, 
		and 
	\item[(v)]
		$(1, 4), (1, 2), (2, 2) \in D$. 
\end{itemize}
For a CDS $D$, 
when a vertex set $C$ is constructed by removing a mobile $v$ in $D$,  
$C$ may not be connected, but still remains dominating. 
Note that 
if we obtain a new vertex set by adding another vertex into $C$ instead of $v$ to be connected, 
then the new vertex set becomes a CDS different from $D$. 
\fi
\ifnum \count12 > 0
\begin{figure*}
	 \begin{center}
	  \includegraphics[width=\linewidth]{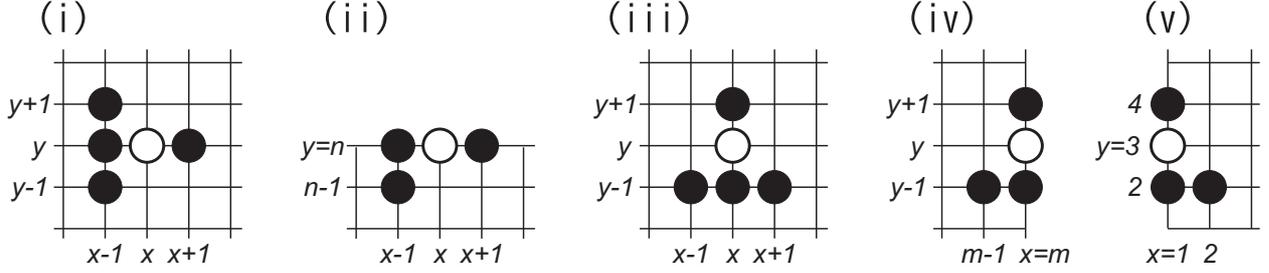}
	 \end{center}
	 \caption{
\ifnum \count10 > 0
%
%
%
%
\fi
\ifnum \count11 > 0
%
%
Circles denote vertices in a CDS in the conditions of mobiles. 
White circles denote mobile vertices. 
\fi
			}
	\label{fig:mv}
\end{figure*}
\fi
%

%
\begin{LMA}\label{LMA:L2}
	\ifnum \count10 > 0
	%
	%
	\fi
	\ifnum \count11 > 0
	%
	%
	Suppose that $D$ is a $(p', q')$-$(p, q)$-regular CDS. 
	Also, 
	suppose that 
	for vertices $\overline{v} \in {\overline{R}}(D)$ and $u \in {R}(D)$, 
	there exists a simple path $P$ consisting of vertices in $D$ between $u$ and $\overline{v}$
	such that 
	(i) 
	the number of vertices in ${\overline{R}}(D)$ in $P$ is minimum, 
	(ii)
	$P$ does not contain $(p, q)$, 
	(iii)
	if $q' = q$, 
	then $P$ does not contain $(p, q+1)$ and 
	(iv)
	if $p' = p$, 
	then $P$ does not contain $(p+1, q)$. 
	Then, 
	$P$ contains a mobile in $D$. 
	\fi
\end{LMA}
\begin{proof}
	\ifnum \count10 > 0
	%
	%
	\fi
	\ifnum \count11 > 0
	%
	%
	Let $D$ be a $(p', q')$-$(p, q)$-regular CDS. 
	Let $\overline{v} \in {\overline{R}}(D)$ and $u \in {R}(D)$ be vertices  
	such that 
	there exists a simple path $P$ consisting of vertices in $D$ between $u$ and $\overline{v}$ satisfying the four conditions in the statement of this lemma. 
	Since $\overline{v} \in {\overline{R}}(D)$ and $u \in {R}(D)$, 
	vertices 
	$\overline{v}' \in {\overline{R}}(D)$
	and 
	$u' \in {R}(D)$ 
	are contained in $P$
	such that $\overline{v}'$ and $u'$ are adjacent. 
	Suppose that $u' = (x, y)$. 
	Let us consider the case in which $q = q'$. 
	Since $u'$ is adjacent to $\overline{v}'$, which is irregular, 
	there exist three cases with respect to the position of $u'$
	(Fig.~\ref{fig:l2}): 
	(a) 
	both $y = q$ and $x \geq p'+1$, 
	(b) 
	both $y = q-3$ and $x \geq p+1$, 
	and 
	(c) 
	$y \geq q+1$, $x = p'$ and $p' \geq 2$. 
	We will discuss these three cases. 
	\noindent
	{\bf\boldmath (a) $y = q$ and $x \geq p'+1$:}
	Since the path $P$ does not contain $(p, q)$ by the definition of $P$, 
	$x \ne p$. 
	Thus, 
	$x \leq p-1$. 
	Let us discuss irregular vertices adjacent to $(x, q)$. 
	By the condition (Q2) of the regularity, 
	$(x, q-1) \notin R(D)$. 
	If $x \geq p'+2$, 
	then $(x-1, q) \in R(D)$
	by the condition (Q1). 
	If $x = p'+1$
	and 
	$q \geq 4$, 
	then $(x-1, q) \in R(D)$ by the condition (Q6). 
	If $x = p'+1$
	and 
	$q \leq 3$, 
	then $q=2$ and $p'=0$ 
	by (i) in Lemma~\ref{LMA:PQ}, 
	which implies that 
	$(x, q) (= (1, 2))$ does not have an adjacent vertex on the left side. 
	Also, 
	if $x + 1 \leq m$, 
	$(x+1, q) \in R(D)$ 
	by the condition (Q1). 
	Hence, 
	an irregular vertex which can be adjacent to $(x, q)$ is located at only $(x, q+1)$. 
	Then, 
	$(x, q+2) \in D$ 
	holds. 
	The reason is as follows: 
	Assume that 
	$(x, q+2) \notin D$. 
	Then, 
	since $P$ does not contain $(p, q+1)$ by definition, 
	there exists $j \in [x+1, p-1]$ such that 
	$(j, q+2) \in D$, and 
	$(x+1, q+1), \ldots, (j, q+1) \in D$ or 
	there exists $j' \leq x-1$ such that 
	$(j', q+2) \in D$, and 
	$(j', q+1), \ldots, (x-1, q+1) \in D$. 
	However, 
	since $P$ is selected such that 
	the number of irregular vertices in $P$ is minimized, 
	either 
	$(j, q) = u'$ 
	(i.e., $j = x$) 
	or 
	$(j', q) = u'$ 
	(i.e., $j' = x$), 
	which contradicts the definitions of $j$ and $j'$. 
	Thus, 
	$(x, q+2) \in D$ and, 
	hence, 
	$(x, q+1) = \overline{v}'$ 
	and 
	$\overline{v}' \in {\overline{R}}(D)$. 
	Then, 
	if $x = 1$, 
	then $\overline{v}'$ satisfies the definition (v) of a mobile, 
	and if $x = m$, 
	it satisfies the definition (iv). 
	Otherwise, 
	it satisfies the definition (iii). 
	\noindent
	{\bf\boldmath (b) $y = q-3$ and $x \geq p+1$:}
	Similarly to the case in which $y = q$, 
	we can show a mobile $\overline{v}' \in {\overline{R}}(D)$. 
	$(x, q-4) \notin R(D)$ 
	by the condition (Q8). 
	Since $x \geq p+1$ by the condition of (b) and $p \geq p'+1$, 
	$x \geq p'+2$. 
	Hence, 	
	$(x-1, q-3) \in R(D)$
	by the condition (Q5). 
	Also, 
	if $x + 1 \leq m$, 
	$(x+1, q-3) \in R(D)$. 
	Hence, 
	an irregular vertex which can be adjacent to $(x, q-3)$ is located at only $(x, q-2)$. 
	Similarly to the proof of the case (a), 
	since $P$ is selected such that 
	the number of irregular vertices in $P$ is minimized, 
	$(x, q-1) \in D$. 
	Hence, 
	$(x, q-2) = \overline{v}'$
	and 
	$\overline{v}' \in {\overline{R}}(D)$. 
	Then, 
	if $x = m$, 
	then $\overline{v}'$ satisfies the definition (iv) of a mobile. 
	Otherwise, 
	it satisfies the definition (iii). 
	\noindent
	{\bf\boldmath (c) $x = p'$ and $y \geq q+1$:}
	Similarly to the above cases, 
	we can show a mobile $\overline{v}' \in {\overline{R}}(D)$. 
	Specifically, 
	$(p'-1, y) \notin R(D)$ by the condition (Q7)
	and 
	$(p', y+1), (p', y-1) \in R(D)$ by the condition (Q6). 
	Thus, 
	an irregular vertex which can be adjacent to $(p', y)$ is located at only $(p'+1, y)$. 
	Similarly to the proof of the case (a), 
	$(p'+2, y) \in D$. 
	Hence, 
	$(p'+1, y) = \overline{v}'$
	and 
	$\overline{v}' \in {\overline{R}}(D)$. 
	Then, 
	if $y = m$, 
	$\overline{v}'$ satisfies the definition (ii) of a mobile 
	and otherwise, 
	it satisfies the definition (i). 
	Since we can prove the statement for the case in which $p = p'$ similarly to $q = q'$,
	we omit the proof. 
	\fi
\end{proof}
\ifnum \count12 > 0
\begin{figure*}
	 \begin{center}
	  \includegraphics[width=\linewidth]{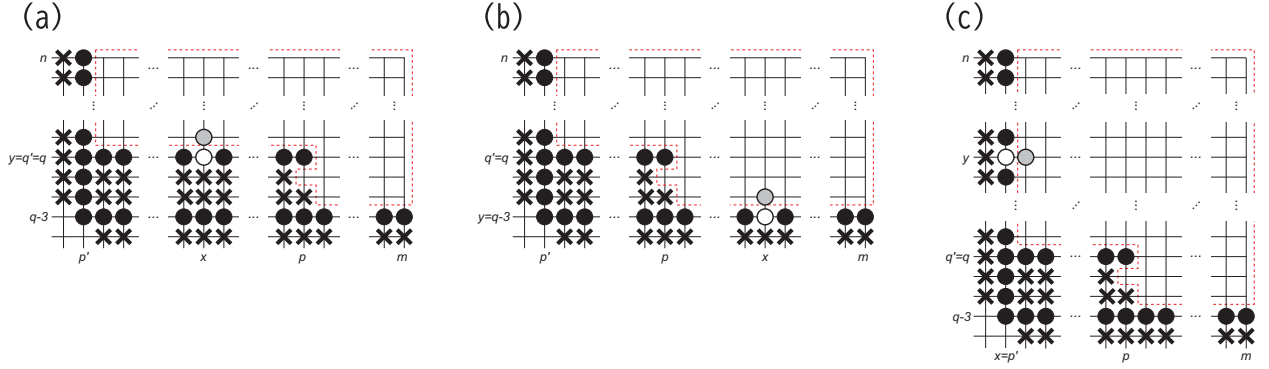}
	 \end{center}
	 \caption{
\ifnum \count10 > 0
\fi
\ifnum \count11 > 0
%
%
Figure for Lemma~\ref{LMA:L2}. 
Suppose that $D$ is a $(p', q')$-$(p, q)$-regular MCDS. 
The figures show that when $q'=q$, 
white (gray, respectively) circles denote the positions of the vertex $u' = (x, y)$ ($\overline{v}'$, respectively). 
Black circles denote regular vertices in $D$ and 
vertices enclosed by dashed lines are irregular. 
\fi
			}
	\label{fig:l2}
\end{figure*}
\fi
\ifnum \count10 > 0
%
%

%
\fi
\ifnum \count11 > 0
%
%

%
\fi
%

%
\begin{LMA}\label{LMA:MV}
	\ifnum \count10 > 0
	%
	%
	\fi
	\ifnum \count11 > 0
	%
	%
	Suppose that $D$ is a $(p', q')$-$(p, q)$-regular MCDS. 
	Then, 
	the following properties hold: 
	\begin{itemize}
	\itemsep=-2.0pt
	\setlength{\leftskip}{0pt}
		\item[(i)]
		Suppose that either $q'=q$ and $p=m$ 
		or 
		$p'=p$ and $q=n$. 
		Also, 
		suppose that 
		$(x, y) \in R(D)$
		and 
		$(x+2, y) \in {\overline{R}}(D)$. 
		Then, 
		an MCDS $D' = D \backslash \{ v \} \cup \{ (x+1, y) \}$ exists, 
		in which 
		$v$ is a mobile in $D$
		(see Fig.~\ref{fig:lmv}).
		\item[(ii)]
		Suppose that either $q'=q$ and $p=m$ 
		or 
		$p'=p$ and $q=n$. 
		Also, 
		suppose that 
		$(x, y) \in R(D)$
		and 
		$(x, y+2) \in {\overline{R}}(D)$. 
		Then, 
		an MCDS $D' = D \backslash \{ v \} \cup \{ (x, y+1) \}$ exists, 
		in which 
		$v$ is a mobile in $D$. 
		\item[(iii)]
		Suppose that $q'=q$ and $p \leq m-2$. 
		Also, 
		suppose that $P$ is a simple path consisting of $D$ between $(p-1, q)$ and $(p+2, q)$
		such that 
		the number of vertices in ${\overline{R}}(D)$ in $P$ is minimum and 
		$P$ contains neither $(p, q)$ nor $(p, q+1)$. 
		Then, 
		an MCDS $D' = D \backslash \{ v \} \cup \{ (p+1, q) \}$ exists, 
		in which 
		$v$ is a mobile in $D$. 
		\item[(iv)]
		Suppose that $p'=p$ and $q \leq n-2$. 
		Also, 
		suppose that $P$ is a simple path consisting of $D$ between $(p, q-1)$ and $(p, q+2)$
		such that 
		the number of vertices in ${\overline{R}}(D)$ in $P$ is minimum and 
		$P$ contains neither $(p, q)$ nor $(p+1, q)$. 
		Then, 
		an MCDS $D' = D \backslash \{ v \} \cup \{ (p, q+1) \}$ exists, 
		in which 
		$v$ is a mobile in $D$. 
	\end{itemize}
	\fi
\end{LMA}
\begin{proof}
	\ifnum \count10 > 0
	%
	%
	\fi
	\ifnum \count11 > 0
	%
	%
	Suppose that $D$ is a $(p', q')$-$(p, q)$-regular MCDS. 
	\noindent
	{\bf (i):}
	First, 
	we consider the case of (i). 
	Suppose that
	either 
	$q'=q$ and $p=m$ or 
	$p'=p$ and $q=n$. 
	Also, 
	suppose that $(x, y) \in R(D)$
	and 
	$(x+2, y) \in {\overline{R}}(D)$. 
	By applying Lemma~\ref{LMA:L2} with $(x+2, y)$ and $(x, y)$ as $\overline{v}$ and $u$ in its statement, respectively, 
	there exists a mobile $\overline{v}' \in {\overline{R}}(D)$ in a path $P$, 
	in which 
	$P$ is a simple path of vertices in $D$ between $(x, y)$ and $(x+2, y)$ 
	such that 
	the number of vertices in ${\overline{R}}(D)$ in $P$ is minimum. 
	Let us define
	$D' = D \backslash \{ \overline{v}' \} \cup \{ (x+1, y) \}$. 
	Since $\overline{v}'$ is a mobile, 	
	$D \backslash \{ \overline{v}' \}$ is dominating. 
	Then, 
	in what follows, 
	we will show that $D'$ is connected. 
	That is, 
	we will show that for any two vertices $a, b \in D'$, 
	there exists a path consisting of $D'$ between $a$ and $b$. 
	Let $D_{1}$ be a set of vertices $\tilde{v}$ 
	such that there exists a path consisting of vertices in $D \backslash \{ \overline{v}' \}$ between $v$ and $(x, y)$. 
	Also, 
	let $D_{2}$ be the vertex set of vertices in $D$ except for those in $D_{1} \cup \{ \overline{v}' \}$. 
	That is, 
	$D_{2} = D \backslash (D_{1} \cup \{ \overline{v}' \})$. 
	There exists a path consisting of vertices in $D_{1}$ between $(x, y)$ and any vertex in $D_{1}$ by definition 
	(called Fact~(a)). 
	That is, 
	for any two vertices in $D_{1}$, 
	there exists a path of vertices in $D_{1}$ between them. 
	Also, 
	suppose that $\overline{v}' = (x', y')$ and 
	$\overline{v}'$ satisfies the condition (i) of a mobile, 
	that is, 
	$(x'+1, y'), (x'-1, y'), (x'-1, y'-1), (x'-1, y'+1) \in D$. 
	We will show that for any two vertices in $D_{2}$, 
	there exists a path of vertices in $D_{2}$ between them 
	by showing the number of subgraphs induced by $D \backslash \{ \overline{v}' \}$ is two. 
	If we can guarantee the existence of such a path, 
	then we have the following: 
	Since $D'$ contains $(x+1, y)$, 
	there exists a path of vertices in $D'$ between $(x, y)$ and $(x+2, y)$. 
	Hence, 
	for any $a \in D_{1}$ and any $b \in D_{2}$, 
	there exists a path of vertices in $D'$ between $a$ and $b$. 
	If both $(x', y'+1) \notin D$ and $(x', y'-1) \notin D$, 
	then the number of subgraphs induced by $D \backslash \{ \overline{v}' \}$ is two. 
	Then, 
	in the case in which $(x', y'+1) \in D$ ($(x', y'-1) \in D$, respectively), 
	we will show that $(x', y'+1)$ ($(x', y'-1)$, respectively) is contained in either $D_{1}$ or $D_{2}$
	(called Property~(b)). 
	Let us consider the case of $(x', y'+1) \in D$. 
	$(x', y'+1)$ is adjacent to $(x'-1, y'+1)$, 
	that is, 
	the path $(x', y'+1)(x'-1, y'+1)(x'-1,y')$ consists of three vertices in $D \backslash \{ \overline{v}' \}$. 
	Thus, 
	if $(x'-1, y') \in D_{1}$, then $(x', y'+1) \in D_{1}$ and 
	if $(x'-1, y') \in D_{2}$, then $(x', y'+1) \in D_{2}$. 
	Similarly, 
	in the case of $(x', y'-1) \in D$, 
	if $(x'-1, y') \in D_{1}$, $(x', y'-1) \in D_{1}$, 
	and 
	if $(x'-1, y') \in D_{2}$, $(x', y'-1) \in D_{2}$. 
	Thus, 
	we have shown that Property~(b) is true, 
	which implies that 
	we have shown that $D'$ is an MCDS.
	In the case in which $\overline{v}'$ satisfies the other conditions except for (i), 
	we can also prove that $D'$ is connected. 
	\noindent
	{\bf (ii):}
	We can prove (ii) similarly to the proof of (i). 
	Suppose that
	either 
	$q'=q$ and $p=m$ or 
	$p'=p$ and $q=n$. 
	Also, 
	suppose that $(x, y) \in R(D)$
	and 
	$(x, y+2) \in {\overline{R}}(D)$. 
	By applying Lemma~\ref{LMA:L2} with $(x, y+2)$ and $(x, y)$ as $\overline{v}$ and $u$, respectively, 
	there exists a mobile $\overline{v}' \in {\overline{R}}(D)$ in a path $P$, 
	in which 
	$P$ is a simple path of vertices in $D$ between $(x, y)$ and $(x, y+2)$ 
	such that 
	the number of vertices in ${\overline{R}}(D)$ in $P$ is minimum. 
	Let us define 
	$D' = D \backslash \{ \overline{v}' \} \cup \{ (x, y+1) \}$. 
	Since $\overline{v}'$ is a mobile, 	
	$D \backslash \{ \overline{v}' \}$ is dominating. 
	We omit the rest of the proof 
	because we can show that $D'$ is connected 
	similarly to the proof of (i). 
	\noindent
	{\bf (iii):}
	We can also prove (iii) similarly to the proof of (i). 
	Suppose that 
	$q'=q$ and $p \leq m-2$. 
	Also, 
	suppose that $P$ is a simple path between $(p-1, q)$ and $(p+2, q)$ which satisfies the conditions in statement (iii) of this lemma. 
	By the definition of the regularity, 
	$(p-1, q) \in R(D)$ 
	and 
	$(p+2, q) \in {\overline{R}}(D)$. 
	By applying Lemma~\ref{LMA:L2} with $(p+2, q)$ and $(p-1, q)$ as $\overline{v}$ and $u$, respectively, 
	there exists a mobile $\overline{v}' \in {\overline{R}}(D)$ in the path $P$.  
	Let us define 
	$D' = D \backslash \{ \overline{v}' \} \cup \{ (p+1, q) \}$. 
	Since $\overline{v}'$ is a mobile, 	
	$D \backslash \{ \overline{v}' \}$ is dominating. 
	We omit the rest of the proof 
	because we can show that $D'$ is connected 
	similarly to the proof of (i). 
	\noindent
	{\bf (iv):}
	We can also prove (iv) similarly to the proof of (i). 
	Suppose that 
	$p'=p$ and $q \leq n-2$. 
	Also, 
	suppose that $P$ is a simple path between $(p, q-1)$ and $(p, q+2)$ which satisfies the conditions in statement (iv). 
	By the definition of the regularity, 
	$(p, q-1) \in R(D)$ 
	and 
	$(p, q+2) \in {\overline{R}}(D)$. 
	By applying Lemma~\ref{LMA:L2} with $(p, q+2)$ and $(p, q-1)$ as $\overline{v}$ and $u$, respectively, 
	there exists a mobile $\overline{v}' \in {\overline{R}}(D)$ in the path $P$.  
	Let us define 
	$D' = D \backslash \{ \overline{v}' \} \cup \{ (p, q+1) \}$. 
	Since $\overline{v}'$ is a mobile, 	
	$D \backslash \{ \overline{v}' \}$ is dominating. 
	We omit the rest of the proof 
	because we can show that $D'$ is connected 
	similarly to the proof of (i). 
	\fi
\end{proof}
\ifnum \count12 > 0
\begin{figure*}
	 \begin{center}
	  \includegraphics[width=\linewidth]{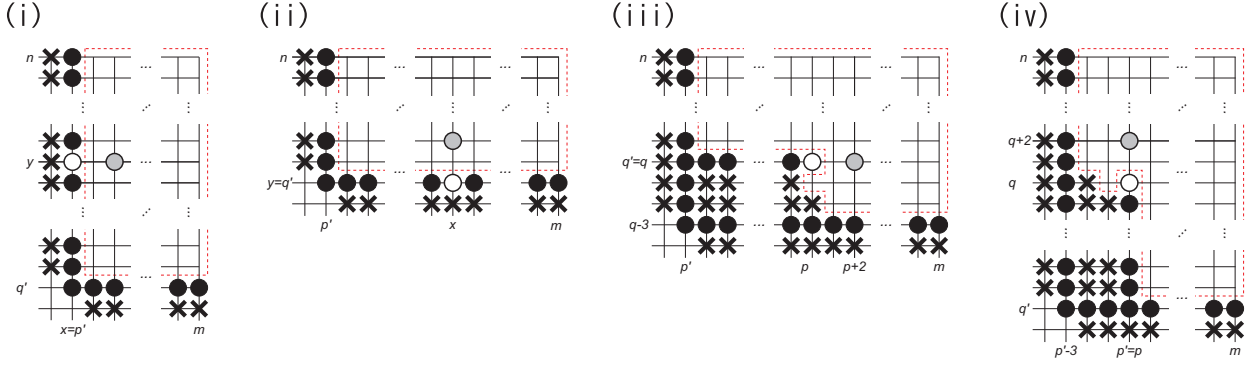}
	 \end{center}
	 \caption{
\ifnum \count10 > 0
\fi
\ifnum \count11 > 0
%
%
Suppose that $D$ is a $(p', q')$-$(p, q)$-regular MCDS. 
Circles and crosses denote vertices in $D$ and not in $D$, respectively. 
Lemma~\ref{LMA:L2} shows that for a path of vertices in $D$ between a white circle and a gray circle, 
there exists a mobile in the path. 
Note that by the regularity, 
if $q' = q$, then $y = q'$, 
and if $p' = p$, then $x = p'$. 
\fi
			}
	\label{fig:lmv}
\end{figure*}
\fi
%

%
\begin{LMA}\label{LMA:LT}
	\ifnum \count10 > 0
	%
	%
	\fi
	\ifnum \count11 > 0
	%
	%
	Suppose that a vertex set $D$ is an MCDS which is either $(p, q)$-$(m, q)$-regular or $(p, q)$-$(p, n)$-regular. 
	Then, 
	the following properties are true: 
	\begin{itemize}
	\itemsep=-2.0pt
	\setlength{\leftskip}{0pt}
		\item[(i)]
		If for some $p \geq 2$ and some $q \geq 2$, 
		$(p+2, q+1) \in D$
		and 
		$(p+3, q+1) \notin D$, 
		then 	
		there exists an MCDS $D'$ into which $D$ is $(p+3, q)$-$(p+3, q+1)$-regularized
		(see Fig.~\ref{fig:lt}).
		\item[(ii)]
		If for some $p \geq 2$ and some $q \geq 2$, 
		$(p+1, q+2) \in D$
		and 
		$(p+1, q+3) \notin D$, 
		then 	
		there exists an MCDS $D'$ into which $D$ is $(p, q+3)$-$(p+1, q+3)$-regularized. 
		\item[(iii)]
		If $p = 0$, 
		$q = 2$, 
		$(1, 3) \in D$
		and 
		$(2, 3) \notin D$, 
		then 	
		there exists an MCDS $D'$ into which $D$ is $(2, 2)$-$(2, 3)$-regularized. 
	\end{itemize}
	\fi
\end{LMA}
\begin{proof}
	\ifnum \count10 > 0
	%
	%
	\fi
	\ifnum \count11 > 0
	%
	%
	First, 
	we consider the cases (i) and (ii). 
	Suppose that a vertex set $D$ is an MCDS which is either $(p, q)$-$(m, q)$-regular or $(p, q)$-$(p, n)$-regular. 
	We will construct a CDS $\hat{D}$ into which $D$ is regularized. 
	Note that since $D$ is an MCDS, 
	$\hat{D}$ is also an MCDS. 
	If $D$ is $(p, q)$-$(m, q)$-regular, 
	then $(p, q), \ldots, (m, q) \in D$ by the conditions (Q1) and (Q6) of the regularity. 
	If $D$ is $(p, q)$-$(p, n)$-regular,
	then $(p, q), \ldots, (m, q) \in D$ by the condition (P6).  
	These vertices are all regular and 
	are not removed from $D$ to construct $\hat{D}$. 
	Thus, 
	$\hat{D}$ satisfies the conditions (Q5) to be $(p, q+3)$-$(p+1, q+3)$-regular 
	and (P6) to be $(p+3, q)$-$(p+3, q+1)$-regular. 
	Similarly, 
	if $D$ is $(p, q)$-$(m, q)$-regular, 
	then $(p, q), \ldots, (p, n) \in D$ by the condition (Q6). 
	If $D$ is $(p, q)$-$(p, n)$-regular, 
	$(p, q), \ldots, (p, n) \in D$ by the conditions (P1) and (P6). 
	Thus, 
	$\hat{D}$ satisfies the conditions (Q6) to be $(p, q+3)$-$(p+1, q+3)$-regular 
	and (P5) to be $(p+3, q)$-$(p+3, q+1)$-regular. 
	Also, 
	since $(p+1, q), (p+2, q), (p,q+1), (p, q+2) \in D$, 
	these vertices dominate all the vertices which $(p+1, q+1)$ can dominate and 
	$(p+1, q+1) \notin D$. 
	Thus, 
	$\hat{D}$ satisfies the conditions (Q4) and (P4). 
	Since $(p+1, q-1), \ldots, (m, q-1) \notin D$ by the conditions (Q2) and (P7), 
	$\hat{D}$ satisfies the conditions (Q8) to be $(p, q+3)$-$(p+1, q+3)$-regular 
	and (P7) to be $(p+3, q)$-$(p+3, q+1)$-regular. 
	Since $(p-1, q+1), \ldots, (p-1, n) \notin D$ by the conditions (Q7) and (P2), 
	$\hat{D}$ satisfies the conditions (Q7) to be $(p, q+3)$-$(p+1, q+3)$-regular 
	and (P8) to be $(p+3, q)$-$(p+3, q+1)$-regular. 
	By the above argument, 
	$\hat{D}$ satisfies all the conditions except for (Q1) ((P8), respectively) to be $(p, q+3)$-$(p+1, q+3)$-regular ($(p+3, q)$-$(p+3, q+1)$-regular, respectively). 
	Then, 
	in order to satisfy these two conditions, 
	we will show that 
	if $(p+2, q+1) \in D$, 
	then we can obtain $\hat{D}$ such that $(p+3, q+1) \in \hat{D}$, 
	and 
	if $(p+1, q+2) \in D$, 
	we can obtain $\hat{D}$ such that $(p+1, q+3) \in \hat{D}$
	by regularizing $D$. 
	Moreover, 
	we will show that 
	$\hat{D}$ is dominating and connected. 
	We can show that this lemma is true by showing them. 
	We first consider the case (i), that is, 
	the case in which 
	$(p+2, q+1) \in D$
	and 
	$(p+3, q+1) \notin D$. 
	In this case, 
	one of the following two cases holds 
	(see Fig.~\ref{fig:lt}):
	\begin{itemize}
	\itemsep=-2.0pt
	\setlength{\leftskip}{0pt}
		\item[(a)]
	For some $y \in [q+1, n]$, 
	$(p+2, q+1), \cdots, (p+2, y) \in D$
	and 
	$(p+3, q+1), \cdots, (p+3, y) \notin D$. 
	Also, 
	If $y+1 \leq n$, 
	$(p+2, y+1), (p+3, y+1) \in D$. 
		\item[(b)]
	For some $y' \in [q+1, n-1]$, 
	$(p+2, q+1), \cdots, (p+2, y') \in D$, 
	$(p+3, q+1), \cdots, (p+3, y') \notin D$
	and 
	$(p+2, y'+1) \notin D$. 
	\end{itemize}
	We will investigate each case. 
	\noindent
	{\bf (a): }
	For each $i = q+1, \ldots, y$, 
	we construct a vertex set $D'$ by removing $(p+2, i)$ from $D$ and adding $(p+3, i)$ to $D'$. 
	That is, 
	$D' = D \backslash 
		\{ (p+2, j) \mid j \in [q+1, y] \} \cup 
		\{ (p+3, j) \mid j \in [q+1, y] \}$. 
	We show that $D'$ is dominating. 
	The vertices dominated by $(p+2, i)$ are 
	$(p+1, i), (p+2, i), (p+3, i), (p+2, i-1)$ and $(p+2, i+1)$, 
	and 
	we confirm that these five vertices are also dominated by vertices in $D'$. 
	If $i \in [q+2, y-1]$, 
	then $(p+2, i), (p+3, i), (p+2, i-1)$ and $(p+2, i+1)$
	are dominated by 
	the newly added vertices 
	$(p+3, i), (p+3, i), (p+3, i-1)$ and $(p+3, i+1)$, respectively. 
	If $i = q+1$, 
	then $(p+2, q+1), (p+3, q+1)$ and $(p+2, q+2)$ 
	are dominated by 
	the new vertices $(p+3, q+1), (p+3, q+1)$ and $(p+3, q+2)$, respectively. 
	Since $D$ is 
	$(p, q)$-$(m, q)$-regular or 
	$(p, q)$-$(p, n)$-regular 
	by the supposition of this lemma, 
	$(p+2, q) \in D$ 
	by the condition (Q1) to be $(p, q)$-$(m, q)$-regular and 
	the condition (P6) to be $(p, q)$-$(p, n)$-regular, 
	which results in $(p+2, q) \in D'$. 
	If $i = y$
	and 
	$y+1 \leq n$, 
	then 
	$(p+2, y+1) \in D'$ 
	because $(p+2, y+1) \in  D$
	by the condition of the case~(a). 
	If $i = y$ and 
	$y = n$, 
	$(p+2, n+1)$ does not exist. 
	Also, 
	since 
	$p \geq 2$ by the condition of (a)
	and 
	$D$ is $(p, q)$-$(m, q)$-regular or 
	$(p, q)$-$(p, n)$-regular by the supposition of the lemma, 
	$(p, i) \in D$ 
	by the conditions (Q6) to be $(p, q)$-$(m, q)$-regular and 
	(P1) to be $(p, q)$-$(p, n)$-regular. 
	Thus, 
	$(p, i) \in D'$
	and 
	$(p, i)$ dominates $(p+1, i)$. 
	Next, 
	we show that $D'$ is connected. 
	$(p+2, q), (p+3, q), (p+2, q), (p+3, y+1) \in D'$ 
	because $(p+2, q), (p+3, q), (p+2, q), (p+3, y+1) \in D$. 
	For any two vertices $a, b \in D$, 
	suppose that there exists a path $P$ between $a$ and $b$ consisting of vertices in $D$ 
	such that $P$ contains some vertex $(p+2, i') \hspace{1mm} (i' \in [q+1, y])$. 
	For any $i'' \in [q+1, y+1]$, 
	$D$ does not contain $(p+1, i'')$
	because of the minimality of $D$. 
	That is, 
	$P$ necessarily contains either $(p+2, y+1)$ or $(p+3, y+1)$. 
	Hence, 
	there exists a path of vertices in $D'$ between $a$ and $b$, 
	which implies that 
	$D'$ is connected. 
	Therefore, 
	$(p+3, q+1) \in D'$
	and 
	$D'$ is a $(p+3, q)$-$(p+3, q+1)$-regular MCDS. 
	\noindent
	{\bf (b): }
	For each $i = q+1, \ldots, y'$, 
	we construct a vertex set $D'$ by removing $(p+2, i)$ from $D$ and adding $(p+3, i+1)$. 
	That is, 
	$D' = D \backslash 
		\{ (p+2, j)   \mid j \in [q+1, y'] \} \cup 
		\{ (p+3, j+1) \mid j \in [q+1, y'] \}$. 
	In what follows, 
	we will show that $D'$ is an MCDS. 
	After that, 
	by applying (ii) in Lemma~\ref{LMA:MV} with   
	$(p+3, q) \in R(D')$ and $(p+3, q+2) \in {\overline{R}}(D')$
	as $(x, y)$ and $(x, y+2)$, respectively, 
	there exists an MCDS $D''$ such that $D'' = D' \backslash \{ v' \} \cup \{ (p+3, q+1) \}$, 
	in which $v' \in {\overline{R}}(D')$ is a mobile. 
	Since $(p+3, q+1) \in D''$, 
	$D''$ is $(p+3, q)$-$(p+3, q+1)$-regular. 
	First, 
	we show that $D'$ is dominating. 
	For each $i = q+1, \ldots, y'$, 
	the vertices dominated by $(p+2, i)$ are 
	$(p+1, i), (p+2, i), (p+3, i), (p+2, i-1)$ and $(p+2, i+1)$. 
	If $i \ne q+1$, 
	then $(p+2, i), (p+3, i), (p+2, i-1)$ and $(p+2, i+1)$ are 
	dominated by 
	$(p+3, i), (p+3, i), (p+3, i-1)$ and $(p+3, i+1)$, respectively. 
	Also, 
	if $i = q+1$, 
	then 
	$(p+2, q+2)$ is dominated by 
	$(p+3, q+2)$. 
	Since $D$ is 
	$(p, q)$-$(m, q)$-regular or 
	$(p, q)$-$(p, n)$-regular 
	by the supposition of this lemma, 
	$(p+2, q), (p+3, q) \in D$ by the conditions (Q1) and (P6). 
	Hence, 
	$(p+2, q), (p+3, q) \in D'$, 
	which dominate $(p+2, q+1), (p+3, q+1)$ and $(p+2, q)$. 
	Next, 
	we show that $D'$ is connected. 
	If $y' \leq n-2$, 
	then either $(p+4, y'+1) \in D$ 
	or 
	$(p+3, y'+2) \in D$ 
	because there exists a vertex to dominate $(p+3, y'+1)$ and 
	$D$ is connected. 
	If $y' = n-1$, 
	then $(p+4, y'+1) \in D$ for the same reason. 
	Hence, 
	either $(p+4, y'+1) \in D'$ 
	or $(p+3, y'+2) \in D'$. 
	We now define 
	$A = \{ (p+3, j+1) \mid j \in [q+1, y'] \}$
	and 
	$B = D' \backslash A$. 
	That is, 
	$B = D \backslash \{ (p+2, j) \mid j \in [q+1, y'] \}$. 
	In what follows, 
	we show that for any two vertices $a, b \in D'$, 
	there exists a path consisting of vertices in $D'$ between $a$ and $b$. 
	If $a, b \in B$, 
	then such a path exists because $D$ is connected. 
	If $a, b \in A$, 
	such a path exists by the definition of $A$. 
	If $a \in A$ and $b \in B$, 
	then there exists a path of vertices in $D'$
	between $a$ and either $(p+4, y'+1)$ or $(p+3, y'+2)$. 
	Since the vertices $(p+4, y'+1)$ and $(p+3, y'+2)$ belong to $B$, 
	there exists a path of vertices in $D'$ between $b$ and either $(p+4, y'+1)$ or $(p+3, y'+2)$. 
	Hence, 
	there also exists such a path between $a$ and $b$ in this case. 
	By the above argument, 
	there exists such a path of vertices in $D'$  between $a$ and $b$. 
	That is, $D'$ is connected. 
	Therefore, 
	$D'$ is an MCDS and 
	we have shown that $D''$ is a $(p+3, q)$-$(p+3, q+1)$-regular MCDS by the initial argument. 
	The case of (ii), 
	that is, 
	the case in which 
	$(p+1, q+2) \in D$ 
	and 
	$(p+1, q+3) \notin D$ is symmetric to the case (i) and thus, 
	we omit the proof of this case. 
	Since the case~(iii) is a special case of the case~(i), 
	we can prove this case in the same way as that of (i) and omit the proof. 
	\fi
\end{proof}
\ifnum \count12 > 0
\begin{figure*}
	 \begin{center}
	  \includegraphics[width=\linewidth]{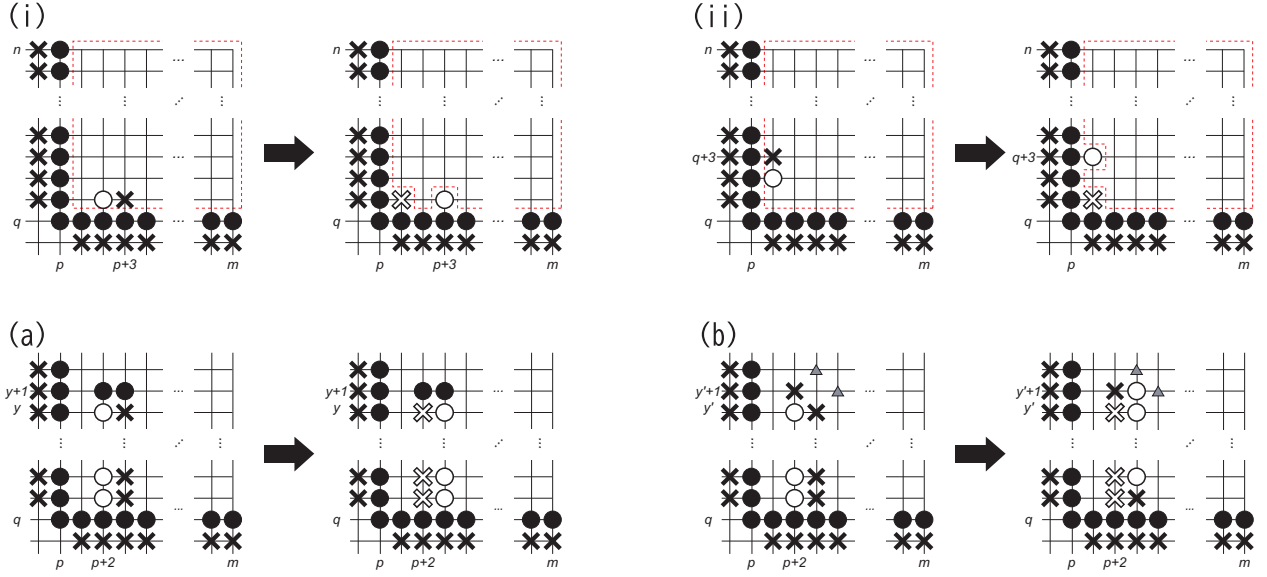}
	 \end{center}
	 \caption{
\ifnum \count10 > 0
\fi
\ifnum \count11 > 0
%
%
Figure for Lemma~\ref{LMA:LT}. 
Suppose that $D$ is an MCDS which is either $(p, q)$-$(m, q)$-regular or $(p, q)$-$(p, n)$-regular. 
Circles and crosses denote vertices in $D$ and not in $D$, respectively. 
\fi
			}
	\label{fig:lt}
\end{figure*}
\fi
\ifnum \count10 > 0
%
%

%
\fi
\ifnum \count11 > 0
%
%

%
\fi
%
%
\begin{LMA}\label{LMA:L4}
	\ifnum \count10 > 0
	\fi
	\ifnum \count11 > 0
	%
	%
	Suppose that a vertex set $D$ is a $(p', q')$-$(p, q)$-regular MCDS. 
	Then, 
	the following properties hold: 
	\begin{itemize}
		\itemsep=-2.0pt
		\setlength{\leftskip}{0pt}
		\item[(i)]
		If $q' = q$
		and 
		$p \leq m-1$, 
		there exists an MCDS $D'$ into which $D$ is $(p', q')$-$(p+1, q)$-regularized. 
		\item[(ii)]
		If $p' = p$ 
		and 
		$q \leq n-1$, 
		there exists an MCDS $D'$ into which $D$ is $(p', q')$-$(p, q+1)$-regularized. 
	\end{itemize}
	\fi
\end{LMA}
\begin{proof}
	\ifnum \count10 > 0
	\fi
	\ifnum \count11 > 0
	%
	%
	Suppose that a vertex set $D$ is a $(p', q')$-$(p, q)$-regular MCDS. 
	First, we consider the case (i) in which 
	$q' = q$
	and 
	$q \leq n-2$. 
	We will construct a vertex set $\hat{D}$ by regularizing $D$. 
	Note that $\hat{D}$ is also an MCDS 
	because $D$ is an MCDS. 
	When $(p', q')$-$(p+1, q)$-regularizing $D$, 
	the conditions (Q5),(Q6),(Q7) and (Q8) of the regularity of $D$ are not changed.
	Thus, 
	$\hat{D}$ also satisfies the conditions (Q5),(Q6),(Q7) and (Q8) to be $(p', q')$-$(p+1, q)$-regular. 
	Also, 
	$(p'+1, q), \ldots, (p, q) \in D$ ($(p'+1, q-1), \ldots, (p-1, q-1) \notin D$, $(p'+1, q-2), \ldots, (p, q-2) \notin D$, respectively)
	by the condition (Q1) ((Q2), (Q4),  respectively) of the regularity of $D$. 
	Thus, 
	$(p'+1, q), \ldots, (p, q) \in \hat{D}$, 
	$(p'+1, q-1), \ldots, (p-1, q-1) \notin \hat{D}$
	and 
	$(p'+1, q-3), \ldots, (p, q-1) \notin \hat{D}$. 
	In what follows, 
	we will construct $\hat{D}$ by regularizing $D$ 
	such that 
	$(p+1, q) \in \hat{D}$, 
	$(p, q-1) \notin \hat{D}$ and 
	$(p+1, q-2) \notin \hat{D}$ 
	in order that 
	$\hat{D}$ satisfies the conditions (Q1),(Q2) and (Q4). 
	Also, 
	we will show that 
	if $p + 1 = m$, 
	then $(m, q-1) \notin \hat{D}$ 
	to satisfy the condition (Q3)
	and show that 
	if $q = n-1$, 
	then $(p, q) \notin \hat{D}$ 
	to satisfy the condition (Q9). 
	Moreover, 
	we will show that 
	$\hat{D}$ is dominating and connected. 
	We can prove (i) in the statement of this lemma by showing them. 
	We first consider the case in which $(p+1, q) \in D$ 
	(see Case~1 in Fig.~\ref{fig:routine}). 
	If $q' = q \leq 2$, 
	then $q' = q = 2$ by (i) in Lemma~\ref{LMA:PQ}. 
	Since the condition (Q1) is satisfied, 
	if $p \geq 2$, 
	then $(p-1, 2), (p, 2), (p+1, 2) \in D$, 
	and 
	if $p = 1$, 
	then $(1, 2), (2, 2) \in D$.
	Thus, 
	all the vertices which $(p, 1)$ can dominate are dominated by other vertices in $D$. 
	$D$ does not contain $(p, q-1) (= (p, 1))$ 
	because $D$ is minimum. 
	If $q' = q > 2$, 
	then $q' \geq 5$ by (ii) in Lemma~\ref{LMA:PQ}. 
	Then, 
	if $p' \geq 2$ and $p \geq p'+2$, 
	then $(p-1, q), (p, q), (p+1, q) \in D$ 
	by the condition (Q1). 
	If $p' \geq 2$ and $p = p'+2$, 
	then $(p-1, q), (p, q), (p+1, q) \in D$ as well  
	by the conditions (Q1) and (Q6). 
	It follows from the condition (Q1) that 
	if $p' \leq 1$, 
	then $(p-1, q), (p, q), (p+1, q) \in D$ 
	and 
	if $p = 1$, 
	then $(p, q)(=(1, q)), (p+1, q)(=(2, q)) \in D$. 
	Moreover, 
	if $q' \geq 5$, 
	$(p, q-3) \in D$ by the condition (Q5). 
	Thus, 
	all the vertices which $(p, q-1)$ can dominate are dominated by other vertices in $D$. 
	$(p, q-1) \notin D$ 
	because $D$ is minimum. 
	By the above argument, 
	$(p, q-1) \notin \hat{D}$, 
	which implies that 
	$\hat{D}$ satisfies the condition (Q2). 
	Similarly, 
	$(p, q-3),(p+1, q-3),(p+2, q-3) \in D$ by the condition (Q5) 
	and 
	$(p+1, q) \in D$ by the above supposition. 
	Thus, 
	$(p, q-2) \notin D$ 
	because $D$ is minimum. 
	$\hat{D}$ satisfies the condition (Q4). 
	Therefore, 
	if $(p+1, q) \in D$, 
	$D$ satisfies the conditions (Q1),(Q2) and (Q4)  to be $(p', q')$-$(p+1, q)$-regular. 
	If $p = m-1$, 
	then $(m, q-1) \notin D$ 
	because $D$ satisfies the condition (Q1) and is minimum.  
	That is, 
	$D$ satisfies the condition (Q3). 
	If $q = n-1$, 
	then $(p'+1, n), \cdots, (p-1, n) \notin D$ 
	because $D$ satisfies the condition (Q1) and is minimum.  
	That is, 
	$D$ satisfies the condition (Q9). 
	Therefore, 
	regarding $D$ as $\hat{D}$ has completed the $(p', q')$-$(p+1, q)$-regularization. 
	Next, we consider the case in which 
	$(p+1, q) \notin D$. 
	If $p + 2 \leq m$ and $(p+2, q) \in D$, 
	let $P$ be a simple path consisting of $D$ between 
	$(p+2, q)$ and $(p-1, q)$ 
	such that 
	the number of vertices in ${\overline{R}}(D)$ in $P$ is minimum. 
	We will discuss the following three cases: 
	\begin{itemize}
	\itemsep=-2.0pt
	\setlength{\leftskip}{0pt}
		\item[(a)]
			$p+2 \leq m$, 
			$(p+2, q) \in D$ and 
			$P$ contains neither $(p, q)$ nor $(p, q+1)$, 
		\item[(b)]
			$p+2 \leq m$, $(p+2, q) \in D$
			and 
			$P$ contains at least of $(p, q)$ and $(p, q+1)$, and 
		\item[(c)]
			either $p+1 = m$ or 
			both $p + 2 \leq m$ and $(p+2, q) \notin D$. 
	\end{itemize}
	We will show that for each of (a),(b) and (c), 
	$D$ can be regularized into some MCDS which contains $(p+1, q)$. 
	Then, 
	we can show that the MCDS is $(p', q')$-$(p+1, q)$-regular similarly to the proof of the above discussed case, 
	that is, the case of $(p+1, q) \in D$. 
	\noindent
	{\bf (a):}
	By applying Lemma~\ref{LMA:MV}~(iii) 
	with $(p+2, q) \in {\overline{R}}(D)$ and $(p-1, q) \in R(D)$ 
	as $(p-1, q)$ and $(p+2, q)$, respectively, 
	MCDS $D' = D \backslash \{ v \} \cup \{ (p+1, q) \}$, in which 
	$v \in {\overline{R}}(D)$ is a mobile. 
	Thus, 
	$D'$ satisfies the statement of this lemma. 
	\noindent
	{\bf (b):}
	In this case, 
	a vertex which can dominate $(p+1, q-1)$ is as follows: 
	if $q \geq 3$, 
	then $(p+1, q-1)$, $(p+1, q-2)$, $(p, q-1)$ or $(p+2, q-1)$, 
	and 
	if $q = 2$, 
	then $(p+1, q-1)$, $(p, q-1)$ or $(p+2, q-1)$. 
	First, 
	we consider the case in which $(p+1, q-2) \in D$ or $(p, q-1) \in D$ 
	(Fig.~\ref{fig:l4}(2-1),(2-2)). 
	By the conditions (Q1) and (Q6), 
	all the vertices except for $(p+1, q-1)$ which $(p+1, q-2)$ or $(p, q-1)$ can dominate are dominated by regular vertices in $D$. 
	Hence, 
	$D$ contains exactly one of $(p+1, q-2)$ and $(p, q-1)$ 
	because $D$ is minimum. 
	Let $v$ be the vertex contained in $D$ 
	and let us construct $D'$ by removing $v$ from $D$ and adding $(p+1, q)$ to $D$ instead. 
	Then, 
	$D'$ is still dominating. 
	By the conditions (Q1) and (Q6), 
	for any two vertices except for $v$ in $D$, 
	no simple path consisting of vertices in $D$ between the two vertices contains $v$, 
	which implies that 
	$D'$ is connected. 
	Thus, 
	$D'$ is an MCDS. 
	Second, 
	let us consider the case in which 
	$(p+1, q-2) \notin D$ and $(p, q-1) \notin D$
	(Fig.~\ref{fig:l4}(4)). 
	That is, 
	$(p+1, q-1) \in D$ or $(p+2, q-1) \in D$. 
	However, 	
	$(p+2, q-1) \in D$ 
	because $D$ is minimum. 
	By the condition of (b), 
	$P$ contains at least one of $(p, q)$ and $(p, q+1)$. 
	Also, 
	since $(p, q-1) \notin D$, 
	$P$ contains $(p, q+1)$ if $P$ contains $(p, q)$. 
	It implies that $(p, q+1) \in D$. 
	Let us consider the case in which $P$ contains $(p, q+2)$. 
	That is, 
	$(p, q+2) \in D$
	(Fig.~\ref{fig:l4}(6)). 
	Then, let us define 
	$D' = D \backslash  \{ (p, q+1) \} \cup  \{ (p+1, q) \}$
	(Fig.~\ref{fig:l4}(7)). 
	By the condition (Q1) and the fact that $(p, q+2) \in D$, 
	a vertex in $D$ which only $(p, q+1)$ can dominate is $(p+1, q+1)$.  
	$(p+1, q+1)$ dominates the vertex in $D'$. 
	Thus, 
	$D'$ is dominating. 
	We next show that $D'$ is connected. 
	Let us define $C = D \backslash \{ (p, q+1) \}$, 
	and define the vertex set $C_{1}$ of vertices each of which 
	there exists a path consisting of vertices in $C$ from $(p+2, q)$ to. 
 	Also, 
	we define $C_{2} = C \backslash C_{1}$. 
	We now show that $D'$ is connected by showing that 
	for any $v \in C_{1}$ and any $u \in C_{2}$, 
	there exists a path of vertices in $D'$ between $v$ and $u$. 
	If $(p, q) \in C_{1}$, 
	there exists a path of vertices in $C$ different from $P$ between $(p+2, q)$ and $(p, q)$ and all the vertices in $C$ belong to $C_{1}$. 
	Hence, 
	$D'$ is connected. 
	If $(p, q) \in C_{2}$, 
	there exists a path $P'$ of vertices in $C_{2}$ between $(p, q)$ and $u$.
	Then, 
	let us consider the path comprising the following three paths: 
	a path of vertices in $C_{1}$ between $v$ and $(p+2, q)$, 
	the path $(p+2, q)(p+1, q)(p, q)$ and $P'$. 
	This path consists of vertices in $D'$, 
	which implies that 
	$D'$ is connected. 
	Next we consider the case in which 
	$q \leq n-2$ and $P$ does not contain $(p, q+2)$. 
	We assume that $P$ does not contain $(p+1, q+1)$. 
	Since $P$ contains $(p, q+1)$, 
	$P$ also contains $(p-1, q+1)$ and 
	there exists some $x \leq p-1$ 
	such that $P$ contains $(x, q+1)$ and $(x, q+2)$, 
	which contradicts the definition that the number of vertices in ${\overline{R}}(D)$ in $P$ is minimum, 
	and thus $P$ contains $(p+1, q+1)$
	(Fig.~\ref{fig:l4}(8)). 
	In the case of $q = n-1$, 
	we assume again that $P$ does not contain $(p+1, q+1) (=(p+1, n))$. 
	Since $(p'+1, n), \cdots, (p-1, n) \notin D$ by the condition (Q9), 
	$P$ does not contain $(p, q+1) (=(p, n))$, 
	which contradicts the fact discussed above. 
	Hence, 
	if $P$ contains $(p, q+1)$ but does not contain $(p, q+2)$, 
	$P$ contains $(p+1, q+1)$. 
	Then, 
	we consider the case in which
	$(p+1, q-2) \notin D$, 
	$(p, q-1) \notin D$, 
	$P$ does not contain $(p, q+2)$ 
	and 
	$(p+2, q+1) \in D$
	(Fig.~\ref{fig:l4}(9)). 
	If $q + 2 \leq n$ and $p + 3 \leq m$, 
	then we define $D' = D \backslash \{ (p, q+1), (p+2, q+1), (p+2, q-1) \} \cup \{ (p+1, q), (p+1, q+2), (p+3, q) \}$
	(Fig.~\ref{fig:l4}(10-1)). 
	If $q + 2 \leq n$ and $p + 3 > m$ (i.e., $p + 2 = m$), 
	then we define 
	$D' = D \backslash \{ (p, q+1), (p+2, q+1) \} \cup \{ (p+1, q), (p+1, q+2) \}$
	(Fig.~\ref{fig:l4}(10-2)). 
	If $q + 2 > n$ (i.e., $q + 1 = n$), 
	then we define 
	$D' = D \backslash \{ (p, q+1) \} \cup \{ (p+1, q) \}$
	(Fig.~\ref{fig:l4}(10-3)). 
	For each $D'$ in these cases, 
	we can show that $D'$ is dominating and connected 
	similarly to the proof of the case in which $P$ contains $(p, q+2)$. 
	Moreover, 
	if $q = n-1$, 
	then $(p-1, n) \notin D'$ because $D'$ is minimum, 
	which satisfies the condition (Q9). 
	Thus, 
	$D'$ is $(p', q')$-$(p+1, q)$-regular. 
	Finally, 
	we consider the case in which 
	$(p+1, q-2) \notin D$, 
	$(p, q-1) \notin D$, 
	$P$ does not contain $(p, q+2)$ and 
	$(p+2, q+1) \notin D$
	(Fig.~\ref{fig:l4}(11)). 
	Then, 
	$P$ contains $(p+1, q+2)$, 
	that is, 
	$(p+1, q+2) \in D$. 
	Let us define $D' = D \backslash \{ (p, q+1) \} \cup \{ (p+1, q) \}$ 
	and we can show that $D'$ is dominating and connected 
	similarly to the proof of the above case. 
	Thus, 
	$D'$ is $(p', q')$-$(p+1, q)$-regular. 
	\noindent
	{\bf (c):}
	If either $(p+1, q-2) \in D$ or $(p, q-1) \in D$, 
	we can prove the statement in the same way as the proof of the case (b)
	(Fig.~\ref{fig:l4}(15)). 
	We consider the case in which $(p+1, q-2) \notin D$
	and 
	$(p, q-1) \notin D$. 
	Then, 
	$(p+2, q-1) \in D$ in the same way as the case (b)
	(Fig.~\ref{fig:l4}(16)). 
	We consider whether $(p+2, q-2) \in D$ or not
	(Fig.~\ref{fig:l4}(17)).
	If $p+2 = m$, 
	then $(p+3, q-1) \in D$ is impossible. 
	Thus, 
	note that $(p+2, q-2) \in D$ 
	because 
	$(p+2, q-2) \in D$ mush hold 
	so that $(p+2, q-1)$ is connected to any other vertex in $D$. 
	We consider the case of $(p+2, q-2) \in D$
	(Fig.~\ref{fig:l4}(18)). 
	If $p+3 \geq m$, 
	then we define $D' = D \backslash \{ (p+2, q-1), (p+2, q-2) \} \cup \{ (p+1, q), (p+3, q-2) \}$
	(Fig.~\ref{fig:l4}(19-1)). 
	If $p+2 = m$, 
	then we define $D' = D \backslash \{ (p+2, q-1) \} \cup \{ (p+1, q) \}$
	(Fig.~\ref{fig:l4}(19-2)). 
	For each $D'$ in the both cases, 
	we can show that $D'$ is dominating and connected 
	similarly to the proof of the above case. 
	Thus, 
	$D'$ is $(p', q')$-$(p+1, q)$-regular. 
	We consider the case of $(p+2, q-2) \notin D$
	(Fig.~\ref{fig:l4}(20)). 
	$(p+3, q-1) \in D$ holds 
	so that $(p+2, q-1)$ with respect to $D$ is connected to any other vertex. 
	We define $D' = D \backslash \{ (p+2, q-1) \} \cup \{ (p+1, q) \}$
	(Fig.~\ref{fig:l4}(21)). 
	We can show that $D'$ in each of these cases is dominating and connected 
	similarly to the proof of the above case. 
	Thus, 
	$D'$ is $(p', q')$-$(p+1, q)$-regular. 
	We omit the proof of (ii) 
	because it can be proved similarly to the proof of the case (i). 
	\fi
\end{proof}
\ifnum \count12 > 0
\begin{figure*}
	 \begin{center}
	  \includegraphics[width=\linewidth]{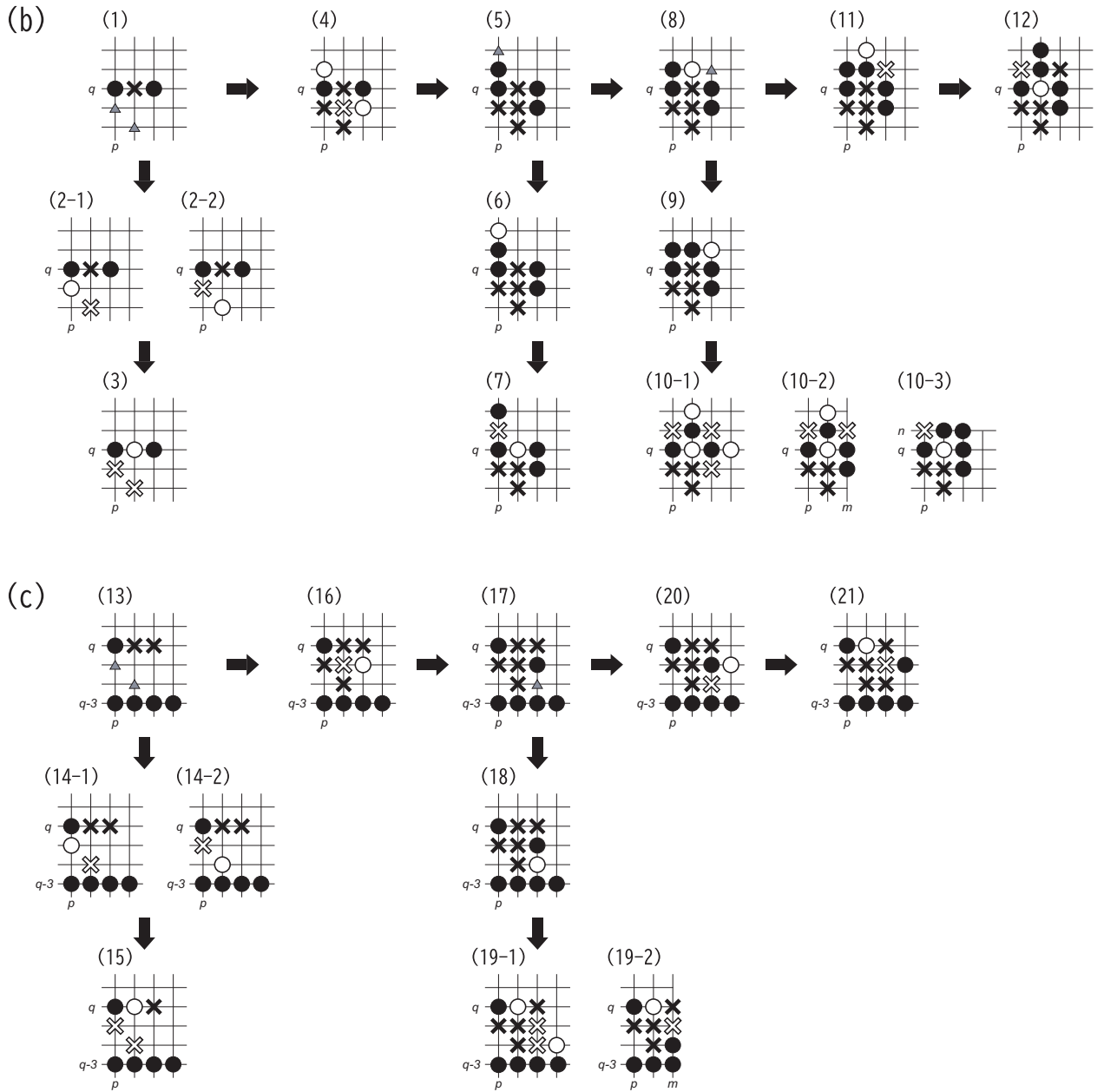}
	 \end{center}
	 \caption{
\ifnum \count10 > 0
\fi
\ifnum \count11 > 0
%
%
Figure for Lemma~\ref{LMA:L4}. 
Suppose that $D$ is a $(p', q')$-$(p, q)$-regular MCDS. 
Circles and crosses denote vertices in $D$ and not in $D$, respectively. 
White circles denote vertices in $D$ which are handled at the regularization.
White crosses denote vertices not in $D$ 
or vertices removed from $D$ at the regularization. 
\fi
			}
	\label{fig:l4}
\end{figure*}
\fi
\ifnum \count10 > 0
%
%

%
\fi
\ifnum \count11 > 0
%
%

%
\fi
%

%
\begin{LMA}\label{LMA:L6}
	\ifnum \count10 > 0
	%
	%
	\fi
	\ifnum \count11 > 0
	%
	%
	Suppose that a vertex set $D$ is a $(0, 2)$-$(m, 2)$-regular MCDS. 
	Then, 
	there exists an MCDS which $D$ is $(2, 2)$-$(2, 3)$-regularized into
	(Fig.~\ref{fig:l6}(2)). 
	\fi
\end{LMA}
\begin{proof}
	\ifnum \count10 > 0
	\fi
	\ifnum \count11 > 0
	%
	%
	Suppose that a vertex set $D$ is a $(0, 2)$-$(m, 2)$-regular MCDS. 
	We will construct a vertex set $\hat{D}$ by regularizing $D$. 
	Note that $\hat{D}$ is also an MCDS 
	because $D$ is an MCDS. 
	$(1, 2), \ldots, (m, 2) \in D$ by the condition (Q1) for $D$ to be $(0, 2)$-$(m, 2)$-regular, 
	which are regular vertices and are not removed from $D$ to construct $\hat{D}$. 
	Hence, 
	$(1, 2), \ldots, (m, 2) \in \hat{D}$
	and 
	$\hat{D}$ satisfies the condition (P6) to be $(2, 2)$-$(2, 3)$-regular. 
	Similarly, 
	$(1, 1), \cdots, (m-1, 1) \notin D$ and $(m, 1) \notin D$ 
	by the conditions (Q2) and (Q3) of $D$. 
	Hence, 
	$(1, 1), \ldots, (m, 1) \notin \hat{D}$
	and 
	$\hat{D}$ satisfies the condition (P2). 
	That is, 
	$\hat{D}$ into which $D$ is regularized into satisfies all the conditions except for (P1) to be $(2, 2)$-$(2, 3)$-regular. 
	Now we will show that we can obtain $\hat{D}$ such that $(2, 3) \in \hat{D}$ to satisfy the condition (P1) by the regularization. 
	Also, 
	we will show that $\hat{D}$ is dominating and connected, 
	which leads to the statement is true. 
	If $(2, 3) \in D$ (Fig.~\ref{fig:l6}(2)), 
	then $D$ is $(2, 2)$-$(2, 3)$-regular and 
	$D$ can be regarded as the MCDS $\hat{D}$. 
	In what follows, 
	we consider the case in which $(2, 3) \notin D$
	(Fig.~\ref{fig:l6}(3)). 
	If $(1, 3) \in D$ (Fig.~\ref{fig:l6}(4)), 
	then we can have a $(2, 2)$-$(2, 3)$-regular MCDS from $D$ using Lemma~\ref{LMA:LT}(iii). 
	We next consider the case in which 
	$(1, 3) \notin D$ 
	and 
	$(2, 4) \in D$ 
	(Fig.~\ref{fig:l6}(6)). 
	Since $(2, 2) \in R(D)$ and $(2, 4) \in {\overline{R}}(D)$, 
	there exists an MCDS $D' = D \backslash \{ v' \} \cup \{ (2, 3) \}$, 
	in which $v' \in {\overline{R}}(D)$ is a mobile,  
	by applying Lemma~\ref{LMA:MV}(ii) 
	with $(2, 2)$ and $(2, 4)$ as $(x, y)$ and $(x, y+2)$, respectively. 
	$(2, 3) \in D'$
	and thus, 
	$D'$ is $(2, 2)$-$(2, 3)$-regular. 
	We next consider the case in which 
	$(1, 3) \notin D$ 
	and 
	$(2, 4) \notin D$ 
	(Fig.~\ref{fig:l6}(7)). 
	$D$ contains either $(1, 4)$ or $(1, 5)$ to dominate $(1, 4)$, and 
	$(1, 5) \in D$ holds because $D$ is connected. 
	Then, 
	since $D$ is $(0, 5)$-$(1, 5)$-regular, 
	$D'$ which is $(0, 5)$-$(m, 5)$-regular can be constructed from $D$ 
	by applying Lemma~\ref{LMA:L4} $m-1$ times
	(Fig.~\ref{fig:l6}(8)). 
	Then, 
	$(1, 4), \ldots, (m, 4) \notin D$ by the conditions (Q2) and (Q3), which contradicts that the constructed CDS is connected. 
	Therefore, 
	both 
	$(1, 3) \notin D$ 
	and 
	$(2, 4) \notin D$ 
	do not hold. 
	We have completed the proof. 
	\fi
\end{proof}
\ifnum \count12 > 0
\begin{figure*}
	 \begin{center}
	  \includegraphics[width=\linewidth]{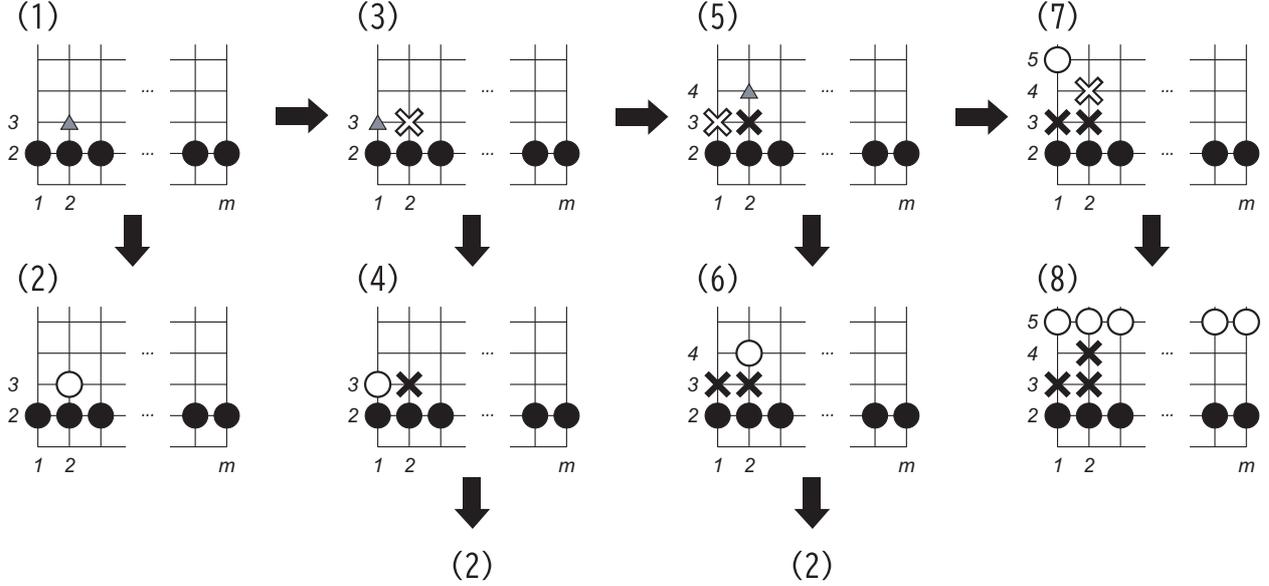}
	 \end{center}
	 \caption{
\ifnum \count10 > 0
\fi
\ifnum \count11 > 0
%
%
Figure for Lemma~\ref{LMA:L6}. 
Suppose that $D$ is a $(0, 2)$-$(m, 2)$-regular MCDS. 
Circles and crosses denote vertices in $D$ and not in $D$, respectively. 
\fi
			}
	\label{fig:l6}
\end{figure*}
\fi
\ifnum \count10 > 0
%
%
%
\fi
\ifnum \count11 > 0
%
%

%
\fi
%

%
\begin{LMA}\label{LMA:L9}
	\ifnum \count10 > 0
	%
	%
	\fi
	\ifnum \count11 > 0
	%
	%
	For $p \in [2, m-3]$ and $q \in [2, n-3]$, 
	suppose that a vertex set $D$ is an MCDS which is $(p, q)$-$(m, q)$-regular or $(p, q)$-$(p, n)$-regular. 
	Then, 
	there exists an MCDS $D'$ into which $D$ is $(p, q)$-$(\hat{p}, \hat{q})$-regularized 
	such that 
	either $(\hat{p}, \hat{q}) = (p+3, q+1)$ 
	or 
	$(\hat{p}, \hat{q}) = (p+1, q+3)$. 
	\fi
\end{LMA}
\begin{proof}
	\ifnum \count10 > 0
	%
	%
	\fi
	\ifnum \count11 > 0
	%
	%
	The proof of this lemma is similar to that of Lemma~\ref{LMA:L6}. 
	First, for $p \in [2, m-3]$ and $q \in [2, n-3]$, 
	suppose that a vertex set $D$ is an $(p, q)$-$(m, q)$-regular MCDS. 
	We will construct a vertex set $\hat{D}$ by regularizing $D$. 
	Note that $\hat{D}$ is also an MCDS 
	because $D$ is an MCDS. 
	By the conditions (Q1) and (Q6) for $D$ to be $(p, q)$-$(m, q)$-regular, 
	$(p+1, q), \ldots, (m, q) \in D$
	and 
	$(p, q), \ldots, (p, n) \in D$, respectively. 
	Since these vertices are regular, 
	they are not removed from $D$ when $D$ is regularized into $\hat{D}$. 
	Thus, 
	$(p+1, q), \ldots, (m, q) \in \hat{D}$
	and 
	$(p, q), \ldots, (p, n) \in \hat{D}$, 
	which implies that 
	$\hat{D}$ satisfies the conditions (Q5), (P5), (Q6) and (P6) to be $(p, q)$-$(\hat{p}, \hat{q})$-regular 
	such that 
	$(\hat{p}, \hat{q}) = (p+3, q+1)$ 
	or 
	$(\hat{p}, \hat{q}) = (p+1, q+3)$. 
	Similarly, 
	by the conditions (Q7) and (Q2), 
	$(p-1, q+1), \ldots, (p-1, n) \notin \hat{D}$
	and 
	$(p+1, q-1), \ldots, (m, q-1) \notin \hat{D}$, respectively. 
	Thus, 
	$\hat{D}$ satisfies the conditions (Q7), (P7), (Q8) and (P8) to be $(p, q)$-$(\hat{p}, \hat{q})$-regular. 
	Since $D$ is minimum, 
	$(p+1, q+1) \notin D$, that is, $(p+1, q+1) \notin \hat{D}$. 
	Hence, 	
	$\hat{D}$ satisfies the conditions (Q4) and (P4) to be $(p, q)$-$(\hat{p}, \hat{q})$-regular. 
	By the above argument,	
	$\hat{D}$ into which $D$ is regularized into satisfies all the conditions except for (Q1) and (P1).
	Now we will show that 
	we can obtain $\hat{D}$ such that $(p+3, q+1) \in \hat{D}$ ($(p+1, q+3) \in \hat{D}$, respectively) to satisfy the condition (Q1) ((P1), respectively) by regularizing $D$. 
	If $(p+3, q+1) \in D$ or $(p+1, q+3) \in D$, 
	then $D$ can be regarded as the MCDS $\hat{D}$
	(see Fig.~\ref{fig:l9}(2-1),(2-2)). 
	In what follows, 
	we consider the case in which 
	$(p+3, q+1) \notin D$ 
	and 
	$(p+1, q+3) \notin D$ 
	(Fig.~\ref{fig:l9}(3-1),(3-2)). 
	If $(p+2, q+1) \in D$, 
	then $D$ can be $(p, q)$-$(p+3, q+1)$-regularized into an MCDS $D'$ 
	by applying Lemma~\ref{LMA:LT}(i) with $p$ and $q$ as $p$ and $q$ in its statement, respectively
	(Fig.~\ref{fig:l9}(4-1)). 
	Also, 
	if $(p+1, q+2) \in D$, 
	$D$ can be $(p, q)$-$(p+1, q+3)$-regularized into an MCDS $D'$ by applying Lemma~\ref{LMA:LT}(ii) with $p$ and $q$ as $p$ and $q$, respectively
	(Fig.~\ref{fig:l9}(4-2)). 
	Hence, 
	$D'$ can be regarded as $\hat{D}$. 
	Second, 
	we discuss the case in which 
	$(p+2, q+1) \notin D$ 
	and 
	$(p+1, q+2) \notin D$. 
	Since $D$ contains vertices to dominate $(p+2, q+2)$ and is connected, 
	either $(p+3, q+2) \in D$ or $(p+2, q+3) \in D$
	(Fig.~\ref{fig:l9}(5-1),(5-2)). 
	If $(p+3, q+2) \in D$, 
	then 
	there exists an MCDS $D' = D' \backslash \{ v' \} \cup \{ (p+3, q+1) \}$, 
	in which $v' \in {\overline{R}}(D)$ is a mobile, 
	by applying Lemma~\ref{LMA:MV} (ii) with $(p+3, q) \in R(D)$ and $(p+3, q+2) \in {\overline{R}}(D)$ as $(x, y)$ and $(x, y+2)$, respectively.  
	Since $(p+3, q+1) \in D'$, $D'$ is $(p+3, q)$-$(p+3, q+1)$-regular, 
	which implies that the statement is true. 
	Similarly, 
	if $(p+2, q+3) \in D$, 
	there exists an MCDS $D' = D' \backslash \{ v' \} \cup \{ (p+1, q+3) \}$, 
	in which $v' \in {\overline{R}}(D)$ is a mobile, 
	by applying Lemma~\ref{LMA:MV} (i) with $(p, q+3) \in R(D)$ and $(p+2, q+3) \in {\overline{R}}(D)$ as $(x, y)$ and $(x+2, y)$, respectively.  
	Since $(p+1, q+3) \in D'$, $D'$ is $(p, q+3)$-$(p+1, q+3)$-regular, 
	which implies that the statement is true. 
	We omit the proof of the case in which $D$ is $(p, q)$-$(p, n)$-regular 
	because this case is symmetric to the case discussed above. 
	\fi
\end{proof}
\ifnum \count12 > 0
\begin{figure*}
	 \begin{center}
	  \includegraphics[width=\linewidth]{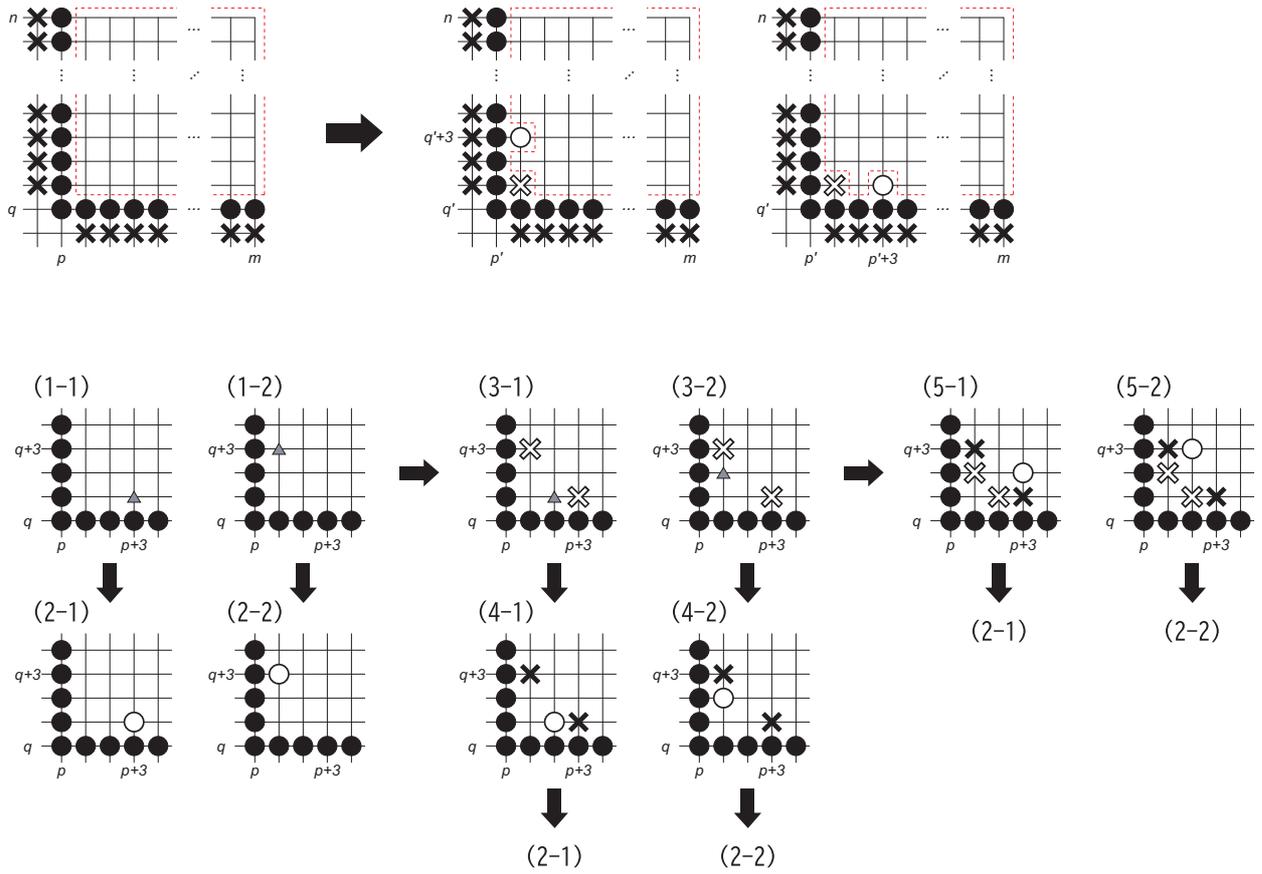}
	 \end{center}
	 \caption{
\ifnum \count10 > 0
%
%
\fi
\ifnum \count11 > 0
%
%
The top figure illustrates the situation of the statement of Lemma~\ref{LMA:L9}.
Suppose that $D$ is an MCDS which is $(p, q)$-$(m, q)$-regular or $(p, q)$-$(p, n)$-regular. 
Circles and crosses denote vertices in $D$ and not in $D$, respectively. 
White circles denote vertices in $D$ which are handled at the regularization.
White crosses denote vertices not in $D$ 
or vertices removed from $D$ at the regularization. 
\fi
			}
	\label{fig:l9}
\end{figure*}
\fi
\ifnum \count10 > 0
%
%
%
\fi
\ifnum \count11 > 0
%
%

%
\fi
%

%
\begin{LMA}\label{LMA:C9}
	\ifnum \count10 > 0
	\fi
	\ifnum \count11 > 0
	%
	%
	For $p \geq 2$ and $q \geq 2$, 
	suppose that a vertex set $D$ is an MCDS which is $(p, q)$-$(m, q)$-regular or $(p, q)$-$(p, n)$-regular. 
	Then, 
	the following properties hold: 
	\begin{itemize}
		\itemsep=-2.0pt
		\setlength{\leftskip}{0pt}
		\item[(i)]
			if $p \leq m-4$
			and 
			$q \leq m-4$, 
			then there exists an MCDS $D'$ 
			into which $D$ is $(p+3, q)$-$(p+3, q+1)$-regularized or $(p, q+3)$-$(p+1, q+3)$-regularized, 
		\item[(ii)]
			if $p \leq m-4$ and $q \in \{ n-3, n-2 \}$, 
			then there exists an MCDS $D'$ into which $D$ is $(p+3, n-2)$-$(p+3, n-1)$-regularized, 
			and 
		\item[(iii)]
			if $q \leq n-4$ 
			and 
			$p \in \{ m-3, m-2 \}$, 
			then there exists an MCDS $D'$ into which $D$ is $(m-2, q+3)$-$(m-1, q+3)$-regularized. 
	\end{itemize}
	\fi
\end{LMA}
\begin{proof}
	\ifnum \count10 > 0
	\fi
	\ifnum \count11 > 0
	%
	%
	For $p \geq 2$ and $q \geq 2$, 
	suppose that a vertex set $D$ is an MCDS which is $(p, q)$-$(m, q)$-regular or $(p, q)$-$(p, n)$-regular. 
	\noindent
	{\bf (i):}
	If $p \leq m-4$
	and 
	$q \leq n-4$, 
	then there exists an MCDS $D'$ into which 
	$D$ is $(p+3, q)$-$(p+3, q+1)$-regularized or $(p, q+3)$-$(p+1, q+3)$-regularized 
	by Lemma~\ref{LMA:L9}, 
	which satisfies the statement of this lemma. 
	\noindent
	{\bf (ii):}
	If $q = n-3$, 
	then there exists an MCDS $D'$ into which 
	$D$ is either $(p+3, q)$-$(p+3, q+1)$-regularized, 
	that is, 
	$(p+3, n-3)$-$(p+3, n-2)$-regularized or 
	$(p, q+3)$-$(p+1, q+3)$-regularized, 
	that is, 
	$(p, n)$-$(p+1, n)$-regularized. 
	Hence, 
	if $D'$ is $(p+3, n-3)$-$(p+3, n-2)$-regular, then the statement is true 
	(Fig.~\ref{fig:c9}(2)). 
	Then, 
	we discuss the case in which $D'$ is $(p, n)$-$(p+1, n)$-regular
	(Fig.~\ref{fig:c9}(3)), 
	that is, 
	the case in which $(p+1, n) \in D'$. 
	In this case, 
	$D'$ satisfies all the conditions except for (Q1) similarly to the proof of Lemma~\ref{LMA:L9}. 
	Thus, 
	we prove that there exists an MCDS into which $D'$ is regularized and satisfies (P1), that is, 
	which contains $(p+3, n-2)$. 
	We can obtain a $(p, n)$-$(p+3, n)$-regular MCDS $D''$ from ~\ref{LMA:L4} by applying Lemma~\ref{LMA:L4} to ~\ref{LMA:L4} twice 	
	(Fig.~\ref{fig:c9}(4)). 
	Since $(p+1, n), (p+2, n), (p+3, n) \in D''$, 
	we define 
	$D''' = D'' \backslash \{ (p+1, n), (p+2, n) \} \cup \{ (p+3, n-2), (p+3, n-1) \}$. 
	Since $D'''$ is still dominating and connected by Fig.~\ref{fig:c9}(5), 
	$D'''$ is an MCDS. 
	Moreover, 
	since $(p+1, n), (p+2, n), (p+3, n-2)$ and $(p+3, n-1)$ are irregular vertices in $D'$ and $D$,  
	$D$ can be $(p+3, n-3)$-$(p+3, n-2)$-regularized into $D'''$. 
	We consider the case in which $q = n-2$. 
	If $(p+3, n-1) \in D$, 
	the statement is clearly true
	(Fig.~\ref{fig:c9}(7)). 
	If $(p+3, n-1) \notin D$
	and 
	$(p+2, n-1) \in D$
	(Fig.~\ref{fig:c9}(9)),  
	then the statement is true 
	because $D$ can be $(p+3, n-2)$-$(p+3, n-1)$-regularized into $D'$ by Lemma~\ref{LMA:LT}(i). 
	If $(p+3, n-1) \notin D$, 
	$(p+2, n-1) \notin D$
	and 
	$(p+1, n) \in D$ 
	(Fig.~\ref{fig:c9}(11)),  
	we define $D' = D \backslash \{ (p+1, n) \} \cup \{ (p+2, n-1) \}$. 
	$D'$ is dominating and connected by the figure. 
	Thus, 
	$(p+3, n-1) \notin D'$
	and 
	$(p+2, n-1) \in D'$, 
	which implies that the statement is true 
	because this situation is the same as the previously discussed case. 
	Finally, 
	we consider the case in which 
	$(p+3, n-1) \notin D$, 
	$(p+2, n-1) \notin D$
	and 
	$(p+1, n) \notin D$. 
	Since $D$ contains vertices to dominate $(p+2, n)$ and is connected, 
	$(p+3, n) \in D$
	(Fig.~\ref{fig:c9}(12)). 
	There exists an MCDS $D' = D \backslash \{ v' \} \cup \{ (p+3, n-1) \}$, 
	in which $v' \in {\overline{R}}(D)$ is a mobile in $D$, 
	by applying Lemma~\ref{LMA:MV}~(ii) with $(p+3, n-2) \in R(D)$ and $(p+3, n) \in {\overline{R}}(D)$ as $(x, y)$ and $(x+2, y)$, respectively. 
	$D'$ is $(p+3, n-2)$-$(p+3, n-1)$-regular 
	because $(p+3, n-1) \in D'$. 
	\noindent
	{\bf (iii):}
	We omit the proof of this case 
	because we can prove it similarly to that of the case (ii). 
	\fi
\end{proof}
\ifnum \count12 > 0
\begin{figure*}
	 \begin{center}
	  \includegraphics[width=\linewidth]{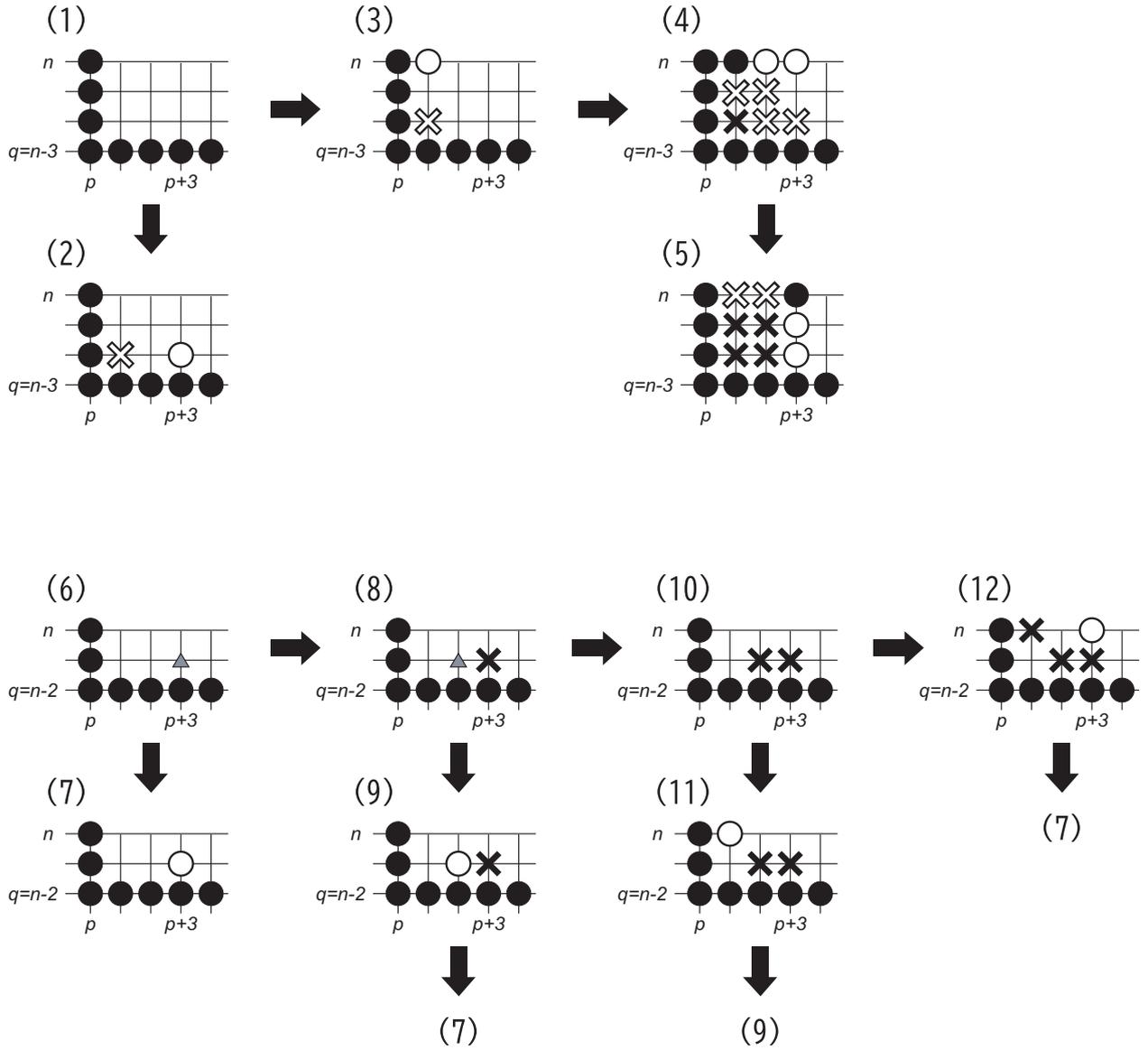}
	 \end{center}
	 \caption{
\ifnum \count10 > 0
%
%
\fi
\ifnum \count11 > 0
%
%
Lemma~\ref{LMA:C9}. 
Suppose that $D$ is an MCDS which is $(p, q)$-$(m, q)$-regular or $(p, q)$-$(p, n)$-regular. 
Circles and crosses denote vertices in $D$ and not in $D$, respectively. 
White circles denote vertices in $D$ which are handled at the regularization.
White crosses denote vertices not in $D$ 
or vertices removed from $D$ at the regularization.
\fi
			}
	\label{fig:c9}
\end{figure*}
\fi
%


%
%
\ifnum \count10 > 0
%
%

%
\fi
\ifnum \count11 > 0
%
%

%
\fi
%

%


\bibliography{cdsg}
\bibliographystyle{plain}


\appendix

%
\ifnum \count10 > 0
%
%

%
\fi
\ifnum \count11 > 0
%
%

%
\fi
%

%

%
\ifnum \count10 > 0
%
%

%

%
\fi
\ifnum \count11 > 0
%
%

%
\fi

\end{document}